\documentclass[a4paper, 12pt]{article}

\usepackage{amsmath} 
\usepackage{theorem}
\usepackage{amssymb}
\usepackage{amsfonts}
\usepackage{latexsym}
\usepackage{amscd}
\usepackage{diagramb}
\input grcalc3.sty
\usepackage{graphicx}

\theoremstyle{plain}
\theoremheaderfont{\normalfont\bfseries}

\newtheorem{thm}{Theorem}[section]

\newtheorem{lma}[thm]{Lemma}
\newtheorem{cor}[thm]{Corollary}

\newtheorem{defn}[thm]{Definition}

\theorembodyfont{\rmfamily}

\newtheorem{rem}[thm]{Remark}
\newtheorem{prop}[thm]{Proposition}

\newcommand{\qed}{\hfill\quad\fbox{\rule[0mm]{0,0cm}{0,0mm}}  \par\bigskip}

\newcommand{\R}{{\mathcal R}}

\newcommand{\w}{\hspace{-0,12cm}}
\newcommand{\q}{\hspace{-0,03cm}}

\newcommand{\iso}{\cong}
\newcommand{\ot}{\otimes}

\newcommand{\C}{{\mathcal C}}
\newcommand{\M}{{\mathcal M}}
\newcommand{\D}{{\mathcal D}}
\newcommand{\F}{{\mathcal F}}
\newcommand{\G}{{\mathcal G}}

\newcommand{\E}{{\mathcal E}}

\newcommand{\U}{{\mathcal U}}
\newcommand{\YD}{{\mathcal YD}}

\def\Zz{{\mathbb Z}}

\newcommand{\crta}{\overline}
\newcommand{\ev}{\it ev}

\newcommand{\Fi}{\varphi}

\newcommand{\teta}{\theta}

\newcommand{\BM}{\operatorname {BM}}
\newcommand{\Br}{\operatorname {Br}}
\newcommand{\Epsilon}{\varepsilon}

\newcommand{\Bi}{\operatorname {Im}}
\newcommand{\Ker}{\operatorname {Ker}}

\newcommand{\End}{\operatorname {End}}
\newcommand{\Aut}{\operatorname {Aut}}
\newcommand{\Alg}{\operatorname {Alg}}
\newcommand{\Coalg}{\operatorname {Co-Alg}}

\newcommand{\BD}{\operatorname {BD}}
\newcommand{\BQ}{\operatorname {BQ}}

\newcommand{\pH}{\hspace{-0,06cm}\cdot_{\hspace{-0,1cm}_H}\hspace{-0,1cm}}
\newcommand{\pB}{\hspace{-0,06cm}\cdot_{\hspace{-0,1cm}_B}\hspace{-0,1cm}}

\newcommand{\tr}{\triangleright} 
\newcommand{\tl}{\triangleleft} 

\newcommand{\rt}{\rtimes}

\newcommand{\Gal}{\operatorname{Gal}}

\newcommand{\BW}{\operatorname {BW}}

\newcommand{\cref}[1]{C.~\ref{c:#1}}

\newcommand{\lelabel}[1]{\label{le:#1}}
\newcommand{\leref}[1]{Lemma~\ref{le:#1}}
\newcommand{\eqlabel}[1]{\label{eq:#1}}
\newcommand{\equref}[1]{(\ref{eq:#1})}
\newcommand{\thlabel}[1]{\label{th:#1}}
\newcommand{\thref}[1]{Theorem~\ref{th:#1}}

\newcommand{\prlabel}[1]{\label{pr:#1}}
\newcommand{\prref}[1]{Proposition~\ref{pr:#1}}
\newcommand{\colabel}[1]{\label{co:#1}}
\newcommand{\coref}[1]{Corollary~\ref{co:#1}}

\newcommand{\rmlabel}[1]{\label{rm:#1}}
\newcommand{\rmref}[1]{Remark~\ref{rm:#1}}

\newcommand{\sslabel}[1]{\label{ss:#1}}
\newcommand{\ssref}[1]{Subsection~\ref{ss:#1}}

\begin{document}

\title{{\bf The Hopf automorphism group \\
and the quantum Brauer group \\ 
in braided monoidal categories}}
\author{{\large B. Femi\'c \vspace{2pt}} \\
{\small Facultad de Ingenier\'ia, \vspace{-2pt}}\\
{\small  Universidad de la Rep\'ublica} \vspace{-2pt}\\
{\small  Julio Herrera y Reissig 565, \vspace{-2pt}}\\ 
{\normalsize 11 300 Montevideo, Uruguay}\\
{\normalsize e-mail: bfemic@fing.edu.uy}}


\date{May 17, 2011}
\maketitle


\vspace{-0.0cm}

\begin{abstract}
With the motivation of giving a more precise estimation of the quantum Brauer group of a Hopf
algebra $H$ over a field $k$ we construct an exact sequence containing the quantum Brauer group
of a Hopf algebra in a certain braided monoidal category. Let $B$ be a Hopf algebra in $\C=_H ^H\YD$,
the category of Yetter-Drinfel'd modules over $H$. We consider the quantum Brauer group $\BQ(\C; B)$
of $B$ in $\C$, which is isomorphic to the usual quantum Brauer group $\BQ(k; B\rtimes H)$ of the
Radford biproduct Hopf algebra $B\rtimes H$. We show that under certain symmetricity condition on
the braiding in $\C$ there is an inner action of the Hopf automorphism group of $B$ on the former.
We prove that the subgroup $\BM(\C; B)$ - the Brauer group of module algebras over $B$ in $\C$ -
is invariant under this action for a family of Radford biproduct Hopf algebras. 
The analogous invariance we study for $\BM(k; B\rtimes H)$. We apply our recent results on the
latter group and generate a new subgroup of the quantum Brauer group of $B\rtimes H$. In particular,
we get new information on the quantum Brauer groups of some known Hopf algebras.
\end{abstract}


\indent \hspace{0.3cm} {\bf Keywords:} Brauer group, Azumaya algebras, Hopf algebras, Drinfel'd double, Braided monoidal categories.

\indent \hspace{0.3cm} {\bf Mathematics Subject Classification (2010):} 18D10, 18D35, 16T05.

\setcounter{tocdepth}{1}

\renewcommand{\theequation}{\thesection.\arabic{equation}}

\section{Introduction}

In the line of development of the theory of the Brauer group, which in its origins rested on
algebras over a field without any further structure, a particular interest in the last dozen of
years occupies the quantum Brauer group.
It is the Brauer group of Yetter-Dinfel'd module algebras introduced in \cite{CVZ1}. From the
point of view of category theory the mentioned development passed from the studies performed
in the category of vector spaces -- the original Brauer group of a field introduced in 1929;
via the category of modules over a commutative ring $R$; that of vector spaces (or modules over $R$)
which are graded by an abelian group; the category of dimodules over a commutative and cocommutative
Hopf algebra $H$, called the Brauer-Long group and denoted by $\BD(R; H)$; the category of Yetter-Dinfel'd
modules over a Hopf algebra $H$ with a bijective antipode -- the quantum Brauer group $\BQ(R; H)$ --
just to mention some of them.
The categorical background of the Brauer group was accomplished in \cite{P4} and \cite{VZ1}. The latter
work revealed that almost all the studied Brauer groups are Brauer groups of a certain braided
monoidal category. When a Hopf algebra $H$ is commutative and cocommutative, then a Yetter-Dinfel'd
module over $H$ is precisely an $H$-dimodule and the respective categories coincide. In this sense
the quantum Brauer group generalizes the Brauer-Long group.

Not much progress has been achieved on the computation of the quantum Brauer group so far. This
seems to be a difficult task. Though, the Brauer-Long group and its subgroups have extensively been
studied 
\cite{B1, Deg, Cae3, Cae4, U3, CVZ2}.
Furthermore, in \cite{CF1} we gave a method to compute $\BM(R; H)$ -- the Brauer group of module
algebras over a Hopf algebra $H$ - where $H$ is a quasitriangular Radford biproduct of a certain kind.
It unifies several computations of this group performed in the last years for certain cases of $H$. A
first approximation of the quantum Brauer group itself was attained in \cite{VZ4}. There the following
exact sequence is constructed:
\vspace{-0,1cm}
\begin{equation} \eqlabel{seq FZ} \hspace{-0,6cm}
\bfig
\putmorphism(280, 0)(1, 0)[1`G(D(H)^*)`]{420}1a
\putmorphism(700, 0)(1, 0)[\phantom{G(D(H)^*)}`G(D(H))`]{610}1a
\putmorphism(1310, 0)(1, 0)[\phantom{G(D(H))}`\Aut(H)`]{600}1a
\putmorphism(1910, 0)(1, 0)[\phantom{\Aut(H)}`\BQ(R;H).`\Upsilon]{600}1a
\efig
\end{equation}
Here $G(D(H)^*)$ and $G(D(H))$ are the groups of group-likes of the dual of the Drinfel'd double $D(H)$ of $H$
and of $D(H)$, respectively, and $\Aut(H)$ is the Hopf automorphism group of $H$. The Hopf algebra $H$
is faithfully projective and has a bijective antipode.
When $D(H)$ is commutative, equivalently, when $H$ is commutative and cocommutative, the map $\Upsilon$ is
injective. Consequently, the Hopf automorphism group can be embedded into the Brauer-Long group. This
recovers Deegan-Caenepeel's embedding result \cite[Theorem 4.1]{Cae4} (originally, Deegan's embedding
theorem claims the corresponding group embedding for a finite abelian group in place of $H$).
As pointed out in \cite{VZ4}, the peculiarity of the
quantum Brauer group is that it may be an infinite torsion group although the Hopf algebra is small and finite-dimensional. This is owed to the same feature of the automorphism group of a non-commutative
non-cocommutative Hopf algebra (whereas if $H$ is a finite-dimensional commutative and cocommutative
Hopf algebra, then $\Aut(H)$ is a finite group).
\par\medskip

The idea in the present paper is to construct an analogous sequence to \equref{seq FZ}
for a Hopf algebra in a suitable braided monoidal category. Let $H$ be a Hopf algebra with a bijective
antipode over a field $k$ and let $\C=_H^H\YD$ be the (braided monoidal)
category of Yetter-Drinfel'd modules over $H$. Consider a Hopf algebra $B$ in $\C$, that is, a braided
Hopf algebra. On the one hand, we have that the Radford biproduct $B\rtimes H$ is a Hopf algebra over $k$
and the sequence \equref{seq FZ} can be applied to it. We will refer to this kind of application of
\equref{seq FZ} as to the 'base level'. On the other hand, our idea is to, so to say, lift the sequence
to one level higher, namely, to substitute the category $Vec$ of vector spaces over $k$ by $\C$, and the Hopf
algebra $H\in Vec$ by the Hopf algebra $B\in\C$. Let ${}_B ^B\YD(\C)$ denote the category of Yetter-Drinfel'd
modules over $B$ in $\C$. In \thref{seqBQ}, the main theorem of the paper, we prove that if $B$ is
finite-dimensional and such that the braiding $\Phi$ in $\C$ fulfills $\Phi_{B,M}=\Phi_{M,B}^{-1}$ for all
$M\in\C$ (we say that $\Phi_{B,M}$ is symmetric), then there is an exact sequence of groups:
\begin{equation} \eqlabel{seqBQintro}
\bfig
\putmorphism(-200, 0)(1, 0)[1`G_{\C}(D(B)^*)`]{500}1a
\putmorphism(300, 0)(1, 0)[\phantom{G_{\C}(D(B)^*)}`G_{\C}(D(B))` ]{700}1a
\putmorphism(1000, 0)(1, 0)[\phantom{G(_{\C}D(B))}`\Aut(\C;B)`\Theta]{700}1a
\putmorphism(1700, 0)(1, 0)[\phantom{\Gal(\C;H)}`\BQ(\C; B).`\Pi]{700}1a
\efig
\end{equation}
The groups involved in the sequence are categoric analogues of their counterparts from the sequence
\equref{seq FZ}. 
When $B$ is commutative and cocommutative in $\C$ we obtain Deegan-Caenepeel's kind of result (\coref{main cor BQ}).
We construct our sequence performing most of the computations in braided diagrams in an arbitrary braided
monoidal category $\E$. Thus proving \thref{seqBQ} we also obtain several more general results
(the whole section 2 and section 3 until the map $\Pi$ including \prref{inner action}). These contain
some results on group-likes and the Drinfel'd double of a Hopf algebra in $\E$, and some results on
inner hom-objects in $_B ^B\YD(\E)$. Observe that direct computations in the category $_B ^B\YD(\C)$ would give
rather complicated expressions 
too involved to handle. This illustrates how powerful tool the theory of
braided monoidal categories is.

We apply the sequence \equref{seqBQintro} to the family of Radford biproduct Hopf algebras $B\rtimes H$ that we studied in \cite{CF1}.
Here $H$ is quasitriangular and $B$ is a module over $H$. Then $B$ can be made into a left $H$-comodule
so that it becomes a Yetter-Drinfel'd module. In this particular case of a Radford biproduct the Hopf
algebra $B\rtimes H$ is called {\em bosonization} in \cite{Maj3}. We prove that in the named 
family of Hopf algebras $\Phi_{B,M}$ is symmetric for every $M\in\C$.
Moreover, $G_{\C}(D(B))$ turns out to be
trivial and $\Aut(\C; B)\iso k^*$ (the group of units of $k$). Thus we obtain that $k^*$ embeds into
$\BD(\C; B)$, the Brauer-Long group of $B$ in $\C$. On the other hand, by \cite{Besp} it is
${}_B ^B\YD(\C) \iso {}_{B\rtimes H} ^{B\rtimes H}\YD$ as braided monoidal categories. This has for a
consequence the group isomorphism $\BQ(\C; B)\iso\BQ(k; B\rtimes H)$. Since $B$ is commutative and
cocommutative in $\C$, we finally get $k^*\hookrightarrow\BQ(k; B\rtimes H)$. That the group of units of
a field (modulo $\Zz_2$) embeds into a quantum Brauer group was shown for Sweedler's Hopf algebra
in \cite{VZ4}. Our result generalizes this fact to the quantum Brauer group of every Radford biproduct
Hopf algebra from our family. 

In \cite{CVZ2} it was proved that there is an inner action of the Hopf automorphism group of $H$ on the
quantum Brauer group of $H$. We lift this result categorically to the Hopf algebra $B$ in $\C$, with the
assumption that $\Phi_{B,B}$ and $\Phi_{B,\U(A)}$ are symmetric for every Azumaya algebra $A$ in ${}_B ^B\YD(\C)$,
where $\U: {}_B ^B\YD(\C) \to\C$ is the forgetful functor,
(\coref{inner action}). Observe that these symmetricity conditions on the braiding are fulfilled in our family
of Radford biproducts. Since $B$ in this family is cocommutative, it is quasitriangular
with the quasitriangular structure $\R_B=1_B\ot 1_B$. 
Then every $B$-module is a trivial $B$-comodule and a Yetter-Drinfel'd module in $\C$.
Therefore we may consider the subgroup $\iota: \BM(\C; B)\hookrightarrow\BQ(\C; B)$, the former being the
Brauer group of module algebras over $B$ in $\C$. We show that $\BM(\C; B)$ is stable under the action
of $\Aut(\C; B)$. This leads us to conclude that $(\iota\ot\Pi)(\BM(\C; B) \rtimes \Aut(\C; B)) \iso
\BM(\C; B)\rtimes k^*$ is a subgroup of $\BQ(k; B\rtimes H)$, \prref{stable subgr}.
Using the results we obtained in \cite{CF1}, 
in \coref{known BM in BQ} we prove that there is an embedding
$$(\BM(k; k\Zz_{2m}) \times (k,+)^{r_1+...+r_n})\rtimes k^*\hookrightarrow \BQ(k; B\rtimes H)$$
for certain natural numbers $r_1,..., r_n$.

At the base level, since $B\rtimes H$ is quasitriangular, we may consider the subgroup $\BM(k; B\rtimes H)$
of $\BQ(k; B\rtimes H)$. We
study under which conditions it is stable under the action of $\Aut(B\rtimes H)$.
We prove that these are fulfilled by a certain subfamily $H(m, n)=B\rtimes H'$ of Hopf algebras $B\rtimes H$.
In \prref{stable subgr BM Hnm} we obtain the following subgroup of its quantum Brauer group:
$$(\BM(k; k\Zz_{2m}) \times (k,+)^{\frac{n(n+1)}{2}})\rtimes GL_n(k)/U_{2m} \hookrightarrow
\BQ(k, H(m,n)),$$
where $U_{2m}$ denotes the group of $2m$-th roots of unity.
In particular, we get approximations for the quantum Brauer
group of Sweedler's Hopf algebra, Radford Hopf algebra and Nichols Hopf algebra.

The Taft algebra is a Radford biproduct Hopf algebra for which the symmetricity condition in our main
theorem is not satisfied. So we can not approximate the quantum Brauer group on the level of the
corresponding braided Hopf algebra, but we do compute the sequence \equref{seq FZ} - on the base level -
contributing thus to the list of examples treated in \cite{VZ4}.
\par\bigskip

The paper is organized as follows. We first recall the necessary definitions and present some
basic structures related to a Hopf algebra $B$ in any braided monoidal category $\E$. The third
section treats group-likes and the Drinfel'd double of $B$ in $\E$. In the fourth one we
proceed to construct the group map from the Hopf automorphism group of $B$ in $\C$ to the respective
quantum Brauer group and study when the former acts by an inner action on the latter. We omit much of the
proofs in this part giving just a sketch of them in order to spear space and make the reading more fluent.
The fifth section carries the construction of the exact sequence containing the above
two groups. In the last section we apply the sequence to a family of Hopf algebras obtaining new results
on the corresponding quantum Brauer groups.
\par\bigskip

{\bf Aknowledgements.} The author wishes to thank to the Mathematical Institute of Serbian Academy of
Sciences and Arts in Belgrade, where a part of the work has been done, and to PEDECIBA from Uruguay
which supported me with a postdoctoral fellowship.



\section{Preliminaries}

Throughout let $k$ be a field 
and let $H$ be a 
Hopf algebra over $k$ with a bijective antipode $S$. A $k$-module $M$ which is simultaneously a left
$H$-module and a left $H$-comodule is called a left-left {\em Yetter-Drinfel'd $H$-module}
if it obeys the equivalent compatibility conditions:
\begin{equation}\eqlabel{YDll-1}
h_{(1)} m_{[-1]}\ot (h_{(2)}\cdot m_{[0]})=(h_{(1)}\cdot m)_{[-1]} h_{(2)}\ot (h_{(1)}\cdot m)_{[0]}
\end{equation}
\begin{equation}\eqlabel{YDll-2}
\lambda(h\cdot m)=h_{(1)} m_{[-1]}S(h_{(3)})\ot (h_{(2)}\cdot m_{[0]})
\end{equation}
for $m\in M, h\in H$. In the literature sometimes the term $H$-crossed module or Yang-Baxter
$H$-module is used. A {\em Yetter-Drinfel'd $H$-module algebra} is a Yetter-Drinfel'd $H$-module $A$
which is a left $H$-module algebra and a left $H$-comodule algebra. Left-left Yetter-Drinfel'd modules
and left $H$-linear and $H$-colinear maps form a braided monoidal category
$\C:={}_H ^H\YD$. A Yetter-Drinfel'd $H$-module algebra $A$ is then equivalently an algebra in $\C$.
A braiding $\Phi: M\ot N\to N\ot M$ in $\C$ and its inverse are given by
\begin{equation} \eqlabel{braidingD}
\Phi(m\ot n)=n_{[0]}\ot S^{-1}(n_{[-1]})\cdot m; \qquad \Phi^{-1}(m\ot n)=m_{[-1]}\cdot n\ot m_{[0]}
\end{equation}
for $m\in M, n\in N$. 

Now assume that $B$ is a Hopf algebra in $\C$ (a braided Hopf algebra). This situation occurs in the
Radford biproduct Hopf algebra $B\rtimes H$ where $H$ is an ordinary Hopf algebra \cite{Rad1}.
Concretely, one has that $B\rtimes H$ is a bialgebra, with 
$$\begin{array}{rl}
\textnormal{smash product:} 
\hskip-1em&(b\rt h)(b'\rt h')= b(h_{(1)}\tr b')\rt h_{(2)}h' \\
\textnormal{smash coproduct:\index{smash coproduct}} \quad &\Delta(b\rt h)=
(b_{(1)}\rt b_{{(2)}_{[-1]}} h_{(1)})\ot  (b_{{(2)}_{[0]}}\rt h_{(2)})
\end{array}$$
and componentwise unit and counit, if and only if $B$ is a bialgebra in $\C$ (that is : $B$ is a
left $H$-module algebra, a left $H$-module coalgebra, a left $H$-comodule algebra and a left
$H$-comodule coalgebra), it fulfills the (bialgebra compatibility in $\C$) condition:
$$
\Delta_B(ab)=a_{(1)}(a_{{(2)}_{[-1]}}\tr b_{(1)})\ot a_{{(2)}_{[0]}}b_{(2)}
$$
for $a, b\in B$ and it satisfies the equivalent Yetter-Drinfel'd conditions \equref{YDll-1} and
\equref{YDll-2}. Moreover, if $B$ has a convolution inverse $S_B$ to $id_B$, then $B\rtimes H$ is a
Hopf algebra with the antipode $S(b\rt h)=
(1_B\rt S_H(b_{[-1]}h))(S_B(b_{[0]})\rt 1_H)$.

\begin{rem}
Note that in the definitions of the smash product and coproduct, the bialgebra compatibility of $B$
and the antipode for $B\rtimes H$ we used the same formulas as in \cite[Theorem 2.1 and Proposition 2]{Rad1}
which are expressed in terms of our $\Phi^{-1}$ which usually is taken to be a braiding in ${}_H ^H\YD$.
Strictly speaking, we should change these definitions according to $\Phi$. However, when $B$ is commutative
or cocommutative it is a Hopf algebra both in $(\C, \Phi)$ and in $(\C, \Phi^{-1})$. Our examples
are of this form.
\end{rem}

Let ${}_B^B\YD(\C)$ be the category of left-left Yetter-Drinfel'd $B$-modules in $\C$.
An object $M\in {}_B^B\YD(\C)$ is a left-left Yetter-Drinfel'd $H$-module which additionally satisfies
the compatibility condition (the analogue of \equref{YDll-1} - instead of the ordinary flip the
braiding $\Phi^{-1}$ appears):
$$(b_{(1)}\pB({b_{(2)}}_{[-1]} \pH m))_{[-1]_B} \cdot (((b_{(1)}\pB({b_{(2)}}_{[-1]}\pH m))_{[0]_B})_{[-1]_H}\tr {b_{(2)}}_{[0]})
\rt ((b_{(1)}\pB({b_{(2)}}_{[-1]}\pH m))_{[0]_B})_{[0]_H} $$
\vspace{-0,4cm}
$$= b_{(1)} ({b_{(2)}}_{[-1]} \pH m_{[-1]_B}) \ot {b_{(2)}}_{[0]} m_{[0]_B}$$
for $m\in M$. Suppose both $S_B$ and $S_H$ are bijective.
By the left version of \cite[Proposition 4.2.3]{Besp} there is an isomorphism of braided monoidal categories
\begin{equation} \eqlabel{iso double-YD}
{}_B ^B\YD(\C) \iso {}_{B\rtimes H} ^{B\rtimes H}\YD.
\end{equation}
An object $M\in {}_B ^B\YD(\C)$ is a Yetter-Drinfel'd $B\rtimes H$-module by:
$(b\rt h)\cdot m=b\pB (h\pH m)$ and $\lambda(m)=m_{[-1]_B} \ot {m_{[0]_B}}_{[-1]_H} \ot {m_{[0]_B}}_{[0]_H}$,
whereas $N\in {}_{B\rtimes H} ^{B\rtimes H}\YD$ becomes an object in ${}_B ^B\YD(\C)$ via:
$h\pH n=(1_B\rt h)n, b\pB n=(b\rt 1_H)n, n_{[-1]_H} \ot n_{[0]_H}=(\Epsilon_B\ot H\ot N)\lambda^{B\rt H}(n)$
and $n_{[-1]_B} \ot n_{[0]_B}=(B\ot \Epsilon_H\ot N)\lambda^{B\rt H}(n)$ for $n\in N$.
The Yetter-Drinfel'd compatibility condition and the braiding $\Psi: M\ot N \to N\ot M$ and its inverse in
$\D:={}_{B\rtimes H} ^{B\rtimes H}\YD$ are analogous of \equref{YDll-1} and \equref{braidingD} where
$H$ should be replaced by $B\rtimes H$.

\subsection{The corresponding Brauer groups}

The Brauer group of left-right Yetter-Drinfel'd modules over a Hopf algebra $H$ over a commutative
ring $R$ was introduced in \cite{CVZ1}. That is the Brauer group of the braided monoidal category
${}_H\YD^{H^{op}}$, \cite{VZ1}. As we proved in \cite{Femic1}, the categories ${}_H\YD^{H^{op}}$ and
$\C$ are braided monoidally isomorphic (assume $R$ is a field). Since the Brauer group is invariant
under the equivalence of braided monoidal categories, the corresponding two Brauer groups are isomorphic.
Let us review the 
key pieces in the construction of the Brauer group of $\C$. A {\em Yetter-Drinfel'd $H$-opposite algebra}
$A^{op}$ of a left-left Yetter-Drinfel'd $H$-module algebra $A$ is $A^{op}=A$ as a vector space with the same
$H$-structures whose multiplication is given via:
$$a^{op} \cdot_{op} b^{op}= b_{[0]}\ot S^{-1}(b_{[-1]})\cdot a$$
for $a, b\in A$. In fact, $A^{op}$ is the opposite algebra of $A$ in the category $\C$.
A left-left {\em  Yetter-Drinfel'd $H$-Azumaya algebra} (equivalently, an Azumaya algebra in $\C$)
is a 
finite dimensional left-left Yetter-Drinfel'd $H$-module algebra $A$ for
which the two Yetter-Drinfel'd module algebra maps:
\begin{equation} \eqlabel{map F}
F: A\ot A^{op} \to \End(A), \quad F(a\ot b^{op})(c)=ac_{[0]} (S^{-1}(c_{[-1]})\cdot b),
\end{equation}
\begin{equation} \eqlabel{map G}
G: A^{op}\ot A \to \End(A)^{op}, \quad G(a^{op}\ot b)(c)=a_{[0]} (S^{-1}(a_{[-1]})\cdot c)b
\end{equation}
are isomorphisms. Note that the multiplication on $A\ot B$ for two algebras $A, B\in\C$ is given
(in terms of $\Phi$) by: $(a\ot b)(a'\ot b')=a a'_{[0]}\ot (S_H^{-1}(a'_{[-1]})\pH b)b'$.
The endomorphism ring $\End(M)$ of any 
finite dimensional left-left Yetter-Drinfel'd
$H$-module $M$ is a left-left Yetter-Drinfel'd $H$-module algebra with the $H$-structures
given by
\begin{equation}\eqlabel{H-mod end}
(h\cdot f)(m)=h_{(1)} f(S(h_{(2)})\cdot m)
\end{equation}
\begin{equation}\eqlabel{H-comod end}
\lambda(f)(m)=f_{[-1]}\ot f_{[0]}(m)=f(m_{[0]})_{[-1]}S^{-1}(m_{[-1]})\ot f(m_{[0]})_{[0]}
\end{equation}
for $h\in H, f\in\End(M)$ and $m\in M$.
Its Yetter-Drinfel'd $H$-opposite algebra $\End(M)^{op}$ is a left-left Yetter-Drinfel'd $H$-module algebra
with the following left $H$-module and $H$-comodule structures: \hspace{-0,3cm}
\begin{equation}\eqlabel{H-mod end-op}
(h\cdot f^{op})(m)=h_{(2)} f(S^{-1}(h_{(1)})\cdot m)
\end{equation}
\begin{equation}\eqlabel{H-comod end-op}
\lambda(f^{op})(m)=f^{op}_{[-1]}\ot f^{op}_{[0]}(m)=S(m_{[-1]})f(m_{[0]})_{[-1]}\ot f(m_{[0]})_{[0]}.
\end{equation}
(The Yetter-Drinfel'd compatibility for $\End(M)^{op}$ is proved using the left $H$-linear and $H$-colinear
isomorphism $M\ot M^*\to \End(M)$, where $M\ot M^*$ is a left-left Yetter-Drinfel'd $H$-module considered
as a left $H^{cop}$-module and $H^{op}$-comodule. Here $M^*$ -- the linear dual of $M$ -- is regarded with
the structures \equref{H-mod end-op} and \equref{H-comod end-op} with $f^{op}\in M^*$.)
\par\medskip

To consider the Brauer group of $\D$, replace 
the Hopf algebra $H$ by the Hopf algebra $B\rt H$ in the above construction. The braiding $\Phi$ in $\C$
is replaced by the braiding $\Psi$ in $\D$. To outline the difference in the structures in $\C$ and
in $\D$ we will denote the $B\rt H$-opposite algebra of $A$ by $\crta A$ and the Azumaya defining maps
in $\D$ by
$$
\crta F: A\crta\ot \crta A \to \End(A), \quad \crta G: \crta A\crta\ot A \to \crta\End(A).
$$
Here $\crta\ot$ stands for the tensor product in $\D$. The algebra $\crta A$ has the same $B\rt H$-structures
as $A$.
\par\medskip



\subsection{Preliminaries for the braided constructions}

The computations in the category ${}_B ^B\YD(\C)$ may become rather complicated. We will simplify them by
taking advantage of the powerful tool of the theory of braided monoidal categories and performing the
computations using braided diagrams. We assume the reader is familiar with the theory. For the definition
of basic algebraic structures in braided monoidal categories we refer to \cite{Besp}.
Let $\E$ be a braided monoidal category with braiding $\Phi$ (the same notation as for the braiding
in $\C$) and the unit object $I$. We assume $\E$ is strict. We use the following notation for the structures
in $\E$ of an algebra $A$, a coalgebra $C$, a Hopf algebra $H$ with antipode $S$, a left $A$-module
and a left $C$-comodule $M$, the braiding and a morphism $f: X\to Y$:
\begin{center}
\begin{tabular}{ccp{0.5cm}cccp{0.5cm}c}
\multicolumn{2}{c}{{\footnotesize Algebra}} & & \multicolumn{2}{c}{{\footnotesize Coalgebra}} & & {\footnotesize Antipode} \\
{\footnotesize Unit} & {\footnotesize Multiplication} & & {\footnotesize Counit} & {\footnotesize Comultiplication} & & \\
$\eta_A=
\gbeg{1}{3}
\got{1}{} \gnl
\gu{1} \gnl
\gob{1}{A}
\gend$ & $\nabla_A=\gbeg{2}{3}
\got{1}{A} \got{1}{A} \gnl
\gmu \gnl
\gob{2}{A}
\gend$ & & $\varepsilon_C=
\gbeg{1}{3}
\got{1}{C} \gnl
\gcu{1} \gnl
\gob{1}{}
\gend$ &  $\Delta_C=
\gbeg{2}{3}
\got{2}{C} \gnl
\gcmu \gnl
\gob{1}{C} \gob{1}{C}
\gend$ & & $\gbeg{3}{5}
\gvac{1} \got{1}{H} \gnl
\gvac{1} \gcl{1} \gnl
\gvac{1} \gmp{+} \gnl
\gvac{1} \gcl{1} \gnl
\gvac{1} \gob{1}{H}
\gend$ \vspace{5pt} \\
\multicolumn{2}{c}{{\footnotesize }} & & \multicolumn{2}{c}{\footnotesize Braiding} & & {\footnotesize Morphism}  \\
{\footnotesize Module} & {\footnotesize Comodule} & &  &  & \\
$\mu_M= \gbeg{3}{3}
\got{1}{A} \got{1}{M} \gnl
\glm \gnl
\gob{3}{M}
\gend$ & $\lambda_M=\gbeg{2}{3}
\got{3}{M} \gnl
\glcm \gnl
\gob{1}{C} \gob{1}{M}
\gend$ & &  $\Phi_{X,Y}=\gbeg{1}{3}
\got{1}{X} \got{1}{Y} \gnl
\gbr \gnl
\gob{1}{Y} \gob{1}{X}
\gend$ & $\quad \Phi_{Y,X}^{-1}=
\gbeg{2}{3}
\got{1}{X} \got{1}{Y} \gnl
\gibr \gnl
\gob{1}{Y} \gob{1}{X}
\gend$ & &
 $\gbeg{3}{5}
\gvac{1} \got{1}{X} \gnl
\gvac{1} \gcl{1} \gnl
\gvac{1} \gbmp{f} \gnl
\gvac{1} \gcl{1} \gnl
\gvac{1} \gob{1}{Y}
\gend$
\end{tabular}
\end{center}
We will normally suppress the strings for the unit object $I$ as well as writing it when it appears
as a domain or a codomain of morphisms. The sign ``+'' stands for the antipode $S$, whereas the
sign ``-'' denotes its inverse.
We will consider the comultiplication of the co-opposite coalgebra of $C\in\E$ to be given by
$\Delta_{C^{cop}}=\Phi^{-1}\Delta_C$ rather than with the positive sign of the braiding - in order that the
opposite co-opposite $B^{op, cop}$ of a
bialgebra $B\in\E$ be a bialgebra, see \cite[1.3, 1.6]{Tak2}, \cite{Femic1}.

We will assume that $\E$ is a closed category, meaning that the tensor functors have their (right) adjoints,
which give rise to the inner hom-objects $[X,Y]$ with $X,Y\in\E$. The counit of the adjunction
$(-\ot X, [X, -])$ is denoted by $ev_{[X,Y]}: [X,Y]\ot X\to Y$.
An object $M$ is {\em finite} in $\E$ if the morphism $db: M\ot [M,I]\to [M,M]$ defined via
$ev_{[M,M]}(db\ot M)=M\ot ev_{[M,I]}$ is an isomorphism.
An object $M^*\in\E$ together with a morphism $e_M:M^*\ot M \to I$ is called a {\em dual object} for $M$
if there exists a morphism $c_M:I\to M\ot M^*$ such that $(M \otimes e_M)(c_M \ot M)=id_M$ and
$(e_M \ot M^*)(M^* \ot c_M)=id_{M^*}$. The morphisms $e_M$ and $c_M$ are called {\em evaluation} and
{\em coevaluation} respectively and we denote them in braided diagrams by: \vspace{-0,44cm}
$$e_M=
\gbeg{3}{1}
\got{1}{M^*} \got{1}{M} \gnl
\gev \gnl
\gob{1}{}
\gend
\qquad\textnormal{and}\qquad
c_M=
\gbeg{3}{3}
\got{1}{} \gnl
\gdb \gnl
\gob{1}{M} \gob{2}{\hspace{-0,2cm}M^*.}
\gend 
$$
For a finite object $M$ it is actually $M^*=[M,I]$ and $\ev_{[M,I]}=\ev_{M^*}$. A morphism between
finite objects $f: M\to N$ induces a morphism $f^*: N^*\to M^*$ so that $\ev_{M^*}(f^*\ot M)=
\ev_{N^*}(N^*\ot f)$.
A finite object $M$ is {\em faithfully projective} if $\tilde\ev\w: \w [M,I] \ot_{[M,M]}M \to I$
induced by $\ev_{[M,I]}$ is an isomorphism.

On inner hom-objects there is an associative
pre-multiplication $\Fi_{M, Y, Z}: [Y, Z]\ot [M, Y]\to [M, Z]$ and a unital morphism
$\eta_{[Y, Y]}: I\to [Y, Y]$ such that \vspace{-0,2cm}
$$\Fi_{M, Y, Y}(\eta_{[Y, Y]}\ot [M, Y])=\lambda_{[M, Y]}\qquad\textnormal{and}\qquad
\Fi_{M, M, Y}([M, Y]\ot \eta_{[M, M]})=\rho_{[M, Y]}$$
where $\lambda$ and $\rho$ are unity constraints. The morphisms $\Fi_{M, Y, Z}$ and
$\eta_{[Y, Y]}$ are given via the universal properties of $ev_{[M, Z]}$ and $ev_{[Y, Y]}$,
respectively, by $\ev_{[M, Z]}(\Fi_{M, Y, Z} \ot M)=\ev_{[Y, Z]}([Y, Z]\ot\ev_{[M, Y]})$ and
$\ev_{[Y, Y]}(\eta_{[Y, Y]} \ot Y)=\lambda_Y$. Consequently, for each $M\in\E$, the object
$[M, M]$ is equipped with an algebra structure via $\Fi_{M, M, M}$ and $\eta_{[M, M]}$.

For a consequence of the adjunction isomorphisms $\E(B\ot M, M) \cong \E(B, [M, M])$ one
has that an object $M\in\E$ is a left $B$-module if and only if there is an algebra morphism
$\teta: B \to [M, M]$ in $\E$. If $\mu:B \otimes M \rightarrow M$ is the structure morphism, then
$\theta$ is the unique morphism such that
\begin{equation} \eqlabel{teta}
ev_{[M,M]}(\theta \otimes M)=\mu.
\end{equation}


Suppose that $\Phi_{B,B}$ is {\em symmetric}, that is, $\Phi_{B,B}=\Phi_{B,B}^{-1}$. Then if $M$ is a
left $B$-module, so are $M^*$ and $[M, M]$ where the action of the latter is given by: \vspace{-0,7cm}
\begin{center}
\begin{tabular}{p{4.8cm}p{2cm}p{5.8cm}}
\begin{eqnarray}  \eqlabel{H-mod-end-eq}
\scalebox{0.9}[0.9]{
\gbeg{3}{5}
\got{1}{B} \got{3}{[M, M]} \gnl
\gcn{1}{1}{1}{3} \gvac{1} \gcl{1} \gnl
\gvac{1} \glm \gnl
\gvac{2}\gcl{1}\gnl
\gvac{2}\gob{1}{[M, M]}
\gend}=\scalebox{0.9}[0.9]{
\gbeg{6}{8}
\got{2}{B}\got{1}{[M, M]} \gnl
\gcmu  \gcl{2} \gvac{1} \gnl
\gcl{1} \gmp{+} \gnl
\gcl{1} \gbr  \gnl
\gbmp{\teta} \gcl{1} \gbmp{\teta}  \gnl
\gcn{1}{1}{1}{0} \gmu \gnl
\hspace{-0,22cm}\gwmu{3} \gnl
\gvac{1} \gob{1}{[M, M]}
\gend}
\end{eqnarray}  & &
\begin{eqnarray} \eqlabel{inner H-mod eval}  \textnormal{i.e.}\quad
\scalebox{0.9}[0.9]{
\gbeg{5}{6}
\got{1}{B} \got{3}{[M, M]} \got{1}{M} \gnl
\gcn{1}{1}{1}{3} \gvac{1} \gcl{1} \gvac{1} \gcl{2} \gnl
\gvac{1} \glm \gnl
\gvac{2} \gcn{1}{1}{1}{3}\gcn{1}{1}{3}{1} \gnl
\gvac{3} \gmp{\ev} \gnl
\gvac{3} \gob{1}{M}
\gend}=\scalebox{0.9}[0.9]{
\gbeg{6}{9}
\got{2}{B}\got{1}{[M, M]}\got{3}{M} \gnl
\gcmu  \gcl{2}\gvac{1} \gcl{2}\gnl
\gcl{1}\gmp{+} \gvac{2}  \gnl
\gcl{1} \gbr \gcn{1}{1}{3}{1} \gnl
\gcl{1} \gcl{1} \glm  \gnl
\gcl{1} \gcn{1}{1}{1}{3}\gcn{1}{1}{3}{1} \gnl
\gcn{1}{1}{1}{3} \gvac{1} \gmp{\ev} \gnl
\gvac{1} \glm \gvac{1} \gnl
\gvac{2} \gob{1}{M.}
\gend}
\end{eqnarray}
\end{tabular}
\end{center} \vspace{-0,5cm}
A $B$-action on $[M, M]$ is said to be {\em strongly inner} if there is an algebra morphism
$\theta: B\to [M, M]$ such that the identity \equref{H-mod-end-eq} holds. Then the following is obvious:

\begin{lma}  \lelabel{strongly inner}
Let $M$ be an object in $\E$.
\begin{enumerate}
\item A $B$-action on $M$ induces a strongly inner action on $[M,M]$.
\item If $[M,M]$ is a $B$-module algebra and the $B$-action on $[M,M]$ is strongly inner, then
\equref{teta} defines a $B$-action on $M$ which induces the $B$-action on $[M,M]$.
\end{enumerate}
\end{lma}
\par\smallskip

Suppose now that $M$ and $N$ are faithfully projective and that the $B$-actions on $[M,M]$ and $[N,N]$
are strongly inner. Let $\teta_M$ and $\teta_N$ 
denote the corresponding algebra morphisms. Recall that the $B$-action $\mu_{M\ot N}$ on $M\ot N$ is
given by $\mu_{M\ot N}= (\mu_M\ot\mu_N)(B\ot\Phi_{B,M}\ot N)(\Delta_B\ot M\ot N)$, and let $\teta_{M\ot N}$
be the corresponding algebra morphism. Let
\begin{equation} \eqlabel{omega}
\omega: [M,M] \ot [N,N] \to [M\ot N, M\ot N]
\end{equation}
denote the algebra isomorphism proved in \cite[Proposition 2.3]{VZ1}, it is given by
$ev_{[M\ot N, M\ot N]}(\omega\ot M\ot N)=(ev_{[M,M]}\ot ev_{[N,N]})([M,M] \ot \Phi_{[N,N],M}\ot N)$.
Then by \equref{teta}, naturality and the definition of $\omega$ we have:
$ev_{[M\ot N, M\ot N]}(\teta_{M \ot N}\ot M\ot N)=(ev_{[M,M]}\ot ev_{[N,N]})([M,M] \ot
\Phi_{[N,N],M}\ot N)((\teta_M\ot\teta_N)\Delta_B\ot M\ot N)=ev_{[M\ot N, M\ot N]}(\omega\ot M\ot N)
((\teta_M\ot\teta_N)\Delta_B\ot M\ot N)$, which by the universal property of $ev_{[M\ot N, M\ot N]}$
yields:
\begin{equation} \eqlabel{teta-tensor}
\teta_{M\ot N}=\omega(\teta_M\ot\teta_N)\Delta_B.
\end{equation}
\par\smallskip

If $M$ is a left $B$-comodule, then so is $[M, M]$ with the coaction given by:
\begin{eqnarray} \eqlabel{H-comod-end-eq}
\gbeg{4}{5}
\got{1}{} \got{1}{[M,M]} \got{3}{M} \gnl
\glcm \gvac{1} \gcl{1} \gnl
\gcl{2} \gcn{1}{1}{1}{3}\gcn{1}{1}{3}{1} \gnl
\gvac{2}  \gmp{\ev} \gnl
\gob{1}{B} \gvac{1} \gob{1}{M}
\gend=\gbeg{3}{10}
\got{1}{[M,M]} \gvac{2} \got{1}{M} \gnl
\gcl{2} \gvac{1} \glcm \gnl
\gvac{1} \gcn{1}{1}{3}{1} \gvac{1} \gcl{1} \gnl
\gbr \gvac{1} \gcl{1} \gnl
\gmp{-} \gcn{1}{1}{1}{3}\gcn{1}{1}{3}{1} \gnl
\gcl{2} \gvac{1} \gmp{\ev} \gnl
\gvac{1} \glcm \gnl
\gbr \gcl{2} \gnl
\gmu \gnl
\gob{2}{B} \gob{1}{M}
\gend
\end{eqnarray}
if $\Phi_{B,B}$ is symmetric. Note that
\equref{inner H-mod eval} and \equref{H-comod-end-eq} generalize \equref{H-mod end} and \equref{H-comod end}.

The Yetter-Drinfel'd compatibility conditions \equref{YDll-1} and \equref{YDll-2} generalized
to any Hopf algebra $B$ in $\E$ in braided diagrams read as: \vspace{-0,8cm}
\begin{center}
\begin{tabular}{p{6cm}p{2cm}p{6cm}}
\begin{equation} \eqlabel{left YD}
\gbeg{4}{8}
\got{2}{B} \got{1}{M} \gnl
\gcmu \gcl{1} \gnl
\gcl{1} \gbr \gnl
\glm \gcl{2} \gnl
\glcm \gnl
\gcl{1} \gbr \gnl
\gmu \gcl{1} \gnl
\gob{2}{B} \gob{1}{M}
\gend=
\gbeg{4}{5}
\got{2}{B} \got{3}{M} \gnl
\gcmu \glcm \gnl
\gcl{1} \gbr \gcl{1} \gnl
\gmu \glm \gnl
\gob{2}{B} \gob{3}{M}
\gend
\quad
\end{equation}
& & \vspace{-0,6cm}
\begin{equation} \eqlabel{left YD-oth} \hspace{-0,2cm}
\gbeg{4}{8}
\got{1}{B} \got{5}{M} \gnl
\gcn{2}{2}{1}{5} \gvac{1} \gcl{2} \gnl \gnl
\gvac{2} \glm \gnl
\gvac{2} \glcm \gnl
\gcn{2}{2}{5}{1} \gvac{1} \gcl{2} \gnl \gnl
\gob{1}{B} \gob{5}{M}
\gend=
\gbeg{6}{9}
\gvac{1} \got{2}{B} \got{4}{M} \gnl
\gvac{1} \gcmu \gcn{1}{1}{4}{4} \gnl
\gvac{1} \gcn{1}{1}{1}{0} \hspace{-0,21cm} \gcmu \glcm \gnl
\gvac{1} \gcl{3} \gbr \gcl{1} \gcl{3} \gnl
\gvac{2}  \gmp{+} \gbr \gnl
\gvac{1} \gvac{1} \gbr \gcl{1} \gnl
\gvac{1} \gcn{1}{1}{1}{2} \gmu \glm \gnl
\gvac{2} \hspace{-0,21cm} \gmu \gcn{1}{1}{4}{4} \gnl
\gvac{2} \gob{2}{B} \gob{4}{M}
\gend
\end{equation}
\end{tabular}
\end{center} \vspace{-0,5cm}
One may easily show that equipped with the regular action and with the adjoint coaction: \vspace{-0,3cm}
$$
\scalebox{0.88}{\gbeg{3}{5}
\gvac{1} \got{3}{B} \gnl
\gvac{2} \gcl{1} \gnl
\gvac{1} \glcm \gnl
\gcn{1}{1}{3}{1} \gvac{1} \gcl{1} \gnl
\gob{1}{B} \gob{3}{B}
\gend}=\scalebox{0.88}{
\gbeg{6}{6}
\gvac{1} \got{3}{B} \gnl
\gvac{1} \gwcm{3} \gnl
\gwcm{3} \gmp{+} \gnl
\gcl{1} \gvac{1} \gbr \gnl
\gwmu{3} \gcl{1} \gnl
\gvac{1} \gob{1}{B} \gob{3}{B}
\gend}
$$
$B$ itself is an object in ${}_B ^B\YD(\E)$ if $\Phi_{B, B}$ is symmetric.

Throughout we will write {\em $\Phi_{B, M}$ is symmetric for all $M\in {}_B ^B\YD(\E)$} by abuse of
notation, when strictly speaking we should say {\em for all $M\in\E$}. Indeed, through the
forgetful functor $\U: {}_B ^B\YD(\E) \to\E$ every $M\in {}_B ^B\YD(\E)$ is an object in $\E$, and
every $N\in\E$ can be equipped with trivial $B$-(co)module structures to make an object in ${}_B ^B\YD(\E)$.
We proved in \cite{Femic1} that if $\Phi_{B, M}$ is symmetric for $M\in {}_B ^B\YD(\E)$, then
the expressions \equref{left YD} and \equref{left YD-oth} are equivalent and ${}_B ^B\YD(\E)$ is a braided
monoidal category. Its tensor product we denote by $\crta\ot$. The braiding $\Psi: M\crta\ot N\to N\crta\ot M$
in ${}_B ^B\YD(\E)$ and its inverse 
are given by: \vspace{-0,2cm}
\begin{equation} \eqlabel{br in YD}
\Psi_{M, N}=\scalebox{0.9}[0.9]{
\gbeg{4}{7}
\gvac{1} \got{1}{M} \got{1}{N} \gnl
\gvac{1} \gibr \gnl
\glcm \gcl{3} \gnl
\gmp{-} \gcl{1} \gnl
\gibr \gnl
\gcl{1} \glm \gnl
\gob{1}{N} \gvac{1} \gob{1}{M}
\gend} \qquad\textnormal{and}\qquad
\Psi_{M, N}^{-1}=
\gbeg{3}{5}
\got{1}{} \got{1}{N} \got{1}{M} \gnl
\glcm \gcl{1} \gnl
\gcl{1} \gbr \gnl
\glm \gcl{1} \gnl
\gob{1}{} \gob{1}{M} \gob{1}{N}
\gend
\end{equation}
Note that these generalize \equref{braidingD}. The opposite algebra of any algebra $A$ in ${}_B ^B\YD(\E)$
we denote by $\crta A$. When $\Phi_{B, A}$ is symmetric it can be shown that $\crta A$
is an algebra in ${}_B ^B\YD(\E)$ with the same
$B$-structures as $A$. For $M\in {}_B ^B\YD(\E)$ one has that $[M,M]\in {}_B ^B\YD(\E)$
with the structures \equref{inner H-mod eval} and \equref{H-comod-end-eq}.
\par\medskip

We finally outline the basic facts about the Brauer group $\Br(\E)$ of $\E$ from \cite{VZ1}. 
An algebra $A\in\C$ is called an {\em Azumaya algebra} if
the functors $A\ot -:\C\to {}_{A\ot A^{op}}\C$ and $-\ot A:\C\to \C_{A^{op}\ot A}$
establish $\C$-module equivalences of categories. For a characterization of an Azumaya algebra
in terms of the morphisms which generalize \equref{map F} and \equref{map G} see \cite[Theorem 3.1]{VZ1}
and \cite[Theorem 5.1]{Femic2}.
For any faithfully projective object $M\in\E$ one has that $[M,M]$ is an Azumaya algebra in $\E$.
If $A$ is an Azumaya algebra in $\E$, then so is $A^{op}$. If $A$ and $B$ are Azumaya algebras in
$\E$, then so is $A\ot B$. Two Azumaya algebras $A$ and $B$ are said to be {\em Brauer equivalent} if
there are faithfully projective objects $M,N\in\E$ so that $A\ot [M,M]\iso B\ot [N,N]$ as algebras
in $\E$. The equivalence class of $[M,M]$ for a faithfully projective object
$M\in\E$ is the unit in $\Br(\E)$ and the inverse of the class of an Azumaya algebra $A$ is the class
of its opposite algebra $A^{op}$. The Brauer group of the category of left-left Yetter-Drinfel'd modules
over $B$ in $\E$ is denoted by $\BQ(\E; B):=\Br({}_B ^B\YD(\E))$ and is called the {\em quantum Brauer group}.
In \cite{Femic1} we proved that ${}_B ^B\YD(\E)$ is braided monoidally isomorphic to the category
${}_{D(B)}\E$ of left modules over the Drinfel'd double $D(B)=(B^{op})^*\bowtie B$ in $\E$ when $B$ is
finite in $\E$ and if 
$\Phi_{B, M}$ is symmetric for all $M\in {}_B ^B\YD(\E)$. 
As a matter of fact, the latter symmetricity condition on $\Phi$ is
unavoidable in order to define any monoidal functor between ${}_{D(B)}\E$ and ${}_B ^B\YD(\E)$, given
that a left $B^*$-module structure gives rise to a left $B$-comodule structure and vice versa in any
such a functor. An isomorphism functor \hspace{-0,9cm} $
\bfig
\putmorphism(200,30)(1,0)[\F:\q {}_{D(B)}\E` {}_B^B\YD(\E): \G`]{800}0a
\putmorphism(110,50)(1,0)[\phantom{{}_H\YD(\C)^{H^{op}}: \G}`\phantom{\F: {}_{(H^{op})^*\bowtie H}\C}` ]{900}1a
\putmorphism(100,10)(1,0)[\phantom{{}_H\YD(\C)^{H^{op}}: \G}`\phantom{\F: {}_{(H^{op})^*\bowtie H}\C}` ]{880}{-1}b
\efig
$ is given by: \vspace{-1,2cm}
\begin{center}
\begin{tabular}{p{4.8cm}p{2cm}p{5.8cm}}
\begin{eqnarray} \eqlabel{Hcomod}
\scalebox{0.9}[0.9]{
\gbeg{3}{4}
\gvac{1} \got{3}{\F(M)} \gnl
\gvac{1} \glcm \gnl
\gcn{1}{1}{3}{1} \gvac{1} \gcl{1} \gnl
\gob{1}{B} \gob{3}{\F(M)}
\gend} = \scalebox{0.9}[0.9]{
\gbeg{7}{5}
\got{1}{} \got{1}{} \got{1}{M} \gnl
\gdb \gcl{2} \gnl
\gcl{1} \gmp{+\hspace{-0,08cm}^*} \gnl
\gcl{1} \glm \gnl
\gob{1}{B} \gob{3}{M}
\gend}
\end{eqnarray}  & &
\begin{eqnarray} \eqlabel{H*mod}
\scalebox{0.9}[0.9]{
\gbeg{3}{4}
\got{1}{B^*} \gvac{1} \got{1}{\G(N)} \gnl
\gcn{1}{1}{1}{3} \gvac{1} \gcl{1} \gnl
\gvac{1} \glm \gnl
\gvac{2} \gob{1}{\G(N)}
\gend} = \scalebox{0.9}[0.9]{
\gbeg{2}{5}
\got{1}{B^*} \got{3}{N} \gnl
\gcl{1} \glcm \gnl
\gcl{1} \gmp{-} \gcl{2} \gnl
\gev \gnl
\gob{5}{N}
\gend}
\end{eqnarray}
\end{tabular}
\end{center} \vspace{-0,5cm}
with $M\in {}_{D(B)}\E$ and $N\in {}_B^B\YD(\E)$.
Consequently, we have isomorphic Brauer groups: $\BQ(\E; B)\iso\Br({}_{D(B)}\E)=\BM(\E; D(B))$.
As a consequence of \leref{strongly inner} we have:

\begin{cor} \colabel{D(H)-strongly inner}
Suppose that $[M,M]$ is a $D(B)$-Azumaya algebra. It is $[[M,M]]=1$ in $\BM(\E; D(B))$ if and only if
the $D(B)$-action on $[M,M]$ is strongly inner.
\end{cor}

\begin{proof}
The condition $[[M,M]]=1$ in $\BM(\E; D(B))$ means that there is a $D(B)$-module $N$ so that the
$D(B)$-action on $[N,N]$ is strongly inner (with an algebra morphism $\theta: D(B)\to [N,N]$), and that
$[M,M]\iso [N,N]$ as $D(B)$-module algebras (with a $D(B)$-module algebra isomorphism $\Fi: [M,M]\to [N,N]$).
Then the algebra morphism $\Fi^{-1}\theta$ makes the $D(B)$-action on $[M,M]$ strongly inner. The converse
follows by \leref{strongly inner}, 2).
\qed\end{proof}

We recall that the {\em right center} of an algebra $A\in\E$ is an object ${}^A A$ together with a
(mono)morphism $i_r: {}^A A \to A$ in $\E$ such that $\nabla_A(i_r\ot A)=\nabla_A\Phi_{A,A}(i_r\ot A)$,
and if $T$ with $t: T\to A$ is another pair fulfilling $\nabla_A(t\ot A)=\nabla_A\Phi_{A,A}(t\ot A)$,
then there is a unique morphism $\crta t: T\to {}^A A$ such that $i_r\crta t=t$. Similarly one defines
the left center of $A$ (compare \cite[Definitions 3.1 and 3.2, identities (8) and (9), Corollary 3.4]{Femic2}).
An Azumaya algebra is left and right central, meaning that its left and right centers are trivial
(isomorphic to $I$), \cite[Proposition 5.3]{Femic2}.



\section{Group-likes of a braided Hopf algebra}

In this section $B$ denotes a finite Hopf algebra in a braided monoidal category $\E$.
We recall that an algebra morphism $B^*\to I$, or equivalently, a coalgebra morphism $I\to B$ in
$\E$ is called a {\em group-like on $B$}, \cite{Tak2}.
Note that any group-like $g$ is convolution invertible, its convolution inverse being $Sg$ or $gS$
depending on whether we consider $g$ as a coalgebra or an algebra morphism, respectively.
The group of group-likes on $B$ will be denoted by $G(B)$. The set of algebra morphisms $B\to I$
we will denote by $\Alg(B,I)$ and the set of coalgebra morphisms $I\to B$ by $\Coalg(I, B)$.

\begin{lma} \lelabel{mods and grouplks}
Let $M\in\E$ be faithfully projective and let $\sigma, \tau: B\to [M,M]$ define two $B$-module structures
on $M$. The respective $B$-modules we denote by $M_{\sigma}$ and $M_{\tau}$. The induced
$B$-module structures on $[M_{\sigma},M_{\sigma}]$ and $[M_{\tau},M_{\tau}]$ coincide if and only if
$\tau^{-1}*\sigma$ factors through a group-like on $B^*$.
\end{lma}

\begin{proof}
The fact that $[M_{\sigma},M_{\sigma}]=[M_{\tau},M_{\tau}]$ as $B$-modules means that the right
hand-side of \equref{H-mod-end-eq} with $\theta=\sigma$ equals to the analogous diagram with
$\theta=\tau$. This is equivalent (compose to it: $(\tau S\ot -)(\Delta_B \ot [M,M])$ from above
and then $\nabla_{[M,M]}$ from below) to $\star$ in:
\begin{equation} \eqlabel{tau-S}
\gbeg{2}{6}
\got{1}{B} \got{1}{A} \gnl
\gmp{+} \gcl{1} \gnl
\gbr \gnl
\gcl{1} \gbmp{\tau} \gnl
\gmu \gnl
\gob{2}{A}
\gend=\gbeg{4}{9}
\gvac{1} \got{2}{B} \got{1}{A} \gnl
\gvac{1} \gcmu \gcl{3} \gnl
\gcn{1}{1}{3}{2} \gvac{1} \gcl{1} \gnl
\gcmu \gmp{+} \gnl
\gmp{+} \gcl{1} \gbr \gnl
\gbmp{\tau} \gbmp{\tau} \gcl{1} \gbmp{\tau} \gnl
\gmu \gmu \gnl
\gvac{1} \hspace{-0,22cm} \gwmu{3} \gnl
\gvac{1} \gob{3}{A}
\gend=\gbeg{4}{11}
\got{2}{B} \gvac{1} \got{1}{A} \gnl
\gcmu \gvac{1} \gcl{4} \gnl
\gcl{1} \gcn{1}{1}{1}{2} \gnl
\gcl{1} \gcmu \gnl
\gcl{1} \gcl{1} \gmp{+} \gnl
\gmp{+} \gcl{1} \gbr \gnl
\gbmp{\tau} \gbmp{\tau} \gcl{1} \gbmp{\tau} \gnl
\gcl{1} \gcn{1}{1}{1}{2} \gmu \gnl
\gcl{1} \gvac{1} \hspace{-0,34cm} \gmu \gnl
\gvac{1} \hspace{-0,22cm} \gwmu{3} \gnl
\gvac{2} \gob{1}{A}
\gend\stackrel{\star}{=}
\gbeg{4}{11}
\got{2}{B} \gvac{1} \got{1}{A} \gnl
\gcmu \gvac{1} \gcl{4} \gnl
\gcl{1} \gcn{1}{1}{1}{2} \gnl
\gcl{1} \gcmu \gnl
\gcl{1} \gcl{1} \gmp{+} \gnl
\gmp{+} \gcl{1} \gbr \gnl
\gbmp{\tau} \gbmp{\sigma} \gcl{1} \gbmp{\sigma} \gnl
\gcl{1} \gcn{1}{1}{1}{2} \gmu \gnl
\gcl{1} \gvac{1} \hspace{-0,34cm} \gmu \gnl
\gvac{1} \hspace{-0,22cm} \gwmu{3} \gnl
\gvac{2} \gob{1}{A}
\gend=\gbeg{4}{9}
\gvac{1} \got{2}{B} \got{1}{A} \gnl
\gvac{1} \gcmu \gcl{3} \gnl
\gcn{1}{1}{3}{2} \gvac{1} \gcl{1} \gnl
\gcmu \gmp{+} \gnl
\gmp{+} \gcl{1} \gbr \gnl
\gbmp{\tau} \gbmp{\sigma} \gcl{1} \gbmp{\sigma} \gnl
\gmu \gmu \gnl
\gvac{1} \hspace{-0,22cm} \gwmu{3} \gnl
\gvac{1} \gob{3}{A}
\gend
\end{equation}
where $A=[M,M]$. The morphism $\chi:=\tau S *\sigma: B\to A$ is convolution invertible with the
convolution inverse $\sigma S *\tau$. Note that \equref{tau-S} implies $\star$ in:
$$
\gbeg{2}{5}
\got{1}{B} \got{1}{A} \gnl
\gbr \gnl
\gcl{1} \gbmp{\chi} \gnl
\gmu \gnl
\gob{2}{A}
\gend=
\gbeg{3}{9}
\got{2}{B} \got{1}{A} \gnl
\gcmu \gcl{2} \gnl
\gmp{+} \gbmp{\sigma} \gnl
\gcl{1} \gbr \gnl
\gbr \gcl{2} \gnl
\gcl{1} \gbmp{\tau} \gnl
\gmu \gcn{1}{1}{1}{0} \gnl
\gvac{1} \hspace{-0,34cm} \gmu \gnl
\gvac{1} \gob{2}{A}
\gend\stackrel{\star}{=}
\gbeg{5}{11}
\gvac{3} \got{1}{B} \got{3}{A} \gnl
\gvac{2} \gwcm{3} \gcl{2} \gnl
\gvac{1} \gwcm{3} \gbmp{\sigma} \gnl
\gvac{1} \hspace{-0,36cm} \gcmu \gcn{1}{1}{2}{1} \gcn{1}{1}{2}{1} \gcn{1}{1}{2}{1} \gnl
\gvac{1} \gcl{1} \gcl{1} \gmp{+} \gbr \gnl
\gvac{1} \gmp{+} \gcl{1} \gbr \gcl{3} \gnl
\gvac{1} \gbmp{\tau} \gbmp{\sigma} \gcl{1} \gbmp{\sigma} \gnl
\gvac{1} \gmu \gmu \gnl
\gvac{2} \hspace{-0,2cm} \gwmu{3} \gcn{1}{1}{2}{1} \gnl
\gvac{3} \gwmu{3} \gnl
\gvac{4} \gob{1}{A}
\gend=
\gbeg{7}{10}
\gvac{2} \got{2}{B} \gvac{2} \got{1}{A} \gnl
\gvac{1} \gwcm{4} \gvac{1} \gcl{4} \gnl
\gwcm{3} \gwcm{3} \gnl
\gmp{+} \gvac{1} \gcl{1} \gmp{+} \gvac{1} \gcl{1} \gnl
\gbmp{\tau} \gvac{1} \gbmp{\sigma} \gbmp{\sigma} \gvac{1} \gbmp{\sigma} \gnl
\gwmu{3} \gwmu{3} \gcn{1}{1}{1}{-1} \gnl
\gvac{1} \gcl{1} \gvac{2} \gbr \gnl
\gvac{1} \gcn{1}{1}{1}{2} \gvac{2} \gmu \gnl
\gvac{2} \hspace{-0,2cm} \gwmu{4} \gnl
\gvac{3} \gob{2}{A}
\gend=\gbeg{2}{5}
\got{1}{B} \got{1}{A} \gnl
\gcl{1} \gcl{2} \gnl
\gbmp{\chi} \gnl
\gmu \gnl
\gob{2}{A.}
\gend
$$
Observe that the above identity means that $\chi$ factors through the right center of $A$. But $A$
is Azumaya, so $\chi$ factors through $\crta\chi: B\to I$. Being
$\sigma$ and $\tau$ algebra morphisms, so becomes $\crta\chi$ and we have proved the direct sense.

For the converse, note that we have
\begin{equation} \eqlabel{sigma-tau related}
\gbeg{2}{5}
\got{1}{B} \gnl
\gcl{1} \gnl
\gbmp{\sigma} \gnl
\gcl{1} \gnl
\gob{1}{A}
\gend=
\gbeg{2}{5}
\got{2}{B} \gnl
\gcmu \gnl
\gbmp{\tau} \gbmp{\lambda} \gnl
\gcl{1} \gnl
\gob{1}{A} \gob{1}{I}
\gend
\end{equation}
for some algebra morphism $\lambda: B\to I$. Let us write this symbolically as $\sigma=\tau*\lambda$.
Then $\sigma^{-1}=\lambda^{-1}*\tau^{-1}$ and
we find:
$$\gbeg{4}{8}
\got{2}{B}\got{1}{A} \gnl
\gcmu  \gcl{2} \gvac{1} \gnl
\gcl{1} \gmp{+} \gnl
\gcl{1} \gbr  \gnl
\gbmp{\sigma} \gcl{1} \gbmp{\sigma}  \gnl
\gcn{1}{1}{1}{0} \gmu \gnl
\hspace{-0,22cm}\gwmu{3} \gnl
\gvac{1} \gob{1}{A}
\gend=
\gbeg{6}{8}
\gvac{2} \got{1}{B} \got{3}{A} \gnl
\gvac{1} \gwcm{3} \gcl{1} \gnl
\gvac{1} \gcn{1}{1}{1}{0} \gvac{1} \gbr \gnl
\gcmu \gcn{1}{1}{3}{2} \gvac{1} \hspace{-0,34cm} \gcmu \gnl
\gvac{1} \hspace{-0,22cm} \gbmp{\tau} \gbmp{\lambda} \gcn{1}{1}{2}{2} \gvac{1} \hspace{-0,34cm}  \gbmp{\hspace{0,14cm}\lambda^{\hspace{-0,08cm}-}}
  \gbmp{\hspace{0,14cm}\tau^{\hspace{-0,08cm}-}} \gnl
\gvac{1} \gcn{1}{1}{2}{3} \gvac{2} \gwmu{3} \gnl
\gvac{2} \hspace{-0,14cm} \gwmu{4} \gnl
\gvac{3} \gob{2}{A}
\gend\stackrel{coass.}{\stackrel{\lambda*\lambda^{-1}}{=}}
\gbeg{4}{8}
\got{2}{B}\got{1}{A} \gnl
\gcmu  \gcl{2} \gvac{1} \gnl
\gcl{1} \gmp{+} \gnl
\gcl{1} \gbr  \gnl
\gbmp{\tau} \gcl{1} \gbmp{\tau}  \gnl
\gcn{1}{1}{1}{0} \gmu \gnl
\hspace{-0,22cm}\gwmu{3} \gnl
\gvac{1} \gob{1}{A}
\gend
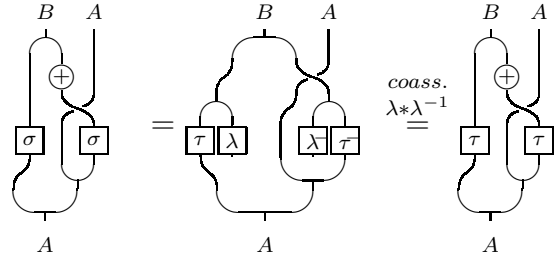
$$
\qed\end{proof}

\begin{defn}
An automorphism $\Fi$ of $B\in\E$ is called {\em inner} if there exists a convolution invertible
morphism $g:I\to B$ fulfilling \equref{spin-down}.
An automorphism $\psi$ of $B\in\E$ is called {\em co-inner} if there exists a convolution invertible
morphism $\lambda:B\to I$ fulfilling \equref{spin-up}. 
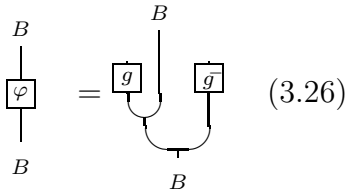
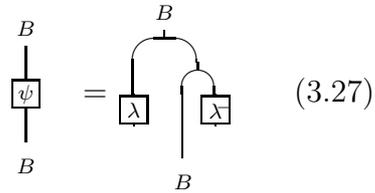
\vspace{-0,4cm}
\begin{center}
\begin{tabular}{p{4.8cm}p{2cm}p{5.8cm}}
\begin{eqnarray} \eqlabel{spin-down}
\gbeg{2}{5}
\got{1}{B} \gnl
\gcl{1} \gnl
\gbmp{\Fi} \gnl
\gcl{1} \gnl
\gob{1}{B}
\gend=
\gbeg{5}{6}
\gvac{1} \got{1}{B} \gnl
\gvac{1} \gcl{2} \gnl
\gbmp{g} \gvac{2} \hspace{-0,34cm} \gbmp{\hspace{0,14cm}g^{\hspace{-0,08cm}-}} \gnl
\gvac{1} \hspace{-0,36cm} \gmu \gcn{1}{1}{2}{2} \gnl
\gvac{2} \hspace{-0,34cm} \gwmu{3} \gnl
\gvac{3} \gob{1}{B}
\gend
\end{eqnarray}  & &
\begin{eqnarray} \eqlabel{spin-up}
\gbeg{2}{5}
\got{1}{B} \gnl
\gcl{1} \gnl
\gbmp{\psi} \gnl
\gcl{1} \gnl
\gob{1}{B}
\gend
=
\gbeg{4}{6}
\gvac{1} \got{1}{B} \gnl
\gwcm{3} \gnl
\gcl{1} \gvac{1} \hspace{-0,34cm} \gcmu \gnl
\gvac{1} \hspace{-0,21cm} \gbmp{\lambda} \gcn{1}{2}{2}{2} \gvac{1} \hspace{-0,34cm}  \gbmp{\hspace{0,14cm}\lambda^{\hspace{-0,08cm}-}} \gnl
\gvac{3} \gob{1}{B}
\gend \vspace{-0,2cm}
\end{eqnarray}
\end{tabular}
\end{center}
\end{defn}

Clearly, a group-like $g$ on $B$ - viewed as an element of $\Coalg(I,B)$ - induces an inner
automorphism $\crta g=\Fi$ in \equref{spin-down},
and a group-like $\lambda$ on $B^*$ - viewed as an element of $\Alg(B,I)$ - induces a co-inner automorphism
$\crta\lambda=\psi$ in \equref{spin-up}.
We have that $\crta\lambda$ commutes with $\crta g$. To prove this we apply the bialgebra property
and use the fact that $g$ and $\lambda$ are (co)algebra morphisms:
$$
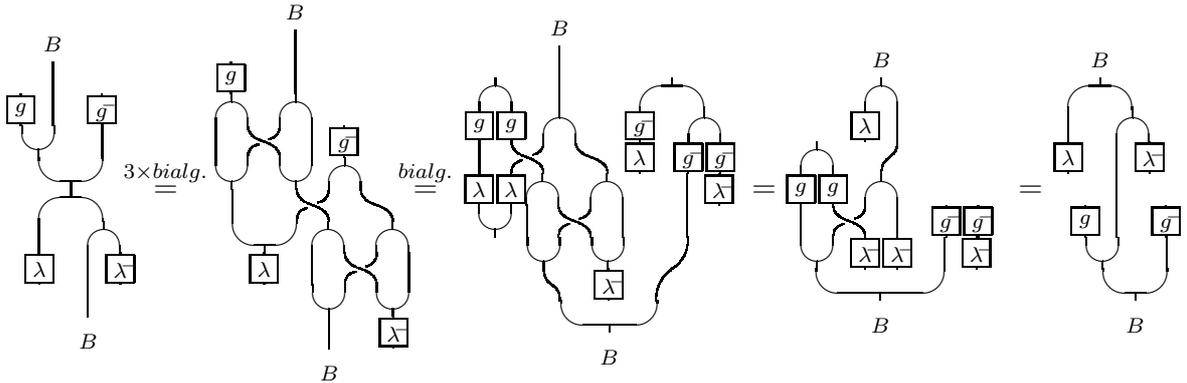
\gbeg{4}{10}
\gvac{1} \got{1}{B} \gnl
\gvac{1} \gcl{2} \gnl
\gbmp{g} \gvac{2} \hspace{-0,34cm} \gbmp{\hspace{0,14cm}g^{\hspace{-0,08cm}-}} \gnl
\gvac{1} \hspace{-0,34cm} \gmu \gcn{1}{1}{2}{2} \gnl
\gvac{2} \hspace{-0,34cm} \gwmu{3} \gnl
\gvac{2} \gwcm{3} \gnl
\gvac{2} \gcl{1} \gvac{1} \hspace{-0,34cm} \gcmu \gnl
\gvac{3} \hspace{-0,21cm} \gbmp{\lambda} \gcn{1}{2}{2}{2} \gvac{1} \hspace{-0,34cm}  \gbmp{\hspace{0,14cm}\lambda^{\hspace{-0,08cm}-}} \gnl
\gvac{5} \gob{1}{B}
\gend \hspace{-0,12cm}{\stackrel{3\times bialg.}{=}}\hspace{0,1cm}
\gbeg{6}{12}
\gvac{2} \got{1}{B} \gnl
\gvac{2} \gcl{2} \gnl
\gbmp{g} \gnl
\hspace{-0,34cm} \gcmu \gcmu \gnl
\gcl{1} \gbr \gcl{1} \gbmp{\hspace{0,14cm}g^{\hspace{-0,08cm}-}} \gnl
\gmu \gmu \hspace{-0,2cm} \gcmu \gnl
\gvac{1} \gcl{1} \gvac{1} \gbr \gcn{1}{1}{1}{3} \gnl
\gvac{1} \gwmu{3} \hspace{-0,22cm} \gcmu \gcmu \gnl
\gvac{3} \hspace{-0,21cm} \gbmp{\lambda} \gcn{1}{1}{2}{2} \gvac{1} \hspace{-0,34cm}  \gbr \gcl{1} \gnl
\gvac{5} \gmu \gmu \gnl
\gvac{5} \gcn{1}{1}{2}{2} \gvac{2} \hspace{-0,34cm} \gbmp{\hspace{0,14cm}\lambda^{\hspace{-0,08cm}-}} \gnl
\gvac{5} \gob{3}{B}
\gend\hspace{-0,12cm}{\stackrel{bialg.}{=}} \hspace{0,14cm}
\gbeg{9}{11}
\gvac{2} \got{2}{B} \gnl
\gvac{2} \gcn{1}{2}{2}{2} \gnl
\gcmu \gvac{3} \gwcm{3} \gnl
\gbmp{g} \gbmp{g} \gcmu \gvac{1} \gbmp{\hspace{0,14cm}g^{\hspace{-0,08cm}-}} \gvac{1} \hspace{-0,34cm} \gcmu \gnl
\gvac{1} \hspace{-0,22cm} \gcl{1} \gbr \gcn{1}{1}{1}{3} \gvac{1} \gbmp{\lambda} 	 \gvac{1} \hspace{-0,34cm}  \gbmp{\hspace{0,14cm}g^{\hspace{-0,08cm}-}} \gbmp{\hspace{0,14cm}g^{\hspace{-0,08cm}-}} \gnl
\gvac{2} \hspace{-0,36cm} \gbmp{\lambda} \gbmp{\lambda} \hspace{-0,22cm} \gcmu \gcmu \gvac{1} \gcl{2} \gbmp{\hspace{0,14cm}\lambda^{\hspace{-0,08cm}-}} \gnl
\gvac{3} \hspace{-0,34cm} \gmu \gcn{1}{1}{0}{0} \hspace{-0,22cm} \gbr \gcl{1} \gnl
\gvac{5} \hspace{-0,12cm} \gmu \gmu \gcn{1}{2}{3}{1} \gnl
\gvac{5} \gcn{1}{1}{2}{3} \gvac{2} \hspace{-0,34cm} \gbmp{\hspace{0,14cm}\lambda^{\hspace{-0,08cm}-}} \gnl
\gvac{7} \hspace{-0,34cm} \gwmu{4} \gnl
\gvac{8} \gob{2}{B}
\gend\hspace{-0,12cm}=
\gbeg{7}{9}
\gvac{2} \got{2}{B} \gnl
\gvac{2} \gcmu \gnl
\gvac{2} \gbmp{\lambda} \gcl{1} \gnl
\gcmu \gvac{1} \gcn{1}{1}{1}{0} \gnl
\gbmp{g} \gbmp{g} \gcmu \gnl
\gcl{1} \gbr \gcl{1} \gvac{1} \hspace{-0,34cm}  \gbmp{\hspace{0,14cm}g^{\hspace{-0,08cm}-}} \gbmp{\hspace{0,14cm}g^{\hspace{-0,08cm}-}} \gnl
\gvac{1} \hspace{-0,36cm} \gmu \gbmp{\hspace{0,14cm}\lambda^{\hspace{-0,08cm}-}} \gbmp{\hspace{0,14cm}\lambda^{\hspace{-0,08cm}-}} \gvac{1} \hspace{-0,36cm} \gcl{1} \gbmp{\hspace{0,14cm}\lambda^{\hspace{-0,08cm}-}} \gnl
\gvac{2} \gwmu{5} \gnl
\gvac{4} \gob{1}{B}
\gend=
\gbeg{4}{9}
\gvac{1} \got{1}{B} \gnl
\gwcm{3} \gnl
\gcl{1} \gvac{1} \hspace{-0,34cm} \gcmu \gnl
\gvac{1} \hspace{-0,21cm} \gbmp{\lambda} \gcn{1}{3}{2}{2} \gvac{1} \hspace{-0,34cm}  \gbmp{\hspace{0,14cm}\lambda^{\hspace{-0,08cm}-}} \gnl
 \gnl
\gvac{2} \gbmp{g} \gvac{2} \hspace{-0,34cm} \gbmp{\hspace{0,14cm}g^{\hspace{-0,08cm}-}} \gnl
\gvac{3} \hspace{-0,34cm} \gmu \gcn{1}{1}{2}{2} \gnl
\gvac{4} \hspace{-0,34cm} \gwmu{3} \gnl
\gvac{5} \gob{1}{B}
\gend
$$
For the third and the fourth equality note that in the third diagram the factors $\lambda g, \lambda g^{-1}:
I\to I$ cancel out and in the fourth diagram so do $\lambda^{-1}g, \lambda^{-1}g^{-1}: I\to I$.

Moreover, the assignment $G(B)=\Coalg(I,B)\to\Aut(\E; B), g\mapsto \crta g$ is a group morphism:
$$\crta{gg'}=
\gbeg{5}{6}
\gvac{2} \got{1}{B} \gnl
\gvac{2} \gcl{2} \gnl
\glmp \gnot{\hspace{-0,24cm}g g'} \grmptb \gvac{1} \glmp \gnot{\hspace{-0,34cm}(g g')^{\hspace{-0,08cm}-}} \grmp \gnl
\gvac{1} \gmu \gcn{1}{1}{2}{2} \gnl
\gvac{2} \hspace{-0,34cm} \gwmu{3} \gnl
\gvac{3} \gob{1}{B}
\gend=
\gbeg{7}{7}
\got{2}{I} \got{1}{B} \got{4}{I} \gnl
\gcmu \gcl{2} \gvac{1} \gcmu \gnl
\gbmp{g} \gbmp{g'} \gvac{1} \glmp \gnot{\hspace{-0,24cm}(g')^{\hspace{-0,08cm}-}} \grmpt  \glmpt \gnot{\hspace{-0,24cm}g^{\hspace{-0,08cm}-}} \grmp \gnl
\gmu \gcn{1}{1}{1}{0} \gvac{1} \gmu \gnl
\gvac{1} \hspace{-0,22cm} \gmu \gvac{1} \gcn{1}{1}{3}{2} \gnl
\gvac{2} \hspace{-0,34cm} \gwmu{4} \gnl
\gvac{3} \gob{2}{B}
\gend=
\gbeg{6}{8}
\gvac{2} \got{1}{B} \gnl
\gvac{2} \gcl{2} \gnl
\gvac{1} \gbmp{g'} \gvac{2} \hspace{-0,34cm} \glmp \gnot{\hspace{-0,24cm}(g')^{\hspace{-0,08cm}-}} \grmpt \gnl
\gvac{1} \hspace{-0,22cm} \gbmp{g} \gvac{1} \hspace{-0,56cm} \gmu \gcn{1}{1}{2}{2} \gnl
\gvac{2} \gcn{1}{1}{1}{2} \gvac{1} \hspace{-0,34cm} \gwmu{3} \gbmp{g^{\hspace{-0,08cm}-}} \gnl
\gvac{3} \gwmu{3} \gcn{1}{1}{3}{1} \gnl
\gvac{4} \gwmu{3} \gnl
\gvac{5} \gob{1}{B}
\gend=\crta g\crta {g'.}
$$

\subsection{Group-likes and the Drinfel'd double}


In \cite[Section 4]{Femic1} we studied the Drinfel'd double $D(B)=(B^{op})^*\bowtie B$ of a finite Hopf
algebra $B\in\E$ when $\Phi_{B,B}$ 
and $\Phi_{B,B^*}$ are symmetric. As a matter of fact we have:

\begin{lma} \cite[Lemma 4.3]{Femic1} \lelabel{symm-conds}
The following conditions are equivalent:

(i) $\Phi_{B,B}$ is symmetric;

(ii) $\Phi_{B^*,B^*}$ is symmetric;

(iii) $\Phi_{B,B^*}$ is symmetric.
\end{lma}
\par\medskip

Until the end of this section we will assume that $\Phi_{B,B}$ is symmetric. 
We recall that 
$(B^{op})^*$ is a left $B$-module and that $B$ is a right $(B^{op})^*$-module via:
\begin{equation} \eqlabel{Dr-actions}
\gbeg{3}{5}
\got{1}{\hspace{-0,2cm}B^{op}} \got{2}{\hspace{-0,4cm}B} \got{1}{\hspace{-0,6cm}(B^{op})^*} \gnl
\gcl{1} \glm \gnl
\gcn{1}{1}{1}{3}\gcn{1}{1}{3}{1} \gnl
\gvac{1}  \gmp{\crta\ev} \gnl
\gob{1}{}
\gend=
\gbeg{3}{8}
\got{2}{B^{op}} \got{1}{\hspace{-0,3cm}B} \got{2}{(B^{op})^*} \gnl
\gcl{1} \gcmu \gcn{1}{5}{2}{2} \gnl
\gmu \gmp{-} \gnl
\gcn{1}{1}{2}{3} \gvac{1} \gcl{1} \gnl
\gvac{1} \gbr \gnl
\gvac{1} \gmu \gnl
\gvac{1} \gcn{1}{1}{2}{4} \gcn{1}{1}{4}{2} \gnl
\gvac{3} \hspace{-0,34cm} \gmp{\crta\ev} \gnl
\gob{1}{}
\gend
\qquad\qquad\textnormal{and}\qquad
\gbeg{3}{6}
\got{1}{B} \got{3}{(B^{op})^*} \gnl
\gcl{1} \gcn{1}{1}{3}{1} \gnl
\grm \gnl
\gcl{2} \gnl
\gob{1}{B}
\gend=
\gbeg{3}{9}
\got{2}{} \got{1}{B} \gvac{1} \got{1}{(B^{op})^*} \gnl
\gvac{1} \gwcm{3}  \gcl{5} \gnl 
\gvac{1} \hspace{-0,34cm} \gcmu \gvac{1} \hspace{-0,2cm} \gmp{-} \gnl
\gvac{2} \hspace{-0,36cm} \gibr \gcn{1}{1}{2}{1} \gnl
\gvac{2} \gcl{4} \gbr \gnl
\gvac{3} \gmu \gnl
\gvac{3} \gcn{1}{1}{2}{4} \gcn{1}{1}{4}{2} \gnl
\gvac{5} \hspace{-0,34cm} \gmp{\crta\ev} \gnl
\gvac{2} \gob{2}{B}
\gend
\end{equation}
Here $\crta\ev: B\ot B^*\to I$ is the evaluation $\crta\ev=\ev_{B^*}\Phi_{B,B^*}$.
The Hopf algebra $D(B)$ is endowed with the codiagonal comultiplication,
componentwise unit $\eta$ and counit $\Epsilon$ (i.e. $\eta_{B^*}\ot\eta_B$ and $\Epsilon_{B^*}\ot\Epsilon_B$
respectively), multiplication \equref{D(H)-mult} and antipode \equref{D(B)antip}:
\vspace{-0,7cm}
\begin{center}
\begin{tabular}{p{6cm}p{3.8cm}p{5cm}}
\begin{eqnarray} \eqlabel{D(H)-mult}
\nabla_{D(B)}=
\gbeg{6}{6}
\got{1}{B^*} \got{2}{B} \got{2}{B^*} \got{1}{B} \gnl
\gcl{1} \gcmu \gcmu \gcl{3} \gnl
\gcl{1} \gcl{1} \gbr \gcl{1} \gnl
\gcl{1} \glm \grm \gnl
\gwmu{3} \gwmu{3} \gnl
\gvac{1} \gob{1}{B^*} \gvac{2} \gob{1}{B}
\gend
=
\gbeg{8}{10}
\got{1}{B^*} \gvac{3} \got{2}{B} \got{1}{B^*} \got{1}{B} \gnl
\gcl{7} \gvac{3} \gcmu \gcl{4} \gcl{7} \gnl
\gvac{1} \gdb \gvac{1} \hspace{-0,42cm} \gcn{1}{1}{3}{2} \gvac{1} \gmp{-} \gnl
\gvac{2} \gibr \gcmu \gcl{1} \gnl
\gvac{2} \gcl{4} \gmu \gbr \gnl
\gvac{3} \gcn{1}{1}{2}{3} \gvac{1} \gcl{1} \gbr \gnl
\gvac{4} \gbr \gcl{1} \gcl{2} \gnl
\gvac{4} \gmu \gcn{1}{1}{1}{0} \gnl
\gvac{1} \gmu \gvac{2} \hspace{-0,22cm} \glmpt \gnot{\hspace{-0,32cm} \crta \ev} \grmptb
	\gvac{1} \hspace{-0,2cm} \gmu \gnl
\gvac{2} \gob{2}{B^*} \gvac{4} \gob{2}{B}
\gend \quad
\end{eqnarray}  & &  
\begin{eqnarray} \eqlabel{D(B)antip}
S_{D(B)}=
\gbeg{4}{9}
\gvac{1} \got{1}{B^*} \got{1}{B} \gnl
\gvac{1} \gbr \gnl
\gvac{1} \gmp{+} \gmp{\hspace{0,01cm}+{\hspace{-0,1cm}\vspace{-1cm}^*}} \gnl
\gvac{1} \gcn{1}{1}{1}{0} \gcn{1}{1}{1}{2} \gnl
\gcmu \gcmu \gnl
\gcl{1} \gbr \gcl{1} \gnl
\glm \grm \gnl
\gvac{1} \gcl{1} \gcl{1} \gnl
\gvac{1} \gob{1}{B^*} \gob{1}{B.}
\gend
\end{eqnarray}
\end{tabular}
\end{center}
In $\nabla_{D(B)}$ note that $(B^{op})^*\iso B^*$ as algebras. From \equref{D(B)antip} it follows:
\begin{equation}\eqlabel{1S}
S_{D(B)}(\eta_{B^*}\ot B)=\eta_{B^*}\ot S_B.
\end{equation}
For a left $D(B)$-module $M$ in $\E$ we will consider the identity \equref{BHmod-conv}. An object
$M\in\E$ is a left $D(B)$-module if and only if it is a left $B$- and a $B^*$-module satisfying
\equref{B tie H -mod} (\cite[Lemma 4.1]{Femic1}).
\begin{center}
\begin{tabular}{p{5.8cm}p{1cm}p{6.8cm}}
\begin{equation} \eqlabel{BHmod-conv}
\gbeg{3}{5}
\got{1}{D(B)} \got{3}{M} \gnl
\gcn{1}{1}{1}{3} \gvac{1} \gcl{1} \gnl
\gvac{1} \glm \gnl
\gvac{2} \gcl{1} \gnl
\gvac{1} \gob{3}{M}
\gend=
\gbeg{3}{5}
\got{1}{B^*} \got{1}{B} \got{1}{M} \gnl
\gcl{1} \glm \gnl
\gcn{1}{1}{1}{3} \gvac{1} \gcl{1} \gnl
\gvac{1} \glm \gnl
\gvac{1} \gob{3}{M}
\gend
\end{equation}  & & \vspace{-1,6cm}
\begin{equation} \eqlabel{B tie H -mod}
\gbeg{5}{6}
\got{1}{B} \got{1}{B^*} \got{3}{M} \gnl
\gcl{1} \gcn{1}{1}{1}{3} \gvac{1} \gcl{1} \gnl
\gcn{1}{2}{1}{5} \gvac{1} \glm \gnl
\gvac{3} \gcl{1} \gnl
\gvac{2} \glm \gnl
\gvac{2} \gob{3}{M}
\gend=
\gbeg{5}{9}
\got{2}{B} \got{2}{B^*} \got{1}{M} \gnl
\gcmu \gcmu \gcl{7} \gnl
\gcl{1} \gbr \gcl{1} \gnl
\glm \grm \gnl
\gvac{1} \gcl{1} \gcn{1}{1}{1}{3} \gnl
\gvac{1} \gcn{1}{2}{1}{5} \gvac{1} \glm \gnl
\gvac{4} \gcl{1} \gnl
\gvac{3} \glm \gnl
\gvac{3} \gob{3}{M}
\gend
\end{equation}
\end{tabular}
\end{center} \vspace{-0,6cm}
The actions \equref{Dr-actions} induce trivial actions on the group-likes on $(B^{op})^*$ and $B$, as
we show next. Let $g\in\Coalg(I,B), \lambda^*\in\Coalg(I,B^*)$ and let $\mu$ and $\nu$ denote the two
actions depicted in the diagrams \equref{Dr-actions}, respectively. We identify $g\tr\lambda^*=
(I\stackrel{g\ot\lambda^*}{\to}B\ot B^*\stackrel{\mu}{\to}B^*)$
and $g\tl\lambda^*=(I\stackrel{g\ot\lambda^*}{\to}B\ot B^*\stackrel{\nu}{\to}B)$.
Then $\tr$ and $\tl$ define a left $G(B)$-action on $G(B^*)$ and a right $G(B^*)$-action on $G(B)$,
respectively. Moreover, we have:
$$g\tr\lambda^*=
\gbeg{3}{6}
\got{1}{\hspace{-0,2cm}B^{op}} \gnl
\gcl{1} \gbmp{g} \gbmp{\lambda^*} \gnl
\gcl{1} \glm \gnl
\gcn{1}{1}{1}{3}\gcn{1}{1}{3}{1} \gnl
\gvac{1}  \gmp{\crta\ev} \gnl
\gob{1}{}
\gend=
\gbeg{4}{9}
\got{1}{B^{op}}  \gnl
\gcn{1}{2}{1}{1} \gvac{1} \hspace{-0,34cm} \gbmp{g} \gvac{1} \gbmp{\lambda^*} \gnl
\gvac{2} \hspace{-0,34cm} \gcmu \gcn{1}{5}{2}{2} \gnl
\gvac{1} \gmu \gmp{-} \gnl
\gcn{1}{1}{4}{5} \gcn{1}{1}{5}{5} \gnl
\gvac{2} \gbr \gnl
\gvac{2} \gmu \gnl
\gvac{2} \gcn{1}{1}{2}{4} \gcn{1}{1}{4}{2} \gnl
\gvac{4} \hspace{-0,34cm} \gmp{\crta\ev} \gnl
\gob{1}{}
\gend=
\gbeg{3}{9}
\got{1}{B^{op}} \gnl
\gcn{1}{2}{1}{1} \gvac{1} \hspace{-0,34cm} \gbmp{g}  \gnl
\gvac{2} \hspace{-0,34cm} \gcmu \gnl
\gvac{1} \gmu \gmp{-} \gnl
\gcn{1}{1}{4}{5} \gcn{1}{1}{5}{5} \gnl
\gvac{2} \gbr \gnl
\gvac{2} \gmu \gnl
\gvac{3} \hspace{-0,22cm} \gbmp{\lambda} \gnl
\gob{7}{}
\gend\stackrel{\lambda\in\Alg(B,I)}{=}
\gbeg{3}{9}
\got{1}{B^{op}} \gnl
\gcn{1}{2}{1}{1} \gvac{1} \hspace{-0,34cm} \gbmp{g} \gnl
\gvac{2} \hspace{-0,34cm} \gcmu \gnl
\gvac{1} \gbmp{\lambda} \gbmp{\lambda} \gcl{1} \gnl
\gvac{1} \gmu \gmp{-} \gnl
\gcn{1}{1}{4}{5} \gvac{2} \gbmp{\lambda} \gnl
\gvac{2} \gbr \gnl
\gvac{2} \gmu \gnl
\gob{6}{I}
\gend\stackrel{ass.}{\stackrel{I comm.}{=}}
\gbeg{3}{9}
\got{1}{B^{op}} \gnl
\gcn{1}{2}{1}{1} \gvac{1} \hspace{-0,34cm} \gbmp{g} \gnl
\gvac{2} \hspace{-0,34cm} \gcmu \gnl
\gvac{1} \gbmp{\lambda} \gcl{1}  \gmp{-} \gnl
\gvac{1} \gcl{2} \gbmp{\lambda} \gbmp{\lambda} \gnl
\gvac{2} \gibr \gnl
\gcn{1}{1}{3}{4} \gvac{1} \gmu \gnl
\gvac{2} \hspace{-0,22cm} \gmu \gnl
\gob{6}{I}
\gend=
\gbeg{2}{5}
\got{1}{B^{op}} \gnl
\gcl{1} \gnl
\gbmp{\lambda} \gbmp{g} \gnl
\gvac{1} \gcu{1} \gnl
\gob{1}{}
\gend=\lambda
$$

$$g\tl\lambda^*=
\gbeg{2}{6}
\got{1}{} \gnl
\gbmp{g} \gbmp{\lambda^*} \gnl
\grm \gnl
\gcl{2} \gnl
\gob{1}{B}
\gend=
\gbeg{4}{10}
\got{1}{} \gnl
\gvac{1} \gbmp{g} \gvac{1} \gbmp{\lambda^*} \gnl
\gwcm{3}  \gcl{5} \gnl 
\hspace{-0,34cm} \gcmu \gvac{1} \hspace{-0,2cm} \gmp{-} \gnl
\gvac{1} \hspace{-0,36cm} \gibr \gcn{1}{1}{2}{1} \gnl
\gvac{1} \gcl{4} \gbr \gnl
\gvac{2} \gmu \gnl
\gvac{2} \gcn{1}{1}{2}{4} \gcn{1}{1}{4}{2} \gnl
\gvac{4} \hspace{-0,34cm} \gmp{\crta\ev} \gnl
\gvac{1} \gob{2}{B}
\gend\stackrel{g\in\Coalg(I,B)}{=}
\gbeg{4}{10}
\got{2}{} \got{1}{I} \gnl
\gvac{1} \gwcm{3} \gnl
\gvac{1} \hspace{-0,34cm} \gcmu \gvac{1} \hspace{-0,2cm} \gbmp{g} \gnl
\gvac{2} \hspace{-0,34cm} \gibr \gvac{1} \hspace{-0,22cm} \gmp{-} \gnl
\gvac{3} \hspace{-0,34cm} \gbmp{g} \gbmp{g} \gcn{1}{1}{2}{1} \gnl
\gvac{3} \gcl{4} \gbr \gnl
\gvac{4} \gmu \gnl
\gvac{5} \hspace{-0,22cm} \gbmp{\lambda} \gnl
\gvac{3} \gob{2}{B}
\gend\stackrel{coass.}{\stackrel{I coc.}{=}}
\gbeg{4}{9}
\got{1}{} \got{1}{I} \gnl
\gwcm{3} \gnl
\gcl{1} \gvac{1} \hspace{-0,34cm} \gcmu \gnl
\gvac{1} \hspace{-0,22cm} \gbmp{g} \gvac{1} \hspace{-0,34cm} \gbmp{g} \gbmp{g} \gnl
\gvac{1} \gcn{1}{4}{2}{2} \gvac{1} \gcl{1} \gmp{-} \gnl
\gvac{3} \gibr \gnl
\gvac{3} \gmu \gnl
\gvac{4} \hspace{-0,22cm} \gbmp{\lambda} \gnl
\gvac{2} \gob{1}{B}
\gend=
\gbeg{2}{5}
\got{1}{} \gnl
\gvac{1} \gu{1} \gnl
\gbmp{g} \gbmp{\lambda} \gnl
\gcl{1} \gnl
\gob{1}{B}
\gend=g
$$
With these actions we regard the bicrossproduct of the groups $G(B^*)$ and $G(B)$.

\begin{prop} \prlabel{gr x gr}
Let $B\in\E$ be a finite Hopf algebra and let $D(B)$ denote its Drinfel'd double in $\E$. There
is a group isomorphism:
$$G(B^*)\times G(B) \iso G(D(B)).$$
\end{prop}

\begin{proof}
First of all, 
note that there is a set-theoretic bijection $\Coalg(I, B^*\ot B)
\iso \Coalg(I, B^*)\times \Coalg(I,B)$.
We define $\Gamma: \Coalg(I, B^*)\times \Coalg(I,B) \to \Coalg(I,D(B))$ by:
$$\Gamma((\lambda^*, g))=
\gbeg{3}{5}
\got{3}{I} \gnl
\gwcm{3} \gnl
\gbmp{\lambda^*} \gvac{1} \gbmp{g} \gnl
\gcl{1} \gvac{1} \gcl{1} \gnl
\gob{1}{B^*} \gob{3}{B}
\gend
$$
for $(\lambda^*,g)\in\Coalg(I, B^*)\times \Coalg(I,B)$. To prove the compatibility of $\Gamma$
with the product take $\kappa^*,\lambda^*\in \Coalg(I, B^*)$ and $g,h\in\Coalg(I, B)$. We compute:
$$\Gamma((\kappa^*, g))*\Gamma((\lambda^*, h))=
\gbeg{6}{5}
\got{7}{I} \gnl
\gvac{2} \gwcm{3} \gnl
\glmp \gcmp \gnot{\hspace{-0,88cm}\Gamma(\q(\kappa^*, g)\q)} \grmptb  \glmp \gcmpt \gnot{\hspace{-0,86cm}\Gamma(\q(\lambda^*, h)\q)} \grmp  \gnl
\gvac{2} \gwmu{3} \gnl
\gob{7}{D(B)}
\gend=
\gbeg{6}{10}
\got{6}{I} \gnl
\gvac{1} \gwcm{4} \gnl
\gwcm{3} \gwcm{3} \gnl
\gbmp{\kappa^*} \gvac{1} \gbmp{g} \gbmp{\lambda^*} \gvac{1} \gbmp{h} \gnl
\gcl{1} \gcn{1}{1}{3}{2} \gcn{1}{1}{3}{4} \gvac{2} \gcl{4} \gnl
\gcl{1} \gcmu \gcmu \gnl
\gcl{1} \gcl{1} \gbr \gcl{1} \gnl
\gcl{1} \glm \grm \gnl
\gwmu{3} \gwmu{3} \gnl
\gvac{1} \gob{1}{B^*} \gvac{2} \gob{1}{B}
\gend=
\gbeg{7}{9}
\got{7}{I} \gnl
\gvac{1} \gwcm{5} \gnl
\gwcm{3} \gvac{1} \gwcm{3} \gnl
\gcl{1} \gvac{1} \hspace{-0,34cm} \gcmu  \gcmu \gvac{1} \gcn{1}{1}{0}{0} \gnl
\gvac{1} \hspace{-0,34cm} \gbmp{\kappa^*} \gvac{1} \hspace{-0,34cm} \gbmp{g} \gbr \gbmp{\lambda^*} \gvac{1} \hspace{-0,34cm} \gbmp{h} \gnl
\gvac{2} \gcl{1} \gcn{1}{1}{2}{2} \gvac{1} \hspace{-0,34cm} \gbmp{\lambda^*} \gbmp{g} \gcl{1} \gcn{1}{1}{2}{2} \gnl
\gvac{2} \gcn{1}{1}{2}{3} \gvac{1}  \glm \grm \gcn{1}{1}{2}{1} \gnl
\gvac{3} \gwmu{3} \gwmu{3} \gnl
\gvac{4} \gob{1}{B^*} \gvac{2} \gob{1}{B}
\gend
$$

$$
\stackrel{coass.}{\stackrel{I coc.}{=}}
\gbeg{7}{8}
\got{7}{I} \gnl
\gvac{1} \gwcm{5} \gnl
\gwcm{3} \gvac{1} \gwcm{3} \gnl
\gcl{1} \gvac{1} \hspace{-0,34cm} \gcmu  \gcmu \gvac{1} \gcn{1}{1}{0}{0} \gnl
\gvac{1} \hspace{-0,34cm} \gbmp{\kappa^*} \gvac{1} \hspace{-0,34cm} \gbmp{g} \gbmp{\lambda^*} \gbmp{g}  \gbmp{\lambda^*} \gvac{1} \hspace{-0,34cm} \gbmp{h} \gnl 
\gvac{2} \gcn{1}{1}{1}{2} \gvac{1} \hspace{-0,34cm} \glm \grm \gcn{1}{1}{2}{1} \gnl
\gvac{3} \gwmu{3} \gwmu{3} \gnl
\gvac{4} \gob{1}{B^*} \gvac{2} \gob{1}{B}
\gend=
\gbeg{7}{6}
\got{7}{I} \gnl
\gvac{1} \gwcm{5} \gnl
\gwcm{3} \gvac{1} \gwcm{3} \gnl
\gbmp{\kappa^*} \gvac{1} \hspace{-0,34cm} \glmp \gnot{\hspace{-0,4cm}g\tr\lambda^*} \grmp \glmp \gnot{\hspace{-0,4cm}g\tl\lambda^*} \grmp  \gvac{1} \hspace{-0,22cm} \gbmp{h} \gnl
\gvac{1} \gwmu{3} \gvac{1} \gwmu{3} \gnl
\gvac{2} \gob{1}{B^*} \gvac{3} \gob{1}{B}
\gend=\Gamma((\kappa^* * (g\tr\lambda^*)),((g\tl \lambda^*)* h))=\Gamma((\kappa^*,g)\cdot (\lambda^*, h)).
$$
\qed\end{proof}

Now we define the map
\begin{equation} \eqlabel{map Teta}
\Theta: G(D(B))\stackrel{\Gamma^{-1}}{\to} G(B^*)\times G(B) \to \Aut(\E;  B)
\end{equation}
by $\Theta((\lambda,g))=\crta g \crta{\lambda}$. It is a group map, because $\crta g$ and
$\crta{\lambda}$ commute, the assignment $g\mapsto\crta g$ is a group one for all $g\in G(B),
\lambda\in G(B^*)$, and the actions $\tr$ and $\tl$ are trivial.
\par\bigskip

Consider the group $S(B)=\{ (\lambda, g)\in G(D(B)) \hspace{0,2cm} \vert \hspace{0,2cm} \crta{g}
=\crta\lambda\}$. 
In the following lemma we will identify $D(B)^*=((B^{op})^*\bowtie B)^*\cong B^{op}\bowtie B^*$,
and henceforth we will consider $G(D(B)^*)=G(B)\times G(B^*)$, that is: $\Alg(D(B), I)=
\Alg(B^*, I)\times\Alg(B,I)$ as sets.

\begin{lma} \lelabel{equiv cond gr}
For $g\in G(B), \lambda\in G(B^*)$ the following are equivalent: \\
\noindent (i) $g\bowtie\lambda\in G(D(B)^*)$; \vspace{-0,24cm}
\begin{center} \hspace{-0,4cm}
\begin{tabular}{p{4.5cm}p{1.6cm}p{4.8cm}}
(ii) 
$
\gbeg{4}{6}
\got{2}{B} \got{2}{B^*} \gnl
\gcmu \gcmu  \gnl
\gcl{1} \gbr \gcl{1} \gnl
\glm \grm \gnl
\gvac{1} \gbmp{g^*} \gbmp{\lambda} \gnl
\gvac{1} \gob{1}{}
\gend=
\gbeg{3}{4}
\got{1}{B} \got{1}{B^*} \gnl
\gcl{1} \gcl{1} \gnl
\gbmp{\lambda} \gbmp{g^*} \gnl
\gob{1}{}
\gend$;  & &
(iii) 
$
\gbeg{3}{5}
\got{2}{B} \gnl
\gcmu \gnl
\gbmp{\lambda} \gcl{1} \gbmp{g} \gnl
\gvac{1} \gmu \gnl
\gob{4}{B}
\gend=
\gbeg{3}{5}
\got{4}{B} \gnl
\gvac{1} \gcmu \gnl
\gbmp{g} \gcl{1} \gbmp{\lambda} \gnl
\gmu \gnl
\gob{2}{B} 
\gend \quad;$
\end{tabular}
\end{center} \vspace{-0,2cm}
(iv) $(\lambda, g)\in S(B)$.

Consequently, $S(B)\iso G(D(B)^*)$.
\end{lma}

\begin{proof}
We consider $g^*\in\Alg(B^*,I)$ and $\lambda\in\Alg(B, I)$. Then clearly (i) is equivalent to $\star$ in:
$$
\gbeg{4}{6}
\got{2}{B} \got{2}{B^*} \gnl
\gcmu \gcmu  \gnl
\gcl{1} \gbr \gcl{1} \gnl
\glm \grm \gnl
\gvac{1} \gbmp{g^*} \gbmp{\lambda} \gnl
\gvac{1} \gob{1}{}
\gend=
\gbeg{4}{7}
\gvac{1} \got{1}{B} \got{1}{B^*} \gnl
\gu{1} \gcl{1} \gcl{1} \gu{1} \gnl
\glmptb \gnot{\hspace{-0,4cm} D(B)} \grmpt \glmpt \gnot{\hspace{-0,4cm} D(B)} \grmptb  \gnl
\gwmu{4} \gnl
\gvac{1} \glmpb \gnot{\hspace{-0,4cm} D(B)} \grmpb \gnl
\gvac{1} \gbmp{g^*} \gbmp{\lambda} \gnl
\gvac{1} \gob{1}{}
\gend\stackrel{\star}{=}
\gbeg{4}{5}
\gvac{1} \got{1}{B} \got{1}{B^*} \gnl
\gu{1} \gcl{1} \gcl{1} \gu{1} \gnl
\glmptb \gnot{\hspace{-0,4cm} g^*\bowtie\lambda} \grmpt \glmpt \gnot{\hspace{-0,4cm} g^*\bowtie\lambda} \grmptb  \gnl
\gwmu{4} \gnl
\gvac{1} \gob{2}{I}
\gend=
\gbeg{4}{4}
\gvac{1} \got{1}{B} \got{1}{B^*} \gnl
\gu{1} \gcl{1} \gcl{1} \gu{1} \gnl
\gbmp{g^*} \gbmp{\lambda} \gbmp{g^*} \gbmp{\lambda} \gnl
\gob{1}{}
\gend=
\gbeg{3}{4}
\got{1}{B} \got{1}{B^*} \gnl
\gcl{1} \gcl{1} \gnl
\gbmp{\lambda} \gbmp{g^*} \gnl
\gob{1}{}
\gend
$$
By \equref{D(H)-mult}, (ii) is equivalent to $\star$ in:
$$
\gbeg{6}{11}
\got{1}{} \gvac{1} \got{2}{B} \got{1}{B^*} \gnl
\gvac{2} \gcmu \gcl{4} \gnl
\gvac{1} \gcn{1}{1}{3}{2} \gvac{1} \gmp{-} \gnl
\gbmp{g} \gcmu \gcl{1} \gnl
\gmu \gbr \gnl
\gcn{1}{1}{2}{3} \gvac{1} \gcl{1} \gbr \gnl
\gvac{1} \gbr \gcl{1} \gbmp{\lambda} \gnl
\gvac{1} \gmu \gcn{1}{1}{1}{0} \gnl
\gvac{2} \hspace{-0,34cm} \gbr \gnl
\gvac{2} \gev \gnl
\gvac{4} \gob{2}{}
\gend=
\gbeg{8}{11}
\got{1}{} \gvac{3} \got{2}{B} \got{1}{B^*} \gnl
\gvac{4} \gcmu \gcl{4} \gnl
\gvac{1} \gdb \gvac{1} \hspace{-0,42cm} \gcn{1}{1}{3}{2} \gvac{1} \gmp{-} \gnl
\gvac{1} \gbmp{g} \gibr \gcmu \gcl{1} \gnl
\gvac{1} \gbr \gmu \gbr \gnl
\gvac{1} \gev \gcn{1}{1}{2}{3} \gvac{1} \gcl{1} \gbr \gnl
\gvac{4} \gbr \gcl{1} \gbmp{\lambda} \gnl
\gvac{4} \gmu \gcn{1}{1}{1}{0} \gnl
\gvac{5} \hspace{-0,34cm} \gbr \gnl
\gvac{5} \gev \gnl
\gvac{7} \gob{2}{}
\gend\stackrel{\equref{D(H)-mult}}{=}
\gbeg{5}{7}
\got{1}{} \got{2}{B} \got{2}{B^*} \gnl
\gvac{1} \gcmu \gcmu \gnl
\gbmp{g} \gcl{1} \gbr \gcl{1} \gnl
\gcn{1}{1}{1}{3} \glm \grm \gnl
\gvac{1} \gbr \gbmp{\lambda} \gnl
\gvac{1} \gev \gnl
\gvac{3} \gob{1}{}
\gend\stackrel{\star}{=}
\gbeg{3}{5}
\got{1}{B} \got{1}{} \got{1}{B^*} \gnl
\gcl{1} \gbmp{g} \gcl{1} \gnl
\gbmp{\lambda} \gbr \gnl
\gvac{1} \gev \gnl
\gob{1}{}
\gend
$$
By the universal property of $ev_{B^*}$ this is equivalent to:
$$
\gbeg{5}{9}
\got{1}{I} \gvac{1} \got{2}{B} \gnl
\gcl{2} \gvac{1} \gcmu \gnl
\gvac{1} \gcn{1}{1}{3}{2} \gvac{1} \gmp{-} \gnl
\gbmp{g} \gcmu \gcl{1} \gnl
\gmu \gbr \gnl
\gcn{1}{1}{2}{3} \gvac{1} \gcl{1} \gbmp{\lambda} \gnl
\gvac{1} \gbr \gnl
\gvac{1} \gmu \gnl
\gvac{2} \gnl
\gvac{1} \gob{2}{B} \gob{1}{}
\gend=
\gbeg{3}{4}
\got{1}{I} \got{1}{B} \gnl
\gbmp{g} \gcl{1} \gnl
\gcl{1} \gbmp{\lambda} \gnl
\gob{1}{B} \gob{1}{}
\gend
$$
Composing to this: $(B\ot -)(\Phi_{I,B}\ot B)\Phi_{B,B}\Delta_B$ from above and then $\nabla_B$
from below -- which is an invertible process -- we get that it is equivalent to \equref{racun3},
which by the (co)associativity and the antipode rule (use also that $\Phi_{B,B}$ is symmetric)
is further equivalent to \equref{racun4}: \vspace{-0,4cm}
\begin{center}
\begin{tabular}{p{5.8cm}p{2cm}p{5.8cm}}
\begin{equation} \eqlabel{racun3}
\gbeg{5}{13}
\got{1}{I} \gvac{1} \got{2}{B} \gnl
\gcl{3} \gvac{1} \gcmu \gnl
\gvac{2} \gbr \gnl
\gvac{1} \gcn{1}{1}{3}{1} \gcn{1}{1}{3}{4} \gnl
\gbr \gvac{1} \gcmu \gnl
\gcl{5} \gcl{1} \gcn{1}{1}{3}{2} \gvac{1} \gmp{-} \gnl
\gvac{1} \gbmp{g} \gcmu \gcl{1} \gnl
\gvac{1} \gmu \gbr \gnl
\gvac{1} \gcn{1}{1}{2}{3} \gvac{1} \gcl{1} \gbmp{\lambda} \gnl
\gvac{2} \gbr \gnl
\gcn{1}{1}{1}{2} \gvac{1} \gmu \gnl
\gvac{1} \hspace{-0,22cm} \gwmu{3} \gnl
\gvac{2} \gob{1}{B} \gob{4}{}
\gend=
\gbeg{3}{7}
\got{1}{I} \got{2}{B} \gnl
\gcl{2} \gcmu \gnl
\gvac{1} \gbr \gnl
\gbr \gbmp{\lambda} \gnl
\gcl{1} \gbmp{g} \gnl
\gmu \gnl
\gob{2}{B} \gob{1}{}
\gend 
\end{equation}  & & \vspace{1,5cm}
\begin{equation} \eqlabel{racun4} 
\gbeg{3}{5}
\got{4}{B} \gnl
\gvac{1} \gcmu \gnl
\gbmp{g} \gcl{1} \gbmp{\lambda} \gnl
\gmu \gnl
\gob{2}{B}
\gend=
\gbeg{4}{5}
\got{2}{B} \gnl
\gcmu \gnl
\gbmp{\lambda} \gcl{1} \gbmp{g} \gnl
\gvac{1} \gmu \gnl
\gob{4}{B}
\gend \vspace{-0,24cm}
\end{equation}
\end{tabular}
\end{center}
``Multiply'' \equref{racun4} from the left by $g^{-1}$, that is, compose it from below by
$\nabla_B(g^{-1}\ot -)$ -- it is equivalent to:
$$
\gbeg{3}{5}
\got{2}{B} \gnl
\gcmu \gnl
\gcl{2} \gbmp{\lambda} \gnl
 \gnl
\gob{1}{B}
\gend=
\gbeg{5}{8}
\got{3}{B} \gnl
\gvac{1} \hspace{-0,34cm} \gcmu \gnl
\gvac{1} \gbmp{\lambda} \gcl{3} \gnl
 \gnl
\gvac{1} \hspace{-0,34cm} \gbmp{\hspace{0,14cm}g^{\hspace{-0,08cm}-}} \gvac{2} \hspace{-0,34cm} \gbmp{g} \gnl
\gvac{1} \gcn{1}{1}{2}{2} \gvac{1} \gmu \gnl
\gvac{2} \hspace{-0,2cm} \gwmu{3} \gnl
\gvac{3} \gob{1}{B}
\gend
$$
Finally, ``multiply'' this by $\lambda^{-1}$ from the left, i.e. compose it from above with
$(\lambda^{-1}\ot -)\Delta_B$ and observe that it is equivalent to $\crta\lambda=\crta g$.
\qed\end{proof}

\begin{rem}
We defined the morphism $g\mapsto\crta g$ in \equref{spin-down}, and consequently the subgroup $S(B)$ of
$G(D(B))$, differently than in \cite{VZ4}. 
Our version of the map is a group one (and not an anti-group map) and consequently so is
$\Theta$ in \equref{map Teta}.
\end{rem}

\section{The quantum automorphism group and the quantum Brauer group}

In this section we will define a group morphism $\Pi: \Aut(\C; B) \to \BQ(\C; B)$. It will send an
automorphism $\alpha$ of $B$ in $\C$ to the class of $\End(B_{\alpha^{-1}})$ in the quantum Brauer
group of $B$ in $\C$. Here $B_{\alpha}=B$ as an object in $\C$ 
with the $B$-module and comodule structures twisted by $\alpha$.
As a matter of fact, we will prove 
that $[B_{\alpha}, B_{\alpha}]$ is an algebra in ${}_B^B\YD(\E)$ for any closed braided monoidal category $\E$,
a faithfully projective Hopf algebra $B\in\E$ and its arbitrary automorphism $\alpha$. 
The great deal of the construction of the map $\Pi$ we will perform in $\E$ making use of braided diagrams.
We will suppose that $\Phi_{B,B}$ is symmetric.

Using the latter symmetricity assumption, one may prove that \equref{alfa} defines an action
on $B_{\alpha}$. It is easily seen that \equref{co-alfa} makes $B_{\alpha}$ a left $B$-comodule. \vspace{-0,8cm}
\begin{center}
\begin{tabular}{p{4.8cm}p{2cm}p{5.8cm}}
\begin{eqnarray} \eqlabel{alfa}
\scalebox{0.9}[0.9]{
\gbeg{2}{5}
\got{1}{B} \got{1}{\hspace{0,1cm}B_{\alpha}} \gnl
\gcl{1} \gcl{1} \gnl
\glm \gnl
\gvac{1} \gcl{1} \gnl
\gvac{1} \gob{1}{B_{\alpha}}
\gend} = \scalebox{0.9}[0.9]{
\gbeg{2}{7}
\got{2}{B} \got{1}{B} \gnl
\gcmu \gcl{2} \gnl
\gmp{-} \gbmp{\alpha} \gnl
\gcn{1}{1}{1}{2} \gmu \gnl
\gvac{1} \hspace{-0,22cm} \gbr \gnl
\gvac{1} \gmu \gnl
\gvac{1} \gob{2}{B}
\gend}
\end{eqnarray}  & &  \vspace{0,05cm}
\begin{eqnarray} \eqlabel{co-alfa}
\scalebox{0.9}[0.9]{
\gbeg{2}{5}
\got{1}{} \got{1}{\hspace{0,1cm}B_{\alpha}} \gnl
\gvac{1} \gcl{1} \gnl
\glcm \gnl
\gcl{1} \gcl{1} \gnl
\gob{1}{B} \gob{1}{B_{\alpha}}
\gend} =
\scalebox{0.9}[0.9]{
\gbeg{2}{5}
\got{2}{B} \gnl
\gcmu \gnl
\gibr \gnl
\gmp{+} \gcl{1} \gnl
\gob{1}{B} \gob{1}{B}
\gend}
\end{eqnarray}
\end{tabular}
\end{center} \vspace{-0,5cm}
Equipped like this, $B_{\alpha}$ is almost a Yetter-Drinfel'd module. Namely, using that $\Phi_{B,B}$
is symmetric, one proves that $B_{\alpha}$ satifies \equref{alfa-YD}.
A similar identity we have for $B_{\alpha}^*$, \equref{alfa*-YD}. The proof of the latter is somewhat
involved, one inserts twice the antipode identity at appropriate places, and uses the fact that
$\Phi_{B,B^*}$ is symmetric. 
\vspace{-0,9cm} 
\begin{center}
\begin{tabular}{p{6.2cm}p{2cm}p{6.2cm}}
\begin{equation} \eqlabel{alfa-YD}
\gbeg{4}{8}
\got{1}{B} \got{5}{B_{\alpha}} \gnl
\gcn{2}{2}{1}{5} \gvac{1} \gcl{2} \gnl \gnl
\gvac{2} \glm \gnl
\gvac{2} \glcm \gnl
\gcn{2}{2}{5}{1} \gvac{1} \gcl{2} \gnl \gnl
\gob{1}{B} \gob{5}{B_{\alpha}}
\gend=
\gbeg{6}{10}
\gvac{1} \got{2}{B} \got{4}{B_{\alpha}} \gnl
\gvac{1} \gcmu \gcn{1}{1}{4}{4} \gnl
\gvac{1} \gcn{1}{1}{1}{0} \hspace{-0,21cm} \gcmu \glcm \gnl
\gvac{1} \gcl{4} \gbr \gcl{2} \gcl{4} \gnl
\gvac{2} \gbmp{\alpha} \gcl{1} \gnl
\gvac{2}  \gmp{+} \gbr \gnl
\gvac{1} \gvac{1} \gbr \gcl{1} \gnl
\gvac{1} \gcn{1}{1}{1}{2} \gmu \glm \gnl
\gvac{2} \hspace{-0,21cm} \gmu \gcn{1}{1}{4}{4} \gnl
\gvac{2} \gob{2}{B} \gob{4}{B_{\alpha}}
\gend
\end{equation}  & & \vspace{0,01cm}
\begin{equation} \eqlabel{alfa*-YD}
\gbeg{4}{8}
\got{1}{B} \got{5}{B_{\alpha}^*} \gnl
\gcn{2}{2}{1}{5} \gvac{1} \gcl{2} \gnl \gnl
\gvac{2} \glm \gnl
\gvac{2} \glcm \gnl
\gcn{2}{2}{5}{1} \gvac{1} \gcl{2} \gnl \gnl
\gob{1}{B} \gob{5}{B_{\alpha}^*}
\gend=
\gbeg{6}{9}
\gvac{1} \got{2}{B} \got{4}{B_{\alpha}^*} \gnl
\gvac{1} \gcmu \gcn{1}{1}{4}{4} \gnl
\gvac{1} \gcn{1}{1}{1}{0} \hspace{-0,21cm} \gcmu \glcm \gnl
\gvac{1} \gbmp{\alpha} \gbr \gcl{1} \gcl{3} \gnl
\gvac{1} \gcl{2} \gmp{+} \gbr \gnl
\gvac{1} \gvac{1} \gbr \gcl{1} \gnl
\gvac{1} \gcn{1}{1}{1}{2} \gmu \glm \gnl
\gvac{2} \hspace{-0,21cm} \gmu \gcn{1}{1}{4}{4} \gnl
\gvac{2} \gob{2}{B} \gob{4}{B_{\alpha}^*}
\gend
\end{equation}
\end{tabular}
\end{center} \vspace{-0,6cm}
Even though $B_{\alpha}$ is not a Yetter-Drinfel'd module, its inner hom-object is. One proves this
identifying $[B_{\alpha}, B_{\alpha}]\iso B_{\alpha}\ot B_{\alpha}^*$ and applying \equref{alfa-YD} and \equref{alfa*-YD}. 
Thus $[B_{\alpha}, B_{\alpha}]$ is an algebra in ${}_B^B\YD(\E)$.
\par\medskip

\subsection{Two Azumaya algebras in $\E$} \sslabel{ordinary Azumaya}

In this subsection we will prove that $[B_{\alpha}, \q B_{\alpha}] \crta\ot \crta{[B_{\alpha}, \q B_{\alpha}]}$
and $\crta{[B_{\alpha}, \q B_{\alpha}]} \crta\ot [B_{\alpha}, \q B_{\alpha}]$ are Azumaya algebras in $\E$.
This will be used in the construction of the announced group map.

Let $B_{\alpha}'=B_{\alpha}$ as a left $B$-module in $\E$ with a left $B$-comodule structure given by
\begin{eqnarray} \eqlabel{co'-alfa}
\scalebox{0.9}[0.9]{
\gbeg{2}{5}
\got{1}{} \got{1}{\hspace{0,1cm}B_{\alpha}'} \gnl
\gvac{1} \gcl{1} \gnl
\glcm \gnl
\gcl{1} \gcl{1} \gnl
\gob{1}{B} \gob{1}{B_{\alpha}'}
\gend} =
\scalebox{0.9}[0.9]{
\gbeg{2}{5}
\got{1}{} \got{1}{\hspace{0,1cm}B_{\alpha}} \gnl
\gvac{1} \gcl{1} \gnl
\glcm \gnl
\gbmp{\hspace{0,14cm}\alpha^{\hspace{-0,08cm}-}} \gcl{1} \gnl
\gob{1}{B} \gob{1}{B_{\alpha}}
\gend}
\scalebox{0.9}[0.9]{
\gbeg{2}{6}
\got{2}{B} \gnl
\gcmu \gnl
\gibr \gnl
\gmp{S} \gcl{2} \gnl
\gbmp{\hspace{0,14cm}\alpha^{\hspace{-0,08cm}-}} \gnl
\gob{1}{B} \gob{1}{B.}
\gend}
\end{eqnarray}
Here $\alpha^{\hspace{-0,08cm}-}$ stands for $\alpha^{-1}$. Note that \equref{co'-alfa} and
\equref{alfa-YD} yield \equref{alfa'-YD}. Similarly, from
\equref{co'-alfa} and the proof of \equref{alfa*-YD} one obtains \equref{alfa-YD-'*}.
\begin{center}
\begin{tabular}{p{5.8cm}p{2cm}p{5.8cm}}
\begin{equation} \eqlabel{alfa'-YD}
\gbeg{4}{8}
\got{1}{B} \got{5}{B_{\alpha}'} \gnl
\gcn{2}{2}{1}{5} \gvac{1} \gcl{2} \gnl \gnl
\gvac{2} \glm \gnl
\gvac{2} \glcm \gnl
\gcn{2}{2}{5}{1} \gvac{1} \gcl{2} \gnl \gnl
\gob{1}{B} \gob{5}{B_{\alpha}'}
\gend=
\gbeg{6}{9}
\gvac{1} \got{2}{B} \got{4}{B_{\alpha}'} \gnl
\gvac{1} \gcmu \gcn{1}{1}{4}{4} \gnl
\gvac{1} \gcn{1}{1}{1}{0} \hspace{-0,21cm} \gcmu \glcm \gnl
\gvac{1} \gbmp{\hspace{0,14cm}\alpha^{\hspace{-0,08cm}-}} \gbr \gcl{1} \gcl{3} \gnl
\gvac{1} \gcl{2} \gmp{+} \gbr \gnl
\gvac{1} \gvac{1} \gbr \gcl{1} \gnl
\gvac{1} \gcn{1}{1}{1}{2} \gmu \glm \gnl
\gvac{2} \hspace{-0,21cm} \gmu \gcn{1}{1}{4}{4} \gnl
\gvac{2} \gob{2}{B} \gob{4}{B_{\alpha}'}
\gend
\end{equation}  & &
\begin{equation} \eqlabel{alfa-YD-'*}
\gbeg{4}{8}
\got{1}{B} \got{5}{B_{\alpha}'^*} \gnl
\gcn{2}{2}{1}{5} \gvac{1} \gcl{2} \gnl \gnl
\gvac{2} \glm \gnl
\gvac{2} \glcm \gnl
\gcn{2}{2}{5}{1} \gvac{1} \gcl{2} \gnl \gnl
\gob{1}{B} \gob{5}{B_{\alpha}'^*}
\gend=
\gbeg{6}{10}
\gvac{1} \got{2}{B} \got{4}{B_{\alpha}'^*} \gnl
\gvac{1} \gcmu \gcn{1}{1}{4}{4} \gnl
\gvac{1} \gcn{1}{1}{1}{0} \hspace{-0,21cm} \gcmu \glcm \gnl
\gvac{1} \gcl{4} \gbr \gcl{2} \gcl{4} \gnl
\gvac{2} \gmp{+} \gcl{1} \gnl
\gvac{2}  \gbmp{\alpha^{\hspace{-0,08cm}-}} \gbr \gnl
\gvac{1} \gvac{1} \gbr \gcl{1} \gnl
\gvac{1} \gcn{1}{1}{1}{2} \gmu \glm \gnl
\gvac{2} \hspace{-0,21cm} \gmu \gcn{1}{1}{4}{4} \gnl
\gvac{2} \gob{2}{B} \gob{4}{B_{\alpha}'^*}
\gend
\end{equation}
\end{tabular}
\end{center}
We will prove that
\begin{equation} \eqlabel{chain-iso}
[B_{\alpha}, \q B_{\alpha}] \crta\ot \crta{[B_{\alpha}, \q B_{\alpha}]} \iso [B_{\alpha}\q\ot B_{\alpha}'^*,
B_{\alpha}\q\ot B_{\alpha}'^*]\quad\textnormal{and}\quad
\crta{[B_{\alpha}, \q B_{\alpha}]} \crta\ot [B_{\alpha}, \q B_{\alpha}]
\iso [B_{\alpha}'^*\q\ot B_{\alpha}, B_{\alpha}'^*\q\ot B_{\alpha}]
\end{equation}
as algebras in ${}_B^B\YD(\E)$. To do this we will construct three algebra isomorphisms in ${}_B^B\YD(\E)$.


\begin{lma} \lelabel{first lema chain}
Let $\theta: B\to [B_{\alpha}, B_{\alpha}]$ be the algebra morphism from \equref{teta} and suppose
that $\Phi_{B,B}$ and $\Phi_{B,E}$ are symmetric for all inner-hom objects $E\in {}_B^B\YD(\E)$.
The morphism
$$\zeta: \crta{[B_{\alpha}, B_{\alpha}]} \to [B_{\alpha}', B_{\alpha}']^{op}
\quad\qquad\textnormal{defined via}\qquad\quad
\zeta=\gbeg{4}{7}
\got{1}{} \got{1}{\crta{E_{\alpha}}} \gnl
\glcm \gnl
\gmp{-} \gcl{1} \gnl
\gibr \gnl
\gcl{1} \gbmp{\theta} \gnl
\gmu \gnl
\gob{2}{(E_{\alpha}')^{op}}
\gend
$$
is an algebra isomorphism in ${}_B^B\YD(\E)$. Here $E_{\alpha}=[B_{\alpha}, B_{\alpha}]$ and
$E_{\alpha}'=[B_{\alpha}', B_{\alpha}']$, and $\crta{E_{\alpha}}$ and $(E_{\alpha}')^{op}$
are the corresponding opposite algebras in the categories ${}_B^B\YD(\E)$ and $\E$, respectively.
\end{lma}

\begin{proof}
We only sketch the proof. One proves that $\zeta$ is an $H$-module and $H$-comodule algebra morphism.
Similarly, one proves that the morphism
$$\xi: [B_{\alpha}', B_{\alpha}']^{op} \to \crta{[B_{\alpha}, B_{\alpha}]}
\quad\qquad\textnormal{defined by}\qquad\quad
\xi=\gbeg{4}{6}
\got{1}{} \got{1}{(E_{\alpha}')^{op}} \gnl
\glcm \gnl
\gibr \gnl
\gcl{1} \gbmp{\theta} \gnl
\gmu \gnl
\gob{2}{\crta{E_{\alpha}}}
\gend
$$
is $B$-colinear. Then it is easy to see that $\zeta$ and $\xi$ are inverse to each other.
Since $\zeta$ is a $B$-linear and $B$-colinear algebra isomorphism and we know that
$[B_{\alpha}, B_{\alpha}]$ and hence $\crta{[B_{\alpha}, B_{\alpha}]}$ is a Yetter-Drinfel'd module
algebra, it follows that $[B_{\alpha}', B_{\alpha}']^{op}$ is a Yetter-Drinfel'd module algebra, too.
\qed\end{proof}

\begin{lma} \lelabel{second lema chain}
It is:
$$[B_{\alpha}', B_{\alpha}']^{op} \iso [B_{\alpha}'^*, B_{\alpha}'^*]$$
as Yetter-Drinfel'd module algebras.
\end{lma}

\begin{proof}
By the parts 3) and 2) of \cite[Corollary 2.2]{VZ1} we have that $[B_{\alpha}', B_{\alpha}']^{op}
\iso [B_{\alpha}'^*, B_{\alpha}'^*]$ as algebras in $\E$. The proof of this relies on the canonical
Morita contexts in $\E$. Though, with the structures \equref{inner H-mod eval} and \equref{H-comod-end-eq}
we have that the canonical Morita contexts are indeed contexts in ${}^B\E$ and ${}_B\E$, so that
$[B_{\alpha}', B_{\alpha}']^{op}\iso [B_{\alpha}'^*, B_{\alpha}'^*]$ as $B$-module and comodule algebras,
hence as Yetter-Drinfel'd module algebras.
\qed\end{proof}

\begin{lma} \lelabel{iso for N=B-alfa}
Let $M$ be a faithfully projective left $B$-module and a left $B$-comodule in $\E$ such that $[M,M]$
is an algebra in ${}_B ^B\YD(\E)$ and let $N$ be $B_{\alpha}$ or $B_{\alpha}'^*$. Assume that $\Phi_{B,B}$
and $\Phi_{B,E}$ are symmetric for all inner-hom objects $E\in {}_B^B\YD(\E)$. Then there is an
algebra isomorphism
$$[M,M] \crta{\ot} [N,N] \iso [M\ot N, M\ot N]$$
in ${}_B ^B\YD(\E)$.
\end{lma}

\begin{proof}
Let $\omega_{M,N}: [M,M] \ot [N,N] \to [M\ot N, M\ot N]$ be the
algebra isomorphism \equref{omega}. Let $\omega^{\dagger}: [M,M] \crta{\ot} [N,N] \to [M\crta{\ot} N, M\crta{\ot} N]$
denote the corresponding algebra isomorphism in ${}_B ^B\YD(\E)$, given by
$ev_{[M\crta\ot N, M\crta\ot N]}(\omega^{\dagger}\crta\ot M\crta\ot N)=(ev_{[M,M]}\crta\ot ev_{[N,N]})
([M,M] \crta\ot \Psi_{[N,N],M}\crta\ot N)$,
where $\Psi$ is the braiding in ${}_B ^B\YD(\E)$. Furthermore, consider the (iso)morphism
$\crta{\omega}: [M,M] \crta{\ot} [N,N] \to [M\ot N, M\ot N]$ formally equal to $\omega^{\dagger}$. Note that
the compatibility of both $\crta{\omega}$ and $\omega^{\dagger}$ with multiplication means one and the same
thing, while their compatibilities with left $B$-module and comodule structures comprehend different
identities. This is own to different $B$-module and comodule structures of $M\ot N$ and $M\crta{\ot} N$,
and hence of their corresponding inner-hom objects.
Using \equref{alfa-YD} and \equref{alfa-YD-'*} and that $\Phi_{B,[N,N]}$ is symmetric, one may prove that $\crta{\omega}$ is left
$B$-colinear, and thus a $B$-comodule algebra isomorphism. 
It is not $B$-linear, however, we will construct a new $B$-linear algebra isomorphism
$\tilde{\sigma}: [M,M] \crta{\ot} [N,N]\to [N\ot M, N\ot M]$ so that the 
$B$-action on $[N\ot M, N\ot M]$ pulled along $\crta{\omega}\tilde{\sigma}^{-1}$
provides $[M\ot N, M\ot N]$ with a new $B$-action 
with respect to which $\crta{\omega}$ becomes a $B$-module morphism.

By the assumption the $B$-action on $[M,M]$ is strongly inner and so far we know that so is the
$B$-action on $[N,N]$, for $N$ being both $B_{\alpha}$ and $B_{\alpha}'^*$. Let $\teta_M, \teta_N$
and $\teta_{M\ot N}$ be the corresponding algebra morphisms. Now consider $\sigma:
[M,M] \crta{\ot} [N,N] \to [N,N]\ot [M,M]$ defined by:
$$
\sigma=
\gbeg{3}{5}
\got{1}{} \got{1}{E_M} \got{1}{\hspace{0,16cm} E_N} \gnl
\glcm \gcl{1} \gnl
\gbmp{\hspace{0,16cm}\theta_{\q N}} \gbr \gnl
\gmu \gcl{1} \gnl
\gob{2}{E_N} \gob{1}{\hspace{0,16cm} E_M}
\gend\quad\textnormal{with the inverse}\quad
\sigma^{-1}=\scalebox{0.9}[0.9]{
\gbeg{4}{8}
\gvac{1} \got{1}{E_N} \got{1}{\hspace{0,16cm} E_M} \gnl
\gvac{1} \gibr \gnl
\glcm \gcl{4} \gnl
\gmp{-} \gcl{2} \gnl
\gbmp{\hspace{0,16cm}\theta_{\q N}} \gnl
\gibr \gnl
\gcl{1} \gmu \gnl
\gob{1}{E_M} \gob{2}{\hspace{0,16cm} E_N}
\gend}
$$
Here $E_M=[M,M]$ and $E_N=[N,N]$.
This is shown to be an algebra morphism without any symmetricity assumptions on the braiding.
Then $\tilde{\sigma}:=\omega_{N,M}\sigma: [M,M] \crta{\ot} [N,N] \to [N\ot M,N\ot M]$ is an
algebra morphism, too. One proves that $\tilde{\sigma}$ is $B$-linear using \equref{teta-tensor}
and that $\Phi_{E_M, B}$ is symmetric. 
Then the $B$-action on $[M\ot N,M\ot N]$ becomes strongly inner via the algebra morphism
$\crta{\omega}\tilde{\sigma}^{-1}\theta_{N\ot M}: B\to [M\ot N,M\ot N]$. 
With this new action $\crta{\omega}$ becomes $B$-linear by the construction. Thus it is a
$B$-module and comodule algebra isomorphism which makes $[M\ot N,M\ot N]$ a Yetter-Drinfel'd module.
\qed\end{proof}

Setting first $M=B_{\alpha}$ and $N=B_{\alpha}'^*$ and then $M=B_{\alpha}'^*$ and $N=B_{\alpha}$,
the subsequent applications of \leref{first lema chain}, \leref{second lema chain} and
\leref{iso for N=B-alfa} yield \equref{chain-iso}. Then the Yetter-Drinfel'd module algebras
$[B_{\alpha}, \q B_{\alpha}] \crta\ot \crta{[B_{\alpha}, \q B_{\alpha}]}$ and
$\crta{[B_{\alpha}, \q B_{\alpha}]} \crta\ot [B_{\alpha}, \q B_{\alpha}]$ are Azumaya algebras
in $\E$.

\subsection{A group map to the quantum Brauer group}

Let $\alpha, \beta\in\Aut(\E; B)$. We define $K=B_{\alpha} \ot B_{\beta}$ and $L=K\ot B_{\gamma}$,
with $\gamma=(\beta\alpha)^{-1}$, both with the usual left $B$-comodule structures, whereas their
structures of left $B$-modules are given by:
$$\gbeg{4}{4}
\got{1}{B} \got{3}{B_{\alpha} \ot B_{\beta}} \gnl
\gcn{2}{1}{1}{3} \gcl{1} \gnl
\gvac{1} \glm \gvac{1} \gnl
\gvac{1} \gob{3}{B_{\alpha} \ot B_{\beta}}
\gend=
\gbeg{4}{6}
\got{2}{B} \got{1}{B_{\alpha}} \got{1}{B_{\beta}} \gnl
\gcmu \gcl{1} \gcl{3} \gnl
\gcl{1} \gbr \gnl
\gcl{1} \gcl{1} \gbmp{\alpha} \gnl
\glm \glm \gnl
\gvac{1}\gob{1}{B_{\alpha}} \gvac{1}\gob{1}{B_{\beta}}
\gend
\quad\qquad\textnormal{and}\qquad\quad
\gbeg{4}{4}
\got{1}{B} \got{3}{K \ot B_{\gamma}} \gnl
\gcn{2}{1}{1}{3} \gcl{1} \gnl
\gvac{1} \glm \gvac{1} \gnl
\gvac{1} \gob{3}{K \ot B_{\gamma}}
\gend=
\gbeg{4}{6}
\got{2}{B} \got{1}{K} \got{1}{B_{\gamma}} \gnl
\gcmu \gcl{1} \gcl{3} \gnl
\gcl{1} \gbr \gnl
\gcl{1} \gcl{1} \gbmp{\hspace{0,1cm}\gamma^{\hspace{-0,08cm}-}} \gnl
\glm \glm \gnl
\gvac{1}\gob{1}{K} \gvac{1}\gob{1}{B_{\gamma}}
\gend
$$
respectively. Then, inserting the antipode rule at an appropriate place and using that $\Phi_{B,B}$
is symmetric, one may show that $K$ obeys:
\begin{equation} \eqlabel{B-alfa-beta-YD}
\gbeg{4}{8}
\got{1}{B} \got{5}{K} \gnl
\gcn{2}{2}{1}{5} \gvac{1} \gcl{2} \gnl \gnl
\gvac{2} \glm \gnl
\gvac{2} \glcm \gnl
\gcn{2}{2}{5}{1} \gvac{1} \gcl{2} \gnl \gnl
\gob{1}{B} \gob{5}{K}
\gend=
\gbeg{6}{10}
\gvac{1} \got{2}{B} \got{4}{K} \gnl
\gvac{1} \gcmu \gcn{1}{1}{4}{4} \gnl
\gvac{1} \gcn{1}{1}{1}{0} \hspace{-0,21cm} \gcmu \glcm \gnl
\gvac{1} \gcl{4} \gbr \gcl{2} \gcl{4} \gnl
\gvac{2} \gbmp{\beta\alpha} \gcl{1} \gnl
\gvac{2}  \gmp{+} \gbr \gnl
\gvac{1} \gvac{1} \gbr \gcl{1} \gnl
\gvac{1} \gcn{1}{1}{1}{2} \gmu \glm \gnl
\gvac{2} \hspace{-0,21cm} \gmu \gcn{1}{1}{4}{4} \gnl
\gvac{2} \gob{2}{B} \gob{4}{K.}
\gend
\end{equation}
As a byproduct of the proof of \equref{B-alfa-beta-YD} one obtains
an identity which changing $\alpha$ to $\gamma^{-1}$ and $\beta$ to $\gamma$ yields that
$L$ is a Yetter-Drinfel'd module.
This means that $[L,L]$ represents the unit in the quantum Brauer group $\BQ(\E; B)$. Similarly,
putting $\beta=\alpha^{-1}$ we conclude from \equref{B-alfa-beta-YD} that $B_{\alpha}\ot B_{\alpha^{-1}}$
is a Yetter-Drinfel'd module, hence $[B_{\alpha}\ot B_{\alpha^{-1}}, B_{\alpha}\ot B_{\alpha^{-1}}]$
represents the unit in $\BQ(\E; B)$ as well. Once we prove that $[B_{\alpha}, B_{\alpha}]$ is an
Azumaya algebra in ${}_B ^B\YD(\E)$ for any $\alpha$, by \leref{iso for N=B-alfa} we could draw
the following conclusions on the classes in $\BQ(\E; B)$:
$$[B_{\alpha}, B_{\alpha}] \crta\ot [B_{\alpha^{-1}}, B_{\alpha^{-1}}] \iso [B_{\alpha}\ot B_{\alpha^{-1}},
B_{\alpha}\ot B_{\alpha^{-1}}] \Rightarrow [[B_{\alpha^{-1}}, B_{\alpha^{-1}}]]=[[B_{\alpha}, B_{\alpha}]]^{-1}$$
$$[K,K] \crta\ot [B_{\gamma}, B_{\gamma}] \iso [K\ot B_{\gamma}, K\ot B_{\gamma}]=[L,L] \Rightarrow
[[K,K]]=[[B_{\gamma}, B_{\gamma}]]^{-1}=[[B_{\gamma^{-1}}, B_{\gamma^{-1}}]].$$
Note that the last identity because of the isomorphism $[B_{\alpha}, B_{\alpha}] \crta\ot
[B_{\beta}, B_{\beta}]\iso [B_{\alpha}\ot B_{\beta}, B_{\alpha}\ot B_{\beta}]$
yields:
\begin{equation} \eqlabel{anti-group}
[[B_{\alpha}, B_{\alpha}]] \cdot [[B_{\beta}, B_{\beta}]]= [[B_{\beta\alpha}, B_{\beta\alpha}]].
\end{equation}
When we applied \leref{iso for N=B-alfa}, note that we used that $\Phi_{B,[M,M]}$ is symmetric
for $M$ being $B_{\alpha}$
and $B_{\alpha}\ot B_{\beta}$ for some $\alpha, \beta\in\Aut(\E; B)$. Identifying $[M,M]\iso M\ot M^*$
and having in mind the braiding axiom $\Phi_{B, X\ot Y}=(X\ot\Phi_{B,Y})(\Phi_{B,X}\ot Y)$, observe that
it suffices to assume that $\Phi_{B,B}$ (and $\Phi_{B,B^*}$) is symmetric. 
\par\bigskip

Let us now interpret the above results for $\C={}_H ^H\YD$, assuming that $\Phi_{B,B}$ is symmetric.
Let $\alpha\in\Aut(\C; B)$ and denote $E_{\alpha}=[B_{\alpha}, B_{\alpha}]=\End(B_{\alpha})$. We saw
that it is an algebra in ${}_B^B\YD(\C)\iso {}_{B\rtimes H} ^{B\rtimes H}\YD$.

\begin{prop}
With assumptions as above the algebra $E_{\alpha}=[B_{\alpha}, B_{\alpha}]$ is an Azumaya algebra in
${}_{B\rtimes H} ^{B\rtimes H}\YD$.
\end{prop}

\begin{proof}
Consider the Azumaya defining morphism $\crta F: E_{\alpha}\crta\ot \crta{E_{\alpha}} \to \End(E_{\alpha})$ in
${}_{B\rtimes H} ^{B\rtimes H}\YD$.
If $\sum_i a_i\ot\crta b_i\in\Ker\crta F$, then for all $c\in E_{\alpha}$ we have:
$\sum_i a_ic_{[0]} (S_{B\rtimes H}^{-1}(c_{[-1]})\cdot_{\hspace{-0,1cm}_{B\rtimes H}} b_i)=0$. Then in particular
for $c=1_{E_{\alpha}}$ we find: $\sum_i a_i b_i=0$. From \ssref{ordinary Azumaya} we know that
$E_{\alpha}\crta\ot \crta{E_{\alpha}}$ is isomorphic to an endomorphism ring in $\C$, hence it is simple. Thus
if $\Ker\crta F\not=\{0\}$
then $\sum_i a_i\ot\crta b_i$ can be arbitrary. Taking $\sum_i a_i\ot\crta b_i=a\ot\crta 1$ for some $a\not= 0$
we get $a=0$, contradiction. Hence $\crta F$ is injective
and since $E_{\alpha}\crta\ot \crta{E_{\alpha}}$ and $\End(E_{\alpha})$ have equal dimensions,
it is an isomorphism. Similarly, one proves that $\crta G$ is an isomorphism, hence
$E_{\alpha}$ is an Azumaya algebra in ${}_{B\rtimes H} ^{B\rtimes H}\YD$. \label{ord.Az proof}
\qed\end{proof}

Now the identity \equref{anti-group} implies: 

\begin{prop}
The map
\begin{equation} \eqlabel{map Pi}
\Pi: \Aut(\C; B) \to \BQ(\C; B), \quad \Pi(\alpha)=[\End(B_{\alpha^{-1}})]
\end{equation}
is a group map.
\end{prop}

\par\bigskip

Before proving our main theorem, let us next
show that there is an inner action of $\Aut(\C; B)$ on $\BQ(\C; B)$. Let us go back
for a moment to our general monoidal category $\E$ and a Hopf algebra $B\in\E$. Let $A$ be an Azumaya
algebra in ${}_B^B\YD(\E)$ and take $\alpha\in\Aut(\E; B)$. We define $A(\alpha):=A$ as an algebra with
the $B$-(co)module structures given as follows:
\begin{equation} \eqlabel{A(alfa)}
\gbeg{3}{4}
\got{1}{\hspace{-0,12cm}B} \got{1}{\hspace{0,12cm}A(\alpha)} \gnl
\glm \gnl
\gvac{1} \gcl{1} \gnl
\gvac{1} \gob{1}{A(\alpha)}
\gend=
\gbeg{3}{4}
\got{1}{B} \got{1}{A} \gnl
\gbmp{\alpha} \gcl{1} \gnl
\glm \gnl
\gvac{1} \gob{1}{A}
\gend
\quad\textnormal{and}\quad
\gbeg{3}{4}
\gvac{1} \got{1}{\hspace{0,12cm}A(\alpha)} \gnl
\glcm \gnl
\gcl{1} \gcl{1} \gnl
\gob{1}{\hspace{-0,12cm}B} \gob{1}{\hspace{0,12cm}A(\alpha)}
\gend=
\gbeg{3}{4}
\gvac{1} \got{1}{A} \gnl
\glcm \gnl
\gbmp{\hspace{0,1cm}\alpha^{\hspace{-0,08cm}-}} \gcl{1} \gnl
\gob{1}{B} \gob{1}{A.}
\gend
\end{equation}
Then it is easily seen that $A(\alpha)$ is an algebra in ${}_B^B\YD(\E)$.

\begin{prop} \prlabel{inner action}
For any Azumaya algebra $A$ in ${}_B^B\YD(\E)$ and $\alpha\in\Aut(\E; B)$ we have
$A\crta\ot [B_{\alpha}, B_{\alpha}] \iso [B_{\alpha}, B_{\alpha}] \crta\ot A(\alpha^{-1})$
as algebras in ${}_B^B\YD(\E)$ if $\Phi_{B,B}$ and $\Phi_{B,A}$ are symmetric.
\end{prop}

\begin{proof}
As before, we identify $[B_{\alpha}, B_{\alpha}]\iso B_{\alpha} \ot B_{\alpha}^*$. Then one may
prove that the morphism $\Fi: A\crta\ot [B_{\alpha}, B_{\alpha}] \to [B_{\alpha}, B_{\alpha}] \crta\ot A(\alpha^{-1})$
given by:
$$\Fi=
\gbeg{6}{10}
\gvac{1} \got{1}{A} \got{1}{B_{\alpha}} \gvac{1} \got{1}{B_{\alpha}^*} \gnl
\glcm \gcl{1} \gvac{1} \gcl{3} \gnl
\gcl{1} \gibr \gnl
\glm \gcn{1}{1}{1}{3} \gnl
\gvac{1} \gcl{5} \gvac{1} \gibr \gnl
\gvac{2} \glcm \gcl{4} \gnl
\gvac{2} \gmp{-} \gcl{1} \gnl
\gvac{2} \gbr \gnl
\gvac{2} \gcl{1} \glm \gnl
\gvac{1} \gob{1}{B_{\alpha}} \gob{1}{\hspace{0,12cm}B_{\alpha}^*} \gob{3}{\hspace{0,12cm}A(\alpha^{-1})}
\gend=
\gbeg{3}{11}
\gvac{1} \got{1}{A} \got{1}{B_{\alpha}} \gvac{1} \got{1}{B_{\alpha}^*} \gnl
\glcm \gcl{1} \gvac{1} \gcl{3} \gnl
\gcl{1} \gibr \gnl
\glm \gcn{1}{1}{1}{3} \gnl
\gvac{1} \gcl{6} \gvac{1} \gibr \gnl
\gvac{2} \glcm \gcl{4} \gnl
\gvac{2} \gmp{-} \gcl{1} \gnl
\gvac{2} \gbr \gnl
\gvac{2} \gcl{2} \gbmp{\hspace{0,1cm}\alpha^{\hspace{-0,08cm}-}} \gnl
\gvac{3} \glm \gnl
\gvac{1} \gob{1}{B_{\alpha}} \gob{1}{\hspace{0,12cm}B_{\alpha}^*} \gob{3}{A}
\gend
$$
is left $B$-linear and colinear and an algebra morphism.
\qed\end{proof}

In the case of $\E=\C$ we may claim the following:

\begin{lma}
Let $A$ and $A'$ be two Azumaya algebras in $\C$ and $\alpha\in\Aut(\C; B)$. Then $A(\alpha)$
is an Azumaya algebra in $\C$ and $A(\alpha)\crta\ot A'(\alpha) \iso (A\crta\ot A')(\alpha)$
as algebras in $\C$.
\end{lma}

\begin{proof}
The first part is proved similarly as we proved lately 
that $E_{\alpha}=\End(B_{\alpha})$ is an Azumaya algebra in ${}_{B\rtimes H} ^{B\rtimes H}\YD$
as $\alpha$ preserves the unit. The proof of the second part is straightforward and easy.
\qed\end{proof}

\begin{cor} \colabel{inner action}
Assume that $\Phi_{B,B}$ and $\Phi_{B,A}$ are symmetric for all Azumaya algebras $A$ in ${}_B^B\YD(\C)$.
Then there is an inner action
$$\BQ(\C; B)\ot\Aut(\C; B) \to \BQ(\C; B), \quad ([A], \alpha)\mapsto [A(\alpha)]$$
given by \equref{A(alfa)}.
\end{cor}

\begin{proof}
The definition of the action makes sense by the first part of the above lemma. 
Observe that $\left[\crta{A(\alpha)}\right]=\left[\crta A(\alpha)\right]$ in $\BQ(\C; B)$, since the
module structure (twisted by $\alpha$) and the algebra structure are independent. Then if $A\w\sim\w A'$, it
is $1=[(A\crta\ot\hspace{0,08cm} \crta{A'})(\alpha)]=[A(\alpha)][\crta{A'}(\alpha)]=[A(\alpha)][A'(\alpha)]^{-1}$,
which implies $A(\alpha)\w\sim\w A'(\alpha)$, by the second part of the above lemma.
The properties of the action are checked straightforwardly. Recall that it is $[E_{\alpha}]^{-1}=
[E_{\alpha^{-1}}]$ in ${}_B^B\YD(\C)$. Then \prref{inner action} interpreted for $\E=\C$
yields that $[E_{\alpha}][A][E_{\alpha}]^{-1}=[A(\alpha)]$ in $\BQ(\C; B)$.
\qed\end{proof}

\par\bigskip

\section{A sequence for the quantum Brauer group}

In this section the braided monoidal category $\E$ from the previous pages is replaced by $\C$.
Whenever two strings cross each other in the braided diagrams in the next proves, it represents the
braiding \equref{braidingD}. In the next proof we will identify the
grops $\BQ(\C; B)$ and $\BM(\C; D(B))$. This imposes a symmetricity condition on the braiding, as we
commented before \coref{D(H)-strongly inner}. Even without this identification $\Phi_{B,A}$ has to
be symmetric for all algebras $A\in {}_B ^B\YD(\C)$ in order that $\crta A$ be an algebra there,
as we stressed in the paragraph after \equref{br in YD}. Our main result is:

\begin{thm} \thlabel{seqBQ}
Let $H$ be a Hopf algebra over a field with a bijective antipode and set $\C={}_H ^H\YD$.
Suppose $B$ is a finite-dimensional Hopf algebra in $\C$ so that 
$\Phi_{B,M}$ is symmetric for all $M\in\C$. 
There is an exact sequence of groups:
\begin{equation} \eqlabel{seqBQ}
\bfig
\putmorphism(-200, 0)(1, 0)[1`G(D(B)^*)`]{500}1a
\putmorphism(300, 0)(1, 0)[\phantom{G(D(B)^*)}`G(D(B))` ]{700}1a
\putmorphism(1000, 0)(1, 0)[\phantom{G(D(B))}`\Aut(\C;B)`\Theta]{700}1a
\putmorphism(1700, 0)(1, 0)[\phantom{\Gal(\C;H)}`\BQ(\C; B).`\Pi]{700}1a
\efig
\end{equation}
\end{thm}

\begin{proof}
By \leref{equiv cond gr} and the definition of $S(B)$ we have that $G(D(B)^*)$ is a subgroup of
$G(D(B))$ and that the sequence is exact at $G(D(B))$.  Suppose that $\alpha^{-1}\in\Ker\Pi$,
that is $[E_{\alpha}]=1$ in $\BM(\C; D(B))$. 
Due to \coref{D(H)-strongly inner} the $D(B)$-action on $E_{\alpha}$ is strongly inner, let
$\theta: D(B)\to E_{\alpha}$ be the corresponding algebra morphism in $\C$. Denote by $\theta_B$ the
restriction of $\theta$ to $B$, that is $\theta_B=(B\stackrel{\eta_{B^*}\ot B}{\hookrightarrow} D(B)
\stackrel{\theta}{\to} E_{\alpha}$). The $B$-action on $B_{\alpha}$ induced by $\theta_B$ via
\equref{teta} induces a new $B$-action $\mu$ on $E_{\alpha}$ given by \equref{H-mod-end-eq}:
$$
\gbeg{3}{5}
\got{1}{B} \got{3}{E_{\alpha}} \gnl
\gcn{1}{1}{1}{3} \gvac{1} \gcl{1} \gnl
\gvac{1} \glmpt \gnot{\hspace{-0,44cm}\mu} \grmptb \gnl
\gvac{2}\gcl{1}\gnl
\gvac{2}\gob{1}{E_{\alpha}}
\gend=
\gbeg{4}{8}
\got{2}{B} \got{1}{E_{\alpha}} \gnl
\gcmu  \gcl{2} \gvac{1} \gnl
\gcl{1} \gmp{+} \gnl
\gcl{1} \gbr  \gnl
\gbmp{\teta_B} \gcl{1} \gbmp{\teta_B}  \gnl
\gcn{1}{1}{1}{0} \gmu \gnl
\hspace{-0,22cm}\gwmu{3} \gnl
\gvac{1} \gob{1}{E_{\alpha}}
\gend=
\gbeg{5}{9}
\got{2}{} \got{2}{B} \got{2}{E_{\alpha}} \gnl
\gvac{1} \hspace{-0,4cm} \gu{1} \gcn{1}{1}{3}{3} \gcl{1} \gvac{1} \gcl{1} \gnl
\gvac{1} \hspace{-0,34cm} \gcmu \gcmu \gcn{1}{1}{2}{1} \gnl
\gvac{1} \gcl{1} \gbr \gbr \gnl
\gvac{1} \glmpt \gnot{\hspace{-0,42cm}\theta} \grmptb \gbr \gmp{+} \gnl
\gvac{2} \gcl{2} \gcl{1} \glmpt \gnot{\hspace{-0,42cm}\theta} \grmptb \gnl
\gvac{3} \gwmu{3} \gnl
\gvac{2} \gwmu{3} \gnl
\gvac{3} \gob{1}{E_{\alpha}}
\gend\stackrel{\equref{1S}}{=}
\gbeg{4}{10}
\gvac{1} \got{2}{B} \got{1}{\hspace{-0,4cm}E_{\alpha}} \gnl
\gvac{1} \hspace{-0,4cm} \gu{1} \gcl{1} \gcl{1} \gnl
\gvac{1} \glmpt \gnot{\hspace{-0,44cm}D(B)} \grmpt \gcl{1} \gnl
\gvac{1} \gcmu \gcl{2} \gvac{1} \gnl
\gvac{1} \gcl{1} \gmp{S} \gnl
\gvac{1} \gcl{1} \gbr  \gnl
\gvac{1} \gbmp{\teta} \gcl{1} \gbmp{\teta}  \gnl
\gvac{1} \gcn{1}{1}{1}{0} \gmu \gnl
\gvac{1} \hspace{-0,22cm}\gwmu{3} \gnl
\gvac{2} \gob{1}{E_{\alpha}}
\gend=
\gbeg{4}{6}
\gvac{1} \got{2}{B} \got{1}{\hspace{-0,2cm}E_{\alpha}} \gnl
\gvac{1} \hspace{-0,4cm} \gu{1} \gcl{1} \gcl{2} \gnl
\gvac{1} \glmpt \gnot{\hspace{-0,44cm}D(B)} \grmptb \gnl
\gvac{2} \glm \gnl
\gvac{3} \gcl{1}\gnl
\gvac{3} \gob{1}{E_{\alpha}}
\gend\stackrel{\equref{BHmod-conv}}{=}
\gbeg{3}{5}
\got{1}{B} \got{3}{E_{\alpha}} \gnl
\gcn{1}{1}{1}{3} \gvac{1} \gcl{1} \gnl
\gvac{1} \glm \gnl
\gvac{2}\gcl{1}\gnl
\gvac{2}\gob{1}{E_{\alpha}.}
\gend
$$
Note that here $S$ stands for the antipode on $D(B)$. So, this newly induced $B$-action on $E_{\alpha}$
coincides with the original one. 
By \leref{mods and grouplks} there is an algebra morphism $\lambda: B\to I$ such that:
\begin{equation} \eqlabel{gr-lk mod}
\gbeg{3}{6}
\got{1}{B} \got{3}{B_{\alpha}} \gnl
\gcl{1} \gvac{1} \gcl{2} \gnl
\gbmp{\theta_{\q B}} \gnl
\gcn{1}{1}{1}{3}\gcn{1}{1}{3}{1} \gnl
\gvac{1} \gmp{\ev} \gnl
\gvac{1} \gob{1}{B_{\alpha}}
\gend=
\gbeg{3}{6}
\got{2}{B} \got{1}{B_{\alpha}} \gnl
\gcmu \gcl{3} \gnl
\gcl{1} \gbmp{\lambda} \gnl
\gcn{1}{1}{1}{3} \gnl
\gvac{1} \glm \gnl
\gvac{2} \gob{1}{B_{\alpha}}
\gend\stackrel{\equref{alfa}}{=}
\gbeg{7}{9}
\gvac{1} \got{2}{B} \got{1}{B} \gnl
\gvac{1} \gcmu \gcl{4} \gnl
\gcn{1}{1}{3}{2} \gvac{1} \gbmp{\lambda} \gnl
\gcmu \gnl
\gmp{-} \gbmp{\alpha} \gnl
\gcn{1}{1}{1}{3} \gwmu{3} \gnl
\gvac{1} \gbr \gnl
\gvac{1} \gmu \gnl
\gvac{1} \gob{2}{B}
\gend
\end{equation}
Symmetrically, let $\theta_{B^*}$ be the restriction of $\theta$ to $B^*$, that is $\theta_{B^*}=
(B^*\stackrel{B^*\ot\eta_{B}}{\hookrightarrow} D(B)\stackrel{\theta}{\to} E_{\alpha}$). The same way
as above, we have that the newly obtained $B^*$-action on $B_{\alpha}$ via $\theta_{B^*}$ induces
the same $B^*$-action on $E_{\alpha}$ as the original one (the restriction of the $D(B)$-action on
$E_{\alpha}$, which is induced by the $B^*$-action on $B_{\alpha}$ that comes form the left $B$-coaction
on $B_{\alpha}$, \equref{H*mod}). In this way we get that there is an algebra morphism $g^*: B^*\to I$ such that:
\begin{equation} \eqlabel{gr-lk comod} 
\gbeg{3}{6}
\got{1}{(B^{op})^*} \got{3}{B_{\alpha}} \gnl
\gcl{1} \gvac{1} \gcl{2} \gnl
\gbmp{\hspace{0,12cm}\theta_{\hspace{-0,07cm} B^{\q*}}} \gnl
\gcn{1}{1}{1}{3}\gcn{1}{1}{3}{1} \gnl
\gvac{1} \gmp{\ev} \gnl
\gvac{1} \gob{1}{B_{\alpha}}
\gend=
\gbeg{3}{6}
\got{2}{(B^{op})^*} \got{1}{B_{\alpha}} \gnl
\gcmu \gcl{3} \gnl
\gcl{1} \gbmp{g^*} \gnl
\gcn{1}{1}{1}{3} \gnl
\gvac{1} \glm \gnl
\gvac{2} \gob{1}{B_{\alpha}}
\gend=
\gbeg{3}{7}
\got{2}{B^*} \got{1}{B_{\alpha}} \gnl
\gcmu \gcl{4} \gnl
\gbr \gnl
\gcl{1} \gbmp{g^*} \gnl
\gcn{1}{1}{1}{3} \gnl
\gvac{1} \glm \gnl
\gvac{2} \gob{1}{B_{\alpha}}
\gend\hspace{-0,1cm}=\hspace{-0,1cm}
\gbeg{4}{6}
\got{2}{B^*} \gvac{1} \got{1}{B_{\alpha}} \gnl
\gcn{1}{1}{2}{2} \gvac{1} \gbmp{g} \gcl{3} \gnl
\gcmu \gcl{1} \gnl
\gcl{1} \gbr \gnl
\gev \glm \gnl
\gvac{3} \gob{1}{B_{\alpha}}
\gend\stackrel{\equref{H*mod}}{\stackrel{\equref{co-alfa}}{=}}
\gbeg{5}{6}
\got{2}{B^*} \gvac{1} \got{2}{B} \gnl
\gcn{1}{1}{2}{2} \gvac{1} \gbmp{g} \gcmu \gnl
\gcmu \gcl{1} \gibr \gnl
\gcl{1} \gbr \gcl{1} \gcl{2} \gnl
\gev \gev \gnl
\gvac{4} \gob{1}{B}
\gend\hspace{-0,1cm}=\hspace{-0,1cm}
\gbeg{5}{6}
\got{1}{B^*} \got{1}{} \gvac{1} \got{2}{B} \gnl
\gcl{1} \gvac{2} \gcmu \gnl
\gcl{1} \gbmp{g} \gvac{1} \gibr \gnl
\gcl{1} \gwmu{3} \gcl{2} \gnl
\gwev{3} \gnl
\gvac{4} \gob{1}{B}
\gend
\end{equation}
In the second equality above we used the fact that $(B^{op})^*\iso (B^*)^{cop}$ as coalgebras,
\cite[Remark 4.5]{Femic1}, and that $\Phi_{B^*, B^*}$ is symmetric.
Using \equref{D(H)-mult} one may compute the identity: 
$$
\gbeg{7}{12}
\gvac{1} \got{1}{B^*} \gvac{3} \got{1}{B} \gnl
\gwcm{3} \gvac{1} \gwcm{3} \gnl
\gcn{1}{1}{1}{2} \gvac{1} \hspace{-0,34cm} \gcmu  \gcmu \gcn{1}{1}{2}{1} \gnl
\gvac{1} \gibr \gcl{1} \gmp{-}  \gbr \gnl
\gvac{1} \gcl{2} \gcl{1} \gev \gcl{1} \gcl{2} \gnl
\gvac{2} \gwev{4} \gnl
\gvac{1} \gcn{1}{1}{1}{5} \gvac{3} \gcn{1}{1}{3}{-1} \gnl
\gvac{3} \gibr \gnl
\gvac{2} \gu{1} \gcl{1} \gcl{1} \gu{1} \gnl
\gvac{2} \glmptb \gnot{\hspace{-0,44cm}D(B)} \grmpt \glmpt \gnot{\hspace{-0,44cm}D(B)} \grmptb \gnl
\gvac{2} \gwmu{4} \gnl
\gvac{3} \gob{2}{D(B)}
\gend=id_{B^*\ot B}
$$
Let us now act with $D(B)$ on $B_{\alpha}$ via the $B^*$- and $B$-actions \equref{gr-lk mod} and
\equref{gr-lk comod}, recall \equref{BHmod-conv}. Because of the above identity we get: 
$$
\gbeg{7}{12}
\got{2}{B^*} \gvac{2} \got{2}{B} \got{1}{B} \gnl
\gcmu \gbmp{g} \gvac{1} \gcmu \gcl{4} \gnl
\gcl{1} \gbr \gcn{1}{1}{3}{2} \gvac{1} \gbmp{\lambda} \gnl
\gev \gcl{7} \gcmu \gnl
\gvac{3} \gmp{-} \gbmp{\alpha} \gnl
\gvac{3} \gcn{1}{1}{1}{3} \gwmu{3} \gnl
\gvac{4} \gbr \gnl
\gvac{4} \gmu \gnl
\gvac{4} \gcmu \gnl
\gvac{4} \gibr \gnl
\gvac{2} \gwev{3} \gcl{1} \gnl
\gvac{4} \gob{3}{B}
\gend=\hspace{-0,24cm}
\gbeg{9}{16}
\gvac{1} \got{1}{B^*} \gvac{3} \got{1}{B} \gvac{1} \got{1}{B} \gnl
\gwcm{3} \gvac{1} \gwcm{3} \gcn{1}{4}{1}{1} \gnl
\gcn{1}{1}{1}{2} \gvac{1} \hspace{-0,34cm} \gcmu  \gcmu \gcn{1}{1}{2}{1} \gnl
\gvac{1} \gibr \gcl{1} \gmp{-}  \gbr \gnl
\gvac{1} \gcl{2} \gcl{1} \gev \gcl{1} \gcl{2} \gnl
\gvac{2} \gwev{4} \gvac{1} \gcmu \gnl
\gvac{1} \gcn{1}{1}{1}{5} \gvac{3} \gcn{1}{1}{3}{-1} \gvac{1} \gibr \gnl
\gvac{3} \gibr \gvac{1} \gbmp{g} \gcl{1} \gcl{2} \gnl
\gvac{3} \hspace{-0,34cm} \gcmu \gcn{1}{1}{0}{3} \hspace{-0,22cm} \gvac{1} \gmu \gnl
\gvac{3} \gcn{1}{1}{2}{1} \gvac{1} \hspace{-0,34cm} \gbmp{\lambda} \gvac{1} \gev \gcn{1}{1}{2}{0} \gnl
\gvac{3} \gcmu \gcn{2}{2}{8}{3} \gnl
\gvac{3} \gmp{-} \gbmp{\alpha} \gnl
\gvac{3} \gcn{1}{1}{1}{3} \gwmu{3} \gnl
\gvac{4} \gbr \gnl
\gvac{4} \gmu \gnl
\gvac{4} \gob{2}{B}
\gend\hspace{-0,2cm}=\hspace{0,14cm}
\gbeg{9}{15}
\gvac{1} \got{1}{\hspace{-0,2cm}B^*} \got{2}{\hspace{0,12cm}B} \gvac{1} \got{2}{B} \gnl
\gcmu \gcmu \gvac{1} \gcn{1}{2}{2}{2} \gnl
\gibr \gbr \gnl
\gcl{3} \gev \gcn{1}{1}{1}{2} \gvac{1} \gcmu \gnl
\gvac{2} \gbmp{g} \gcmu \gibr \gnl
\gvac{2} \gcl{1} \gmp{-} \gbr \gcl{1} \gnl
\hspace{-0,34cm} \gcmu \gvac{1} \hspace{-0,2cm} \gibr \gcl{1} \gcn{1}{1}{1}{2} \gcn{1}{1}{1}{4} \gnl
\gcn{1}{1}{2}{2} \gcn{2}{2}{2}{4} \gcn{1}{1}{1}{2} \gmu \gcmu \gcn{1}{4}{2}{2} \gnl
\gcn{2}{2}{2}{4} \gvac{2} \hspace{-0,34cm} \gbr \gcn{1}{1}{2}{2} \gvac{1} \hspace{-0,34cm} \gbmp{\lambda} \gnl
\gvac{4} \hspace{-0,34cm} \gbr \gcl{1} \gcmu \gnl
\gvac{3} \gev \gev \gmp{-} \gbmp{\alpha} \gnl
\gvac{7} \gcn{1}{1}{1}{3} \gwmu{3} \gnl
\gvac{8} \gbr \gnl
\gvac{8} \gmu \gnl
\gvac{8} \gob{2}{B}
\gend=
\gbeg{9}{15}
\got{2}{B^*} \got{2}{B} \got{2}{B} \gnl
\gcmu \gcmu \gcmu \gnl
\gibr \gbr \gibr \gnl
\gcl{5} \gev \gbr \gcn{2}{2}{1}{5} \gnl
\gvac{1} \gbmp{g} \gvac{1} \gcl{1} \gcn{1}{1}{1}{2} \gnl
\gvac{1} \gwmu{3} \gcmu \gvac{1} \gcl{6} \gnl
\gvac{2} \gcl{1} \gvac{1} \gmp{-} \gcn{1}{1}{1}{2} \gnl
\gvac{2} \gwmu{3} \gcmu \gnl
\gwev{4} \gcn{1}{1}{3}{2} \gvac{1} \gbmp{\lambda} \gnl
\gvac{4} \gcmu \gnl
\gvac{4} \gmp{-} \gbmp{\alpha} \gnl
\gvac{4} \gcn{1}{1}{1}{3} \gwmu{3} \gnl
\gvac{5} \gbr \gnl
\gvac{5} \gmu \gnl
\gvac{5} \gob{2}{B.}
\gend
$$
Acting on both sides above by $\Epsilon$ we obtain:
$$
\gbeg{6}{12}
\got{2}{B^*} \gvac{1} \got{2}{B} \got{1}{B} \gnl
\gcmu \gvac{1} \gcmu \gcl{4} \gnl
\gcl{7} \gcl{6} \gcn{1}{1}{3}{2} \gvac{1} \gbmp{\lambda} \gnl
\gvac{2} \gcmu \gnl
\gvac{2} \gmp{-} \gbmp{\alpha} \gnl
\gvac{2} \gcn{1}{1}{1}{3} \gwmu{3} \gnl
\gvac{3} \gbr \gnl
\gvac{2} \gbmp{g} \gmu \gnl
\gvac{1} \gbr \gcn{1}{1}{2}{1} \gnl
\gev \gev \gnl
\gvac{4} \gob{3}{}
\gend=
\gbeg{5}{12}
\got{1}{B^*} \gvac{1} \got{2}{B} \got{1}{B} \gnl
\gcl{8} \gvac{1} \gcmu \gcl{4} \gnl
\gvac{1} \gcn{1}{1}{3}{2} \gvac{1} \gbmp{\lambda} \gnl
\gvac{1} \gcmu \gnl
\gvac{1} \gmp{-} \gbmp{\alpha} \gnl
\gvac{1} \gcn{1}{1}{1}{3} \gwmu{3} \gnl
\gvac{2} \gbr \gnl
\gvac{2}  \gmu \gnl
\gvac{1} \gbmp{g} \gcn{1}{1}{2}{1} \gnl
\gcn{1}{1}{1}{2} \gmu \gnl
\hspace{-0,2cm} \gvac{1} \gev \gnl
\gvac{4} \gob{3}{}
\gend=
\gbeg{7}{12}
\got{2}{B^*} \got{2}{B} \got{1}{B} \gnl
\gcmu \gcmu \gcl{2} \gnl
\gibr \gbr \gnl
\gcl{5} \gev \gbr \gnl
\gvac{1} \gbmp{g} \gvac{1} \gcl{1} \gcn{1}{1}{1}{2} \gnl
\gvac{1} \gwmu{3} \gcmu \gnl
\gvac{2} \gcl{1} \gvac{1} \gmp{-} \gbmp{\lambda} \gnl
\gvac{2} \gwmu{3} \gnl
\gwev{4} \gnl
\gvac{4} \gob{3}{}
\gend=
\gbeg{7}{12}
\got{2}{B^*} \gvac{1} \got{2}{B} \got{3}{B} \gnl
\gcmu \gvac{1} \gcmu \gvac{1} \gcl{4} \gnl
\gcl{8} \gcl{7} \gcn{1}{1}{3}{2} \gvac{1} \gcl{3} \gnl
\gvac{2} \gcmu \gnl
\gvac{2} \gmp{-} \gbmp{\lambda} \gvac{1} \gbmp{g} \gnl
\gvac{2} \gcl{1} \gcn{1}{1}{3}{1} \gvac{1} \gmu \gnl
\gvac{2} \gbr \gcn{1}{1}{4}{1} \gnl
\gvac{2} \gcl{2} \gbr \gnl
\gvac{3} \gmu \gnl
\gvac{1} \gbr \gcn{1}{1}{2}{1} \gnl
\gev \gev \gnl
\gvac{4} \gob{3}{}
\gend=
\gbeg{7}{12}
\got{1}{B^*} \gvac{1} \got{2}{B} \got{3}{B} \gnl
\gcl{9} \gvac{1} \gcmu \gvac{1} \gcl{4} \gnl
\gvac{1} \gcn{1}{1}{3}{2} \gvac{1} \gcl{3} \gnl
\gvac{1} \gcmu \gnl
\gvac{1} \gmp{-} \gbmp{\lambda} \gvac{1} \gbmp{g} \gnl
\gvac{1} \gcl{1} \gcn{1}{1}{3}{1} \gvac{1} \gmu \gnl
\gvac{1} \gbr \gcn{1}{1}{4}{1} \gnl
\gvac{1} \gcl{1} \gbr \gnl
\gvac{1} \gcn{1}{1}{1}{2} \gmu \gnl
\gvac{2} \hspace{-0,22cm} \gmu \gnl
\gvac{1} \hspace{-0,2cm} \gwev{3} \gnl
\gvac{4} \gob{3}{}
\gend
$$
The second and the last diagram in the above computation by the universal property of $ev_{B^*}$
yield \equref{racun1}, which applied to $B\ot\eta_B$ yields \equref{racun2}: \vspace{-0,3cm}
\begin{center}
\begin{tabular}{p{6.8cm}p{1cm}p{6.8cm}}
\begin{equation} \eqlabel{racun1}
\gbeg{4}{11}
\gvac{1} \got{2}{B} \got{1}{B} \gnl
\gvac{1} \gcmu \gcl{4} \gnl
\gcn{1}{1}{3}{2} \gvac{1} \gbmp{\lambda} \gnl
\gcmu \gnl
\gmp{-} \gbmp{\alpha} \gnl
\gcn{1}{1}{1}{3} \gwmu{3} \gnl
\gvac{1} \gbr \gnl
\gvac{1}  \gmu \gnl
\gbmp{g} \gcn{1}{1}{2}{1} \gnl
\gmu \gnl
\gob{2}{B}
\gend=
\gbeg{6}{11}
\gvac{1} \got{2}{B} \got{3}{B} \gnl
\gvac{1} \gcmu \gvac{1} \gcl{4} \gnl
\gcn{1}{1}{3}{2} \gvac{1} \gcl{3} \gnl
\gcmu \gnl
\gmp{-} \gbmp{\lambda} \gvac{1} \gbmp{g} \gnl
\gcl{1} \gcn{1}{1}{3}{1} \gvac{1} \gmu \gnl
\gbr \gcn{1}{1}{4}{1} \gnl
\gcl{1} \gbr \gnl
\gcn{1}{1}{1}{2} \gmu \gnl
\gvac{1} \hspace{-0,23cm} \gmu \gnl
\gob{4}{B}
\gend
\end{equation}  & &
\begin{equation} \eqlabel{racun2}
\gbeg{4}{9}
\gvac{2} \got{2}{B} \gnl
\gvac{2} \gcmu \gnl
\gvac{1} \gcn{1}{1}{3}{2} \gvac{1} \gbmp{\lambda} \gnl
\gvac{1} \gcmu \gnl
\gvac{1} \gmp{-} \gbmp{\alpha} \gnl
\gbmp{g} \gbr \gnl
\gcn{1}{1}{1}{2} \gmu \gnl
\gvac{1} \hspace{-0,2cm} \gmu \gnl
\gob{4}{B}
\gend=
\gbeg{6}{11}
\gvac{1} \got{2}{B} \gnl
\gvac{1} \gcmu \gnl
\gcn{1}{1}{3}{2} \gvac{1} \gcl{3} \gnl
\gcmu \gnl
\gmp{-} \gbmp{\lambda} \gnl
\gcl{1} \gcn{1}{1}{3}{1} \gnl
\gbr \gbmp{g} \gnl
\gcl{1} \gbr \gnl
\gcn{1}{1}{1}{2} \gmu \gnl
\gvac{1} \hspace{-0,23cm} \gmu \gnl
\gob{4}{B}
\gend
\end{equation}
\end{tabular}
\end{center}
Composing $(-\ot B)\Phi_{B,B}\Delta_B$ to \equref{racun2} from above and then $\nabla_B$
from below, by the (co)associativity and the antipode rule we end up with:
$$
\gbeg{3}{5}
\gvac{1} \got{2}{B} \gnl
\gvac{1} \gcmu \gnl
\gbmp{g} \gbmp{\alpha} \gbmp{\lambda} \gnl
\gmu \gnl
\gob{2}{B}
\gend=
\gbeg{3}{5}
\got{2}{B} \gnl
\gcmu \gnl
\gbmp{\lambda} \gcl{1} \gbmp{g} \gnl
\gvac{1} \gmu \gnl
\gob{4}{B}
\gend
$$
This finally implies:
$$
\gbeg{1}{5}
\got{1}{B} \gnl
\gcl{1} \gnl
\gbmp{\alpha} \gnl
\gcl{1} \gnl
\gob{1}{B}
\gend=
\gbeg{6}{6}
\gvac{3} \got{1}{B} \gnl
\gvac{2} \gwcm{3} \gnl
\gbmp{g^{\hspace{-0,08cm}-}}  \gbmp{g} \gbmp{\alpha} \gvac{1} \hspace{-0,34cm} \gcmu \gnl
\gvac{1} \hspace{-0,22cm} \gmu \gcn{1}{1}{1}{0} \gvac{1} \hspace{-0,34cm} \gbmp{\lambda} \gbmp{\hspace{0,1cm}\lambda^{\hspace{-0,08cm}-}} \gnl
\gvac{2} \gmu \gnl
\gvac{2} \gob{2}{B}
\gend=
\gbeg{6}{7}
\gvac{3} \got{1}{B} \gnl
\gvac{3} \hspace{-0,34cm} \gcmu \gnl
\gvac{3} \hspace{-0,2cm} \gcmu \gcn{1}{1}{0}{1} \gnl
\gvac{1} \gbmp{g^{\hspace{-0,08cm}-}} \gbmp{g} \gbmp{\alpha} \gbmp{\lambda} \gbmp{\hspace{0,1cm}\lambda^{\hspace{-0,08cm}-}} \gnl
\gvac{1} \gcn{1}{1}{1}{2} \gmu \gnl
\gvac{2} \hspace{-0,22cm} \gmu \gnl
\gvac{2} \gob{2}{B}
\gend=
\gbeg{6}{7}
\gvac{2} \got{2}{B} \gnl
\gvac{2} \hspace{-0,34cm} \gwcm{3} \gnl
\gvac{2} \hspace{-0,34cm} \gcmu \gcn{1}{1}{2}{3} \gnl
\gvac{1} \gbmp{g^{\hspace{-0,08cm}-}} \gbmp{\lambda} \gcl{1} \gbmp{g}  \gbmp{\hspace{0,1cm}\lambda^{\hspace{-0,08cm}-}} \gnl
\gvac{1} \gcn{1}{1}{1}{2} \gvac{1} \gmu \gnl
\gvac{2} \hspace{-0,22cm} \gwmu{3} \gnl
\gvac{3} \gob{1}{B}
\gend=
\gbeg{1}{6}
\got{1}{B} \gnl
\gcl{1} \gnl
\gbmp{\crta\lambda} \gnl
\gbmp{\vspace{-0,08cm}\crta{g^{\hspace{-0,08cm}-}}} \gnl
\gcl{1} \gnl
\gob{1}{B}
\gend
$$
proving that $\alpha^{-1}=\crta g(\crta\lambda)^{-1}=\Theta(\lambda^{-1}, g)$, i.e. $\alpha^{-1}\in\Bi\Theta$.
\par\medskip

The final step in the proof is to show that $\Bi\Theta\subseteq\Ker\Pi$. Take $\alpha^{-1}\in\Bi\Theta$,
where $\alpha=\crta{(g^*)^{-1}} \hspace{0,1cm}\crta{\lambda}$ for $g^*\in\Alg(B^*, I)$ and $\lambda\in
\Alg(B, I)$. We will prove that the $D(B)$-action on $E_{\alpha}$ is strongly inner. By
\coref{D(H)-strongly inner} this will finish the proof. We give a new $B$-action $\mu$ on $B_{\alpha}$
given by the second diagram in \equref{gr-lk mod} and a new $B^*$-action $\nu$ on $B_{\alpha}$ given by
the second diagram in \equref{gr-lk comod}. They indeed define actions, since $g^*$ and $\lambda$ are
algebra morphisms. After a long computation which we omit for practical reasons one may check that the
identity \equref{B tie H -mod} holds, i.e. these actions make $B_{\alpha}$ a $D(B)$-module. We further
define maps $\theta_B: B\to E_{\alpha}$ and $\theta_{B^*}: B^*\to E_{\alpha}$
by $ev_{[B_{\alpha},B_{\alpha}]}(\theta_B \otimes B_{\alpha})=\mu$ and
$ev_{[B_{\alpha},B_{\alpha}]}(\theta_{B^*} \otimes B_{\alpha})=\nu$, respectively. Since $\mu$ and $\nu$ are actions,
the morphisms $\theta_B$ and $\theta_{B^*}$ are algebra morphisms. Now $\beta: D(B)\to E_{\alpha}$
defined via $\beta=\nabla_{E_{\alpha}}(\theta_{B^*}\ot\theta_B)$ is an algebra morphism, because
the actions $\mu$ and $\nu$ make $B_{\alpha}$ a $D(B)$-module:
$$
\gbeg{8}{7}
\got{1}{(B^{op})^*} \got{3}{B} \got{1}{(B^{op})^*} \gvac{1} \got{1}{B} \got{1}{B_{\alpha}} \gnl
\gbmp{\hspace{0,12cm}\theta_{\hspace{-0,07cm} B^{\q*}}} \gvac{1} \gbmp{\theta_{\q B}} \gvac{1}
 \gbmp{\hspace{0,12cm}\theta_{\hspace{-0,07cm} B^{\q*}}} \gvac{1} \gbmp{\theta_{\q B}} \gcl{3} \gnl
\gwmu{3} \gvac{1} \gwmu{3} \gnl
\gvac{1} \gwmu{5} \gnl
\gvac{3} \gcn{1}{1}{1}{5} \gvac{1} \gcn{1}{1}{5}{1} \gnl
\gvac{5} \gmp{\ev} \gnl
\gvac{5} \gob{1}{B_{\alpha}}
\gend\stackrel{\nabla_{[B_{\alpha},B_{\alpha}]}}{\stackrel{\theta_{B^*}, \theta_B}{=}}
\gbeg{7}{11}
\got{1}{B^*} \got{1}{B} \got{1}{B^*} \got{1}{B} \got{3}{B_{\alpha}} \gnl
\gcl{5} \gcl{3} \gcl{1} \gcn{1}{1}{1}{3} \gvac{1} \gcl{1} \gnl
\gvac{2} \gcl{1} \gvac{1} \glmpt \gnot{\hspace{-0,44cm}\mu} \grmptb \gnl
\gvac{1} \gcn{2}{1}{3}{7} \gvac{2} \gcl{1} \gnl
\gcn{2}{1}{3}{5} \gvac{2} \glmpt \gnot{\hspace{-0,44cm}\nu} \grmptb \gnl
\gvac{2} \gcn{2}{1}{1}{5} \gvac{1} \gcl{1} \gnl
\gcn{2}{1}{1}{5} \gvac{2} \glmpt \gnot{\hspace{-0,44cm}\mu} \grmptb  \gnl
\gvac{2} \gcn{2}{1}{1}{5} \gvac{1} \gcl{1} \gnl
\gvac{4} \glmpt \gnot{\hspace{-0,44cm}\nu} \grmptb  \gnl
\gvac{5} \gcl{1} \gnl
\gvac{4} \gob{3}{B_{\alpha}}
\gend\stackrel{\equref{B tie H -mod}}{=}
\gbeg{7}{12}
\got{1}{B^*} \got{2}{B} \got{2}{B^*} \got{1}{B} \got{3}{B_{\alpha}} \gnl
\gcl{1} \gcmu \gcmu \gcn{1}{1}{1}{3} \gvac{1} \gcl{1} \gnl
\gcl{1} \gcl{1} \gbr \gcl{1} \gvac{1} \glmpt \gnot{\hspace{-0,44cm}\mu} \grmptb \gnl
\gcl{3} \glm \grm \gvac{2} \gcl{1} \gnl
\gvac{2} \gcl{1} \gcn{2}{1}{1}{5} \gcn{1}{1}{5}{3} \gnl
\gvac{2} \gcn{2}{2}{1}{5} \gvac{1} \glmpt \gnot{\hspace{-0,44cm}\nu} \grmptb \gnl
\gcn{2}{3}{1}{7} \gvac{3} \gcn{1}{1}{3}{1} \gnl
\gvac{4} \glmpt \gnot{\hspace{-0,44cm}\mu} \grmptb \gnl
\gvac{4} \gcn{1}{1}{3}{1} \gnl
\gvac{3} \glmpt \gnot{\hspace{-0,44cm}\nu} \grmptb \gnl
\gvac{4} \gcl{1} \gnl
\gvac{3} \gob{3}{B_{\alpha}}
\gend
$$

$$\stackrel{\theta_{B^*}, \theta_B}{\stackrel{\nabla_{[B_{\alpha},B_{\alpha}]}}{=}}
\gbeg{7}{10}
\got{1}{B^*} \got{2}{B} \got{2}{B^*} \got{1}{B} \got{2}{B_{\alpha}} \gnl
\gcl{1} \gcmu \gcmu \gcl{3} \gcn{1}{6}{2}{2} \gnl
\gcl{1} \gcl{1} \gbr \gcl{1} \gnl
\gcl{1} \glm \grm \gnl
\gbmp{\hspace{0,12cm}\theta_{\hspace{-0,07cm} B^{\q*}}} \gvac{1}
   \gbmp{\hspace{0,12cm}\theta_{\hspace{-0,07cm} B^{\q*}}} \gbmp{\theta_{\q B}} \gvac{1} \gbmp{\theta_{\q B}} \gnl
\gwmu{3} \gwmu{3} \gnl
\gvac{1} \gwmu{4} \gnl
\gvac{3} \gcn{1}{1}{0}{4} \gvac{1} \gcn{1}{1}{4}{0} \gnl
\gvac{5} \hspace{-0,34cm} \gmp{\ev} \gnl
\gvac{5} \gob{1}{B_{\alpha}}
\gend\stackrel{\theta_{B^*}, \theta_B}{\stackrel{alg.m.}{=}}
\gbeg{7}{10}
\got{1}{B^*} \got{2}{B} \got{2}{B^*} \got{1}{B} \got{2}{B_{\alpha}} \gnl
\gcl{1} \gcmu \gcmu \gcl{3} \gcn{1}{6}{2}{2} \gnl
\gcl{1} \gcl{1} \gbr \gcl{1} \gnl
\gcl{1} \glm \grm \gnl
\gwmu{3} \gwmu{3} \gnl
\gvac{1} \gbmp{\hspace{0,12cm}\theta_{\hspace{-0,07cm} B^{\q*}}} \gvac{2} \gbmp{\theta_{\q B}} \gnl
\gvac{1} \gwmu{4} \gnl
\gvac{3} \gcn{1}{1}{0}{4} \gvac{1} \gcn{1}{1}{4}{0} \gnl
\gvac{5} \hspace{-0,34cm} \gmp{\ev} \gnl
\gvac{5} \gob{1}{B_{\alpha}}
\gend=
\gbeg{5}{6}
\got{1}{D(B)} \got{3}{D(B)} \got{1}{B_{\alpha}} \gnl
\gwmu{3} \gvac{1} \gcl{2} \gnl
\gvac{1} \gbmp{\beta} \gnl
\gvac{1} \gcn{1}{1}{1}{4} \gvac{1} \gcn{1}{1}{3}{0} \gnl
\gvac{3} \hspace{-0,34cm} \gmp{\ev} \gnl
\gvac{3} \gob{1}{B_{\alpha}.}
\gend
$$
Then there is a strongly inner action of $D(B)$ on $E_{\alpha}$, \leref{strongly inner}, 1). Finally, this
action coincides with the original one, induced by
the original $B$- and $B^*$-action (i.e. left $B$-coaction) on $B_{\alpha}$. Indeed, we defined $\theta_B$
(resp. $\theta_{B^*}$) so that the $B$-action $\mu$ (resp. $B^*$-action $\nu$) on $B_{\alpha}$ defers
from the original one by a group-like $\lambda$ (resp. $g^*$), see \equref{sigma-tau related},
and apply \leref{mods and grouplks}.
\qed\end{proof}

\begin{rem} \rmlabel{relation BQs}
In the above theorem, if both $H$ and $B$ have bijective antipodes, the isomorphism of
braided monoidal categories \equref{iso double-YD} yields the group isomorphism
$\BQ(\C; B) \iso \BQ(k; B\rtimes H)$. Though, for the automorphism group a similar claim is not
true in general. The relation between $\Aut(\C; B)$ and $\Aut(B\rtimes H)$ will be illustrated
by an example in \ssref{ex-base level}.
\end{rem}

For (co)commutative Hopf algebras $B$ in $\C$ a part of the symmetricity condition in \thref{seqBQ}
is fulfilled, so these may serve as candidates for the application of our theorem:

\begin{prop} \cite{Sch}  
If $B$ is a commutative or cocommutative Hopf algebra in any braided monoidal category $\E$ with braiding
$\Phi$, then $\Phi_{B,B}=\Phi^{-1}_{B,B}$.
\end{prop}



For $B$ which is both commutative and cocommutative, \thref{seqBQ} generalizes Deegan-Caenepeel's
embedding result to the case of a braided
Hopf algebra, \cite[Theorem 4.1]{Cae4}.
Before seeing this we recall that the left and the right center of an algebra are a priori two
different objects, but in a symmetric monoidal category they coincide, \cite[Proposition 3.7]{Femic2}.
Though, a group-like $\lambda\in\Coalg(I, B)$ on a Hopf algebra $B$ in any braided monoidal category
$\E$ factors through the left center $Z_l(B)$ of $B$ if and only if it factors through the right
center $Z_r(B)$ of $B$. This is so, because the domain of $\lambda$ is $I$, so by naturality we
have: $\Phi_{B,B}(\lambda\ot B)=B\ot\lambda=\Phi_{B,B}^{-1}(\lambda\ot B)$.

\begin{cor} \colabel{main cor BQ}
The assumptions are like in \thref{seqBQ}. Moreover, suppose that $G(B)=\Coalg(I,Z_{\bullet}(B))$
and $G(B^*)=\Coalg(I,Z_{\bullet}(B^*))$, i.e. that the group-likes on $B$ and $B^*$ factor through
the coalgebra morphisms $I\to Z_{\bullet}(B)$ and $I\to Z_{\bullet}(B^*)$ to the left/right center
of $B$ and $B^*$, respectively. Then the map $\Pi: \Aut(\C; B) \to \BQ(\C; B)$ is injective.
Consequently, for $B$ commutative and cocommutative $\Aut(\C; B)$ can be embedded into the
Brauer-Long group $\BD(\C; B)$.
\end{cor}

\begin{proof}
Let $g\in\Coalg(I,B)$ and $\lambda^*\in\Coalg(I,B^*)$ be arbitrary. The condition that
$g$ and $\lambda^*$ factor through the
right center of $B$ and left center of $B^*$, respectively, mean that:
$$
\gbeg{2}{4}
\got{3}{B} \gnl
\gbmp{g} \gcl{1} \gnl
\gmu \gnl
\gob{2}{B}
\gend=
\gbeg{2}{5}
\got{3}{B} \gnl
\gbmp{g} \gcl{1} \gnl
\gbr \gnl
\gmu \gnl
\gob{2}{B}
\gend\qquad\textnormal{and}\qquad
\gbeg{2}{4}
\got{3}{B^*} \gnl
\gbmp{\lambda^*} \gcl{1} \gnl
\gmu \gnl
\gob{2}{B}
\gend=
\gbeg{2}{5}
\got{3}{B^*} \gnl
\gbmp{\lambda^*} \gcl{1} \gnl
\gibr \gnl
\gmu \gnl
\gob{2}{B}
\gend\quad\Leftrightarrow\quad
\gbeg{2}{4}
\got{2}{B} \gnl
\gcmu \gnl
\gbmp{\lambda} \gcl{1} \gnl
\gob{3}{B}
\gend=
\gbeg{2}{5}
\got{2}{B} \gnl
\gcmu \gnl
\gibr \gnl
\gbmp{\lambda} \gcl{1} \gnl
\gob{3}{B.}
\gend
$$
This yields:
$$\gbeg{4}{5}
\got{2}{B} \gnl
\gcmu \gnl
\gbmp{\lambda} \gcl{1} \gbmp{g} \gnl
\gvac{1} \gmu \gnl
\gob{4}{B}
\gend=
\gbeg{4}{7}
\got{2}{B} \gnl
\gcmu \gnl
\gibr \gnl
\gbmp{\lambda} \gcl{1} \gbmp{g} \gnl
\gvac{1} \gibr \gnl
\gvac{1} \gmu \gnl
\gob{4}{B}
\gend=
\gbeg{3}{5}
\got{4}{B} \gnl
\gvac{1} \gcmu \gnl
\gbmp{g} \gcl{1} \gbmp{\lambda} \gnl
\gmu \gnl
\gob{2}{B}
\gend
$$
which by \leref{equiv cond gr} means that $g\bowtie\lambda\in G(D(B)^*)$. Then due to \prref{gr x gr}
it follows: $G(D(B))=G(D(B^*))$. From the exactness of the sequence in \thref{seqBQ} we conclude that
$\Theta$ is a trivial map, hence $\Pi$ is injective. The final statement follows from the fact that for
a commutative and cocommutative Hopf algebra $B$ in $\C$ the category ${}_B^B\YD(\C)$ coincides with that
of $B$-dimodules in $\C$.
\qed\end{proof}

\section{Examples}

In \cite{VZ4} there was constructed an exact sequence of the form of \equref{seqBQ} where instead of
$\C$ the category $_R\M$ of modules over a commutative ring $R$ was considered.
Our sequence enables one to regard braided Hopf algebras, in particular those arising from a
Radford biproduct $B\rtimes H$, where $H$ is a Hopf algebra over a field and $B$ is a Hopf algebra in
the category of Yetter-Drinfel'd modules over $H$. In this section we treat some examples of this sort.


\subsection{A family of examples} \sslabel{H(nd)} 

In \cite{CF1} we studied the Brauer group $\BM(\E; B)$ of the braided monoidal category ${}_B\E$
of left modules over a Hopf algebra $B$ in $\E$, where $\E$ is a braided monoidal category.
We computed this group for the family of Hopf algebras studied in
\cite[Section 4]{N} and defined as follows. Let $n,m$ be
natural numbers, $k$ a field such that $char(k) \nmid 2m$ and $\omega$ a $2m$-th primitive root of unity.
For $i=1,...,n$ choose $1 \leq d_i < 2m$ odd numbers and set $d^{\leq n}=(d_1,...,d_n)$. Then
$$H(m,n,d^{\leq n})=k\langle g,x_1,...,x_n\vert g^{2m}=1, x_i^2=0, gx_i=\omega^{d_i} x_ig, x_ix_j=-x_jx_i \rangle$$
is a Hopf algebra, where $g$ is group-like and $x_i$ is a $(g^m, 1)$-primitive element, that is,
$\Delta(x_i)=1\ot x_i+x_i\ot g^m$ and $\Epsilon(x_i)=0$. The antipode is given by $S(g)=g^{-1}$ and
$S(x_i)=-x_ig^{m}$. We proved that $H(m,n,d^{\leq n})$ decomposes as the Radford biproduct:
\begin{equation} \eqlabel{ex-biproduct}
H(m,n,d^{\leq n})\iso B\rtimes H(m,n-1,d^{\leq n-1})
\end{equation}
where the braided Hopf algebra is the exterior algebra $B=K[x_n]/(x_n^2)$. The isomorphism is given by
$G \mapsto 1\times g, X_i \mapsto 1 \times x_i, X_n \mapsto x_n \times g^m$. We have that $B$
is a module over $H=H(m,n-1,d^{\leq n-1})$ by the action $g\cdot x_n=\omega^{d_n}x_n$ and $x_i \cdot x_n=0$
for $i=1,...,n-1$. It becomes a commutative and cocommutative Hopf algebra in $_H\M$ with $x_n$ being a
primitive element, i.e., $\Delta_B(x_n)=1\ot x_n+x_n\ot 1,\ \Epsilon_B(x_n)=0$ and $S_B(x_n)=-x_n.$
The Hopf algebra $H(m,n,d^{\leq n})$ is quasitriangular with the family of quasitriangular structures
\cite[(6.4) on p. 69]{CF1}:
\begin{eqnarray} \label{qtr-Hnu}
\R_s^n=\frac{1}{2m}\Big(\sum_{j,t=0}^{2m-1}\omega^{-jt}g^j \otimes g^{st}\Big)
\end{eqnarray}
where $0\leq s< 2m$ is such that $sd_i \equiv m \ (mod.\ 2m)$ for every $i=1,...,n$. Moreover, $\R_s^n$
is triangular if and only if $s=m$.
As it is well known (\cite{Maj}), every left $H$-module $M$ belongs to ${\ }^H_H\YD$ with the coaction
\begin{eqnarray} \eqlabel{leftcomMaj}
\lambda(m)=\R^{(2)}\ot \R^{(1)} m, \quad m\in M
\end{eqnarray}
- we denote here $\R=\R_s^{n-1}$ for brevity -
and $({}_H\M, \Phi_{\R})$ can be seen as a braided monoidal subcategory of $({\ }^H_H\YD, \Phi)$.
Here $\Phi$ is given by \equref{braidingD}, whereas $\Phi_{\R}$ and its inverse are given by:
\begin{equation} \eqlabel{braid-Hnd}
\Phi_{\R}(m\ot n)=\R^{(1)} n\ot S^{-1}(\R^{(2)}) m; \qquad \Phi_{\R}^{-1}(m\ot n)=\R^{(2)} n\ot\R^{(1)} m.
\end{equation}
Thus $B$ becomes a Hopf algebra in $({\ }^H_H\YD, \Phi)$. The above braided embedding of categories
induces a group embedding $\BM(k; H) \hookrightarrow\BQ(k; H)$.

\begin{rem}
Note that in in this paper we used $\Phi$ as in \equref{braidingD} and \equref{br in YD}, which usually
is the inverse of the braiding. As Schauenburg observed in \cite{Sch}, if $(B, \nabla, \Delta)$ is a Hopf
algebra in $(\E, \Phi)$, then $(B, \nabla\Phi^{-1}, \Delta)$, and similarly $(B, \nabla, \Phi^{-1}\Delta)$,
is a Hopf algebra in $(\E, \Phi^{-1})$. So if $B$ is commutative or cocommutative in $\E$ (with whatever
sign of the braiding), it is irrelevant which of the two braidings one takes.
\end{rem}

We have that \coref{main cor BQ} applies to the described family of Hopf algebras. 
First of all let us prove that $\Phi_{B,M}$ is symmetric for any $M\in\C$.
Take $m\in M$ and let us check if $\Phi(b\ot m)=\Phi^{-1}(b\ot m)$ (see \equref{braid-Hnd} and
(\ref{qtr-Hnu})). For $b=1$ the computation is easier, we do here the case $b=x_n$. We find:
$$\begin{array}{rl}
\Phi_{\R}(x_n\ot m)\hskip-1em&=\frac{1}{2m}\Big(\sum_{j,t=0}^{2m-1}\omega^{-jt}g^j \cdot m \ot
 S^{-1}(g^{st})\cdot x_n\Big) \\
&=\frac{1}{2m}\Big(\sum_{j,t=0}^{2m-1}\omega^{-jt}g^j \cdot m \ot
 g^{-st}\cdot x_n\Big) \\
&=\frac{1}{2m}\Big(\sum_{j,t=0}^{2m-1}\omega^{-jt}g^j \cdot m \ot
 \omega^{-d_n st} x_n\Big) \\
&=\frac{1}{2m}\Big(\sum_{j=0}^{2m-1}[\sum_{t=0}^{2m-1}(\omega^{-(j+s d_n)})^t] g^j \cdot m\Big) \ot x_n \\
&=g^{-sd_n}\cdot m\ot x_n
\end{array}$$
(the sum in the bracket in the penultimate expression is different from $0$ only for $j=-sd_n\ (mod.\ 2m)$,
when it equals $2m$). Similarly, it is:
$$\begin{array}{rl}
\Phi_{\R}^{-1}(x_n\ot m)\hskip-1em&=\frac{1}{2m}\Big(\sum_{j,t=0}^{2m-1}\omega^{-jt}g^{st} \cdot m \ot g^j\cdot x_n\Big) \\
&=\frac{1}{2m}\Big(\sum_{j,t=0}^{2m-1}\omega^{-jt}g^{st} \cdot m \ot \omega^{d_n j} x_n\Big) \\
&=\frac{1}{2m}\Big(\sum_{t=0}^{2m-1}[\sum_{j=0}^{2m-1}(\omega^{d_n-t})^j] g^{st} \cdot m\Big) \ot x_n \\
&=g^{sd_n}\cdot m\ot x_n.
\end{array}$$
Recall that
$sd_i \equiv m \ (mod.\ 2m)$ for every $i=1,...,n$. Hence $g^{sd_n}=-1$ and the two expressions we
computed above are equal.

The group of group-likes of $D(B)$ is isomorphic to $G(B^*)\times G(B)$. Obviously, $G(B)$ is trivial
and it is easily seen that so is $G(B^*)=\Alg(B, k)$. In this case the sequence \equref{seqBQ} collapses
to the group embedding of the Hopf automorphism group of $B$ in $\C$ into the Brauer-Long group of
$B$-dimodule algebras in $\C$, $\Aut(\C; B) \hookrightarrow \BD(\C; B).$
Any bialgebra automorphism $f$ of $B$ in $\C$ is given by $f(1)=1$ and $f(x_n)=\xi x_n$ for some $\xi\in k^*$.
Then we obtain a group embedding:
\begin{equation} \eqlabel{simple aut embeds}
k^*\iso\Aut(\C; B) \hookrightarrow \BD(\C; B).
\end{equation}

\par\medskip

Moreover, since both $H$ and $B$ are finite dimensional and have bijective antipodes, we have
braided monoidal isomorphism of categories ${}_{D(B\rtimes H)}\M\iso {}_{B\rtimes H} ^{B\rtimes H}\YD$
and \rmref{relation BQs} applies, so that we get:
\begin{equation} \eqlabel{Br groups iso}
\BD(\C; B) \iso \BQ(\C; B) \iso \BQ(k; B\rtimes H) \iso \BM(k; D(B\rtimes H)).
\end{equation}
In \cite[Example 7]{VZ4} it was shown that $k^*/{\Zz_2}$ embeds into the quantum Brauer group of the
four-dimensional Sweedler's Hopf algebra $H_4$. By our result $k^*$ embeds into the quantum Brauer group
of all the family of Radford biproducts presented above.


\subsection{A larger subgroup of the quantum Brauer group} \sslabel{ex-base level}

In \coref{inner action} we proved that there is an action of $\Aut(\C; B)$ on $\BQ(\C; B)$  given by
$\mu: (A, \alpha)\mapsto [A(\alpha)]$ via \equref{A(alfa)}, if $\Phi_{B,B}$ and $\Phi_{B,A}$ are symmetric
for all Azumaya algebras $A$ in ${}_B^B\YD(\C)$.
Suppose we are given a subgroup $\iota: Q\hookrightarrow \BQ(\C; B)$
invariant under $\mu$. Then the subgroup generated by $Q$ and $\Pi(\Aut(\C; B))$ is $(\iota\ot\Pi)(Q \rtimes
\Aut(\C; B))$. Here $\rtimes$ is the semi-direct product of groups. 
We want to study for which cases of $B\rtimes H$ from \ssref{H(nd)} the Brauer group $\BM(\C; B)$
is invariant under $\mu$. Observe that the above symmetricity condition is satisfied.
Since $B$ is cocommutative,
it is quasitriangular with the quasitriangular structure $\R_B=1_B\ot 1_B$ (for the categorical
definition of the quasitriangular structure see \cite[Section 2.5]{Maj4}, the result is obvious).
Consequently, every $B$-module in $\C$ is a trivial $B$-comodule and a Yetter-Drinfel'd module over
$B$ in $\C$ (for this categorical version of the classically known fact see \cite[Lemma 5.3.1]{Besp},
the formula for the comodule structure in $\C$ is a categorification of \equref{leftcomMaj}).
Hence we may consider the subgroup $\BM(\C; B)$ of $\BQ(\C; B)$. To examine the invariance of
$\BM(\C; B)$ under the action $\mu$ we study when the following diagram commutes:
$$
\bfig
\putmorphism(0,480)(1,0)[\BM(\C; B) \times k^*`\phantom{\BM(\C; B)}`\mu_r]{1000}1a
\putmorphism(1000,480)(0,-1)[\BM(\C; B)`\BQ(\C; B)`\iota]{400}1r
\putmorphism(0,480)(0,-1)[\phantom{M'\Box_H(M\Box_H M') }`\BQ(\C; B) \times k^*`
\iota\times id]{400}1l
\putmorphism(0,80)(1,0)[\phantom{\BM(\C; B)\times k^*}`\phantom{\BM(\C; B)}`\mu]{1000}1b
\efig
$$
For any $\alpha\in k^*$ and $[A]\in\BM(\C; B)$ it is easily seen that $\iota([A(\alpha)])$
and $(\iota([A]))(\alpha)$ have the same $B$-actions. Their $B$-coactions are the same if and only if
$\alpha^{-1}(\R_B^{(2)})\ot \R_B^{(1)}a=\R_B^{(2)}\ot\alpha(\R_B^{(1)})a$ for all $a\in A$. By the form
of $\R_B$ this is clearly fulfilled. Hence we have:

\begin{prop} \prlabel{stable subgr}
Let $B\rtimes H$ be as in \ssref{H(nd)}. There is a subgroup
$$\BM(\C; B) \rtimes k^* \hookrightarrow \BD(\C; B)\iso\BQ(k; B\rtimes H).$$
\end{prop}

\par\bigskip

The braided embedding of categories $_H\M\hookrightarrow \C$ induces a braided embedding $_{B\rtimes H}\M\iso
_B(_H\M)\hookrightarrow _B\C$. Then the Brauer group $\BM(_H\M; B)\iso\BM(k; B\rtimes H)$ that we
studied in \cite{CF1} is a subgroup of $\BM(\C; B)$. This implies: $\BM(k; B\rtimes H)\rtimes k^*
\hookrightarrow \BQ(k; B\rtimes H)$.
Applying \cite[Theorem 6.11]{CF1} we obtain:

\begin{cor} \colabel{known BM in BQ}
Let $B\rtimes H$ be as in \ssref{H(nd)}.
For each $j=1,...,n$ set $I_j=\{1 \leq i < j: d_i \equiv -d_j \ (mod. \ 2m)\}$. Let $r_j=\vert I_j \vert$ if $d_j \neq m$ and $r_j=\vert I_j \vert+1$ otherwise. Then there is a subgroup:
$$(\BM(k;k\Zz_{2m}) \times (k,+)^{r_1+...+r_n})\rtimes k^*\hookrightarrow \BQ(k; B\rtimes H).$$
\end{cor}


\par\bigskip

We may apply \coref{inner action} also on another level. Namely, for $\C=Vec$ - take $H$ to be trivial.
Let us call this {\em base level}. The symmetricity conditions are trivially satisfied.
Then we recover \cite[Theorem 4.11]{CVZ2} for the field case.
We now want to study for which cases of $B\rtimes H$ from \ssref{H(nd)} the subgroup $\BM(k; B\rtimes H)$
of $\BQ(k; B\rtimes H)$ is invariant under $\mu$. For this purpose we analyse when the following
diagram commutes:
$$
\bfig
\putmorphism(0,500)(1,0)[\BM(k; B\rtimes H) \times \Aut(B\rtimes H)`\phantom{\BM(k; B\rtimes H)}`\mu_r]{1500}1a
\putmorphism(1500,500)(0,-1)[\BM(k; B\rtimes H)`\BQ(k; B\rtimes H)`\iota]{500}1r
\putmorphism(0,500)(0,-1)[\phantom{M'\Box_H(M\Box_H M') }`\BQ(k; B\rtimes H) \times \Aut(B\rtimes H)`
\iota\times id]{500}1l
\putmorphism(0,0)(1,0)[\phantom{\BM(k; B\rtimes H)\times \Aut(B\rtimes H)}`\phantom{\BM(k; B\rtimes H)}`\mu]{1500}1b
\efig
$$
Take $\alpha\in\Aut(B\rtimes H)$ and $[A]\in\BM(k; B\rtimes H)$. As above, $\iota([A(\alpha)])$
and $(\iota([A]))(\alpha)$ have the same $B\rtimes H$-actions. We find that their $B\rtimes H$-coactions are the
same if and only if $\alpha^{-1}(\R_L^{(2)})\ot \R_L^{(1)}a=\R_L^{(2)}\ot\alpha(\R_L^{(1)})a$ for all $a\in A$,
where $L=B\rtimes H$.

In \equref{ex-biproduct} the quasitriangular structure $\R_s^{n-1}$ extends from $H$ to $H(m,n,d^{\leq n})$.
The extension is given by $
\R_L=(\kappa\ot \kappa) \R_s^{n-1}$, which lies in $(B\rtimes H)\ot (B\rtimes H)$, here $\kappa:
H \to H(m,n,d^{\leq n})$
is the Hopf algebra embedding. Then for $\alpha\in\Aut(B\rtimes H)$ the upper rectangular commutes if and only if
$\alpha^{-1}(1_B\ot\R_H^{(2)})\ot (1_B\ot\R_H^{(1)})a=(1_B\ot\R_H^{(2)})\ot\alpha(1_B\ot\R_H^{(1)})a$ for all
$a\in A$. By (\ref{qtr-Hnu}) this is fulfilled if and only if:
\begin{equation} \eqlabel{cond BQ subgr}
\sum_{j,t=0}^{2m-1} \omega^{-jt} \alpha^{-1}(1_B\ot g^{st})\ot (1_B\ot g^j)a=
\sum_{j,t=0}^{2m-1} \omega^{-jt} (1_B\ot g^{st})\ot\alpha(1_B\ot g^j)a.
\end{equation}

\begin{lma} \lelabel{subgroup of BQ}
Let $H(m,n,d^{\leq n})\iso B\rtimes H(m,n-1,d^{\leq n-1})$ be the Hopf algebra described in
\equref{ex-biproduct} and set $H=H(m,n-1,d^{\leq n-1})$. Then
$$\BM(k; B\rtimes H)\rtimes \Pi(\Aut(B\rtimes H))$$
is a subgroup of $\BQ(k; B\rtimes H)$ if \equref{cond BQ subgr} is fulfilled for every
$\alpha\in\Aut(B\rtimes H)$ and $a\in A$, where $A$ is any $B\rtimes H$-module Azumaya algebra.
\end{lma}



Consider the Hopf algebras $H(m,n,m^{\leq n})$ with $d_1=...= d_n=m$ 
and denote them by $H(m,n)$. From \cite[Proposition 11]{Rad3} we know that
$\Aut(H(m,n))\iso GL_n(k)$ if $m>1$ or $char (k) \not= 2$. Every $\alpha\in\Aut(H(m,n))$
is identified with $\alpha_{\vert_{k\langle x_1,...,x_n\rangle}}$ in $GL_n(k)$ and it is $\alpha(g)=g$.
Then \equref{cond BQ subgr} is fulfilled. On the other hand, for a consequence of \cite[Theorem 6.11]{CF1}
we have:
$$\BM(k;H(m,n)) \cong \BM(k;k\Zz_{2m}) \times (k,+)^{\frac{n(n+1)}{2}}.$$
Moreover, in \cite[Example 8]{VZ4}, the base level of the sequence \equref{seqBQ} is presented for
$H(m,n)$. There it is computed: $G(D(H(m,n)))\iso U_{2m}\times U_{2m}$ and $G(D(H(m,n))^*)\iso U_{2m}$,
where $U_{2m}$ denotes the group of $2m$-th roots of unity. Note that $\Pi(\Aut(L))$ is the quotient group
of $\Aut(L)$ modulo ${G(D(L))/G(D(L)^*)}$ for $L=H(m,n)$. Then by \leref{subgroup of BQ} we have:

\begin{prop} \prlabel{stable subgr BM Hnm}
There is a subgroup
$$(\BM(k;k\Zz_{2m}) \times (k,+)^{\frac{n(n+1)}{2}})\rtimes GL_n(k)/U_{2m} \hookrightarrow
\BQ(k; H(m,n)).$$
\end{prop}

This improves the result of \coref{known BM in BQ} for the subfamily $H(m,n)$ of
$H(m,n,d^{\leq n})\iso B\rtimes H(m,n-1,d^{\leq n-1})$. As particular instances of $H(m,n)$ one obtains
Sweedler's Hopf algebra $H_4=H(1,1)$, Radford Hopf algebra $H_m=H(m,1)$ 
and Nichols Hopf algebra $E(n)=H(1,n)$, being:
$$H_4=k\langle g,x\vert g^2=1, x^2=0, gx=-xg\rangle,\
H_m=k\langle g,x\vert g^{2m}=1, x^2=0, gx=-xg\rangle$$
\begin{equation} \eqlabel{Nichols Ha}
E(n)=k\langle g,x_1,...,x_n \vert g^2=1, x_i^2=0, gx_i=-x_ig, x_ix_j=-x_jx_i \rangle 
\end{equation}
with $g$ group-likes and $x_i$'s $(g^m, 1)$-primitive elements for the respective values of $m$.
By the above proposition we then have the subgroups:
$$(\BW(k) \times (k,+))\rtimes k^*/{\Zz_2} \hookrightarrow\BQ(k; H_4) \vspace{0,1cm}$$
$$(\BM(k;k\Zz_{2m}) \times (k, +)) \rtimes k^*/{\Zz_{2m}} \hookrightarrow\BQ(k; H_m) \vspace{0,3cm}$$
$$(\BW(k) \times (k,+)^{n(n+1)/2}) \rtimes GL_n(k)/\Zz_2 \hookrightarrow\BQ(k; E(n)).$$
The first identity appears in \cite[Example 5.11]{VZ3}, whereas the third one gives a larger subgroup
of $\BQ(k; E(n))$ than the one from \cite[Theorem 4.6]{CC2},
where the subgroup $(k,+)^{n(n+1)/2} \rtimes GL_n(k)/\Zz_2$ was identified.
\par\bigskip

Let us summarize that the automorphism groups on the base level and on the braided level may be
very different. For the family of Hopf algebras $H(m, n)$, which we may consider as:
$$H(m, n)\iso k[x_n]/(x_n^2) \rtimes (k[x_{n-1}]/(x_{n-1}^2)\rtimes (... \rtimes (k[x_1]/(x_1^2)\rtimes k\Zz_2)...)),$$
as we presented above, it is: $\Aut(H(m,n))\iso GL_n(k)\iso\Aut(Span\{x_1,..., x_n\})$, whereas in \ssref{H(nd)}
we showed that $\Aut(\C; B)\iso k^*\iso\Aut(Span\{x_1\})$, with $B=k[x_n]/(x_n^2)$. Accordingly,
in the base level we have $G(D(H(m,n)))\iso U_{2m}\times U_{2m}$, while the group of group-likes in the
braided case $G(D(B))$ is trivial.

\subsection{A counter-example} 

The Taft algebra is a Radford biproduct $B\rtimes H$ where 
the braiding in $\C={}_H ^H\YD$ does not obey the symmetricity condition in \thref{seqBQ} even
for $\Phi_{B,B}$. We recall that the Taft algebra is
$$T_n=k\langle g,x\vert g^n=1, x^n=0, gx=\omega xg\rangle$$
where $n\geqslant 2$ is a natural number such that $char(k)$ is coprime to $n$, and $\omega$ is an
$n$-th primitive root of unity. The structure of a Hopf algebra on $T_n$ is such that $g$ is group-like,
$x$ is $(g,1)$-primitive, 
with $S(x)=-xg^{-1}$. When $n=2$ note that we recover Sweedler's Hopf algebra $H_4$. Taft algebra is
isomorphic to a Radford biproduct
$$T_n\iso k[x]/(x^n)\times k\Zz_{n}$$
(send $G\mapsto 1\ot g$ and $X\mapsto x\ot g$), where $g\cdot x=\omega x$.
Here $H=k\Zz_{n}$ has the quasitriangular structure as in (\ref{qtr-Hnu}) where $2m$ is
replaced by $n$ and $1\leqslant s<n$. However, this quasitriangular structure does not extend to $T_n$
if $n>2$. 

Even though we may not apply \thref{seqBQ}, we will compute the sequence \equref{seqBQ} 
on the base level for $T_n$. In \cite[Section 2]{Sch3} the groups of automorphisms of $T_n$ and group-likes
of $T_n$ and its dual are computed:
$$G(T_n)=\langle g\rangle, \quad G(T_n^*)=\{ \lambda_i \hspace{0,2cm} \vert \hspace{0,2cm}
i\in \{1,..., n\} \}, \quad \Aut(T_n)\iso k^*,$$
where $\lambda_i(g)=\omega^i, \lambda_i(x)=0$, and any $\Fi\in \Aut(T_n)$ is such that $\Fi(g)=g$ and
$\Fi(x)=\alpha x$, for some $\alpha\in k^*$. Then $G(D(T_n))\iso G(T_n^*)\times G(T_n)\iso U_n\times U_n$.
One may easily see that for any $i\in \{1,..., n\}$ and $\lambda_i$ it is:
$$\crta{g^i}(h) = \crta{\lambda_i}(h)=
\left\{
\begin{array}{c l}
h, & h=g \\
\omega^{-i}x, & h=x.
\end{array}\right.$$
It follows that $G(D(T_n)^*)\iso S(T_n)=\{ (g^i, \lambda_i) \hspace{0,2cm} \vert \hspace{0,2cm}
i\in \{1,..., n\} \} \iso U_n$. Putting all ingredients together we obtain the exact sequence:
$$
\bfig
\putmorphism(-40, 0)(1, 0)[1`U_n`]{340}1a
\putmorphism(300, 0)(1, 0)[\phantom{U_n}`U_n\times U_n` ]{480}1a
\putmorphism(780, 0)(1, 0)[\phantom{U_n\times U_n}`k^*`]{480}1a
\putmorphism(1260, 0)(1, 0)[\phantom{k^*}`\BQ(k; T_n).`]{560}1a
\efig
$$

\end{document}